\documentclass[11pt]{amsart}
\usepackage{amsfonts}

\newtheorem{theorem}{Theorem}
\theoremstyle{plain}

\newtheorem{claim}{Claim}

\newtheorem{corollary}{Corollary}

\newtheorem{definition}{Definition}

\newtheorem{lemma}{Lemma}

\newtheorem{proposition}{Proposition}
\newtheorem{remark}{Remark}

\numberwithin{equation}{section}
\input{tcilatex}

\begin{document}
\title[Integro-differential problems]{On $L^{p}$-theory for parabolic and
elliptic integro-differential equations with scalable operators in the whole
space}
\author{R. Mikulevi\v{c}ius and C. Phonsom}
\address{University of Southern California, Los Angeles}
\date{April 26, 2016}

\begin{abstract}
Elliptic and parabolic integro-differential model problems are considered in
the whole space. By verifying H\"{o}rmander condition, the existence and
uniqueness is proved in $L_{p}$-spaces of functions whose regularity is
defined by a scalable, possibly nonsymmetric, Levy measure. Some rough
probability density function estimates of the associated Levy process are
used as well.
\end{abstract}

\maketitle
\tableofcontents

\section{Introduction}

Let $\sigma \in \left( 0,2\right) $ and $\mathfrak{A}^{\sigma }$ be the
class of all nonnegative measures $\pi $\ on $\mathbf{R}_{0}^{d}=\mathbf{R}%
^{d}\backslash \left\{ 0\right\} $ such that $\int \left\vert y\right\vert
^{2}\wedge 1d\pi <\infty $ and 
\begin{equation*}
\sigma =\inf \left\{ \alpha <2:\int_{\left\vert y\right\vert \leq
1}\left\vert y\right\vert ^{\alpha }d\mathfrak{\pi }<\infty \right\} .
\end{equation*}%
In addition we assume that for $\pi \in \mathfrak{A}^{\sigma },$ 
\begin{eqnarray*}
\int_{\left\vert y\right\vert >1}\left\vert y\right\vert d\pi &<&\infty 
\text{ if }\sigma \in \left( 1,2\right) , \\
\int_{R<\left\vert y\right\vert \leq R^{\prime }}yd\pi &=&0\text{ if }\sigma
=1\text{ for all }0<R<R^{\prime }<\infty .\text{ }
\end{eqnarray*}

In this paper we consider the parabolic Cauchy problem with $\lambda \geq 0$%
\begin{eqnarray}
\partial _{t}u(t,x) &=&Lu(t,x)-\lambda u\left( t,x\right) +f(t,x)\text{ in }%
E=[0,T]\times \mathbf{R}^{d},  \label{1'} \\
u(0,x) &=&0,  \notag
\end{eqnarray}%
and the elliptic problem with $\lambda >0$,%
\begin{equation}
\lambda u\left( x\right) -Lu\left( x\right) =g\left( x\right) ,x\in \mathbf{R%
}^{d},  \label{2'}
\end{equation}%
with $\lambda \geq 0$ and $\lambda >0$ and integrodifferential operator 
\begin{equation*}
L\varphi \left( x\right) =L^{\pi }\varphi \left( x\right) =\int \left[
\varphi (x+y)-\varphi \left( x\right) -\chi _{\sigma }\left( y\right) \nabla
\varphi \left( x\right) y\right] \pi \left( dy\right) ,\varphi \in
C_{0}^{\infty }\left( \mathbf{R}^{d}\right) ,
\end{equation*}%
where $\pi \in \mathfrak{A}^{\sigma },$ $\chi _{\sigma }\left( y\right) =0$
if $\sigma \in \lbrack 0,1),\chi _{\sigma }\left( y\right) =1_{\left\{
\left\vert y\right\vert \leq 1\right\} }\left( y\right) $ if $\sigma =1$ and 
$\chi _{\sigma }\left( y\right) =1$ if $\sigma \in (1,2).$ The symbol of $L$
is 
\begin{equation*}
\mathfrak{\psi }\left( \xi \right) =\psi ^{\pi }\left( \xi \right) =\int %
\left[ e^{i\xi \cdot y}-1-i\chi _{\sigma }\left( y\right) \xi \cdot y\right]
\pi \left( dy\right) ,\xi \in \mathbf{R}^{d}.
\end{equation*}%
Note that $\pi \left( dy\right) =dy/\left\vert y\right\vert ^{d+\sigma }\in 
\mathfrak{A}^{\sigma }$ and, in this case, $L=L^{\pi }=c\left( \sigma
,d\right) \left( -\Delta \right) ^{\sigma /2}$, where $\left( -\Delta
\right) ^{\sigma /2}$\ is a fractional Laplacian. Let $\pi _{0}\in \mathfrak{%
A}^{\sigma }$ and 
\begin{equation}
c_{1}\left\vert \psi ^{\pi _{0}}\left( \xi \right) \right\vert \leq |\psi
^{\pi }\left( \xi \right) |\leq c_{2}\left\vert \psi ^{\pi _{0}}\left( \xi
\right) \right\vert ,\xi \in \mathbf{R}^{d},  \label{3'}
\end{equation}%
for some $0<c_{1}\leq c_{2}$. Given $\pi _{0}\in \mathfrak{A}^{\sigma },p\in
\lbrack 1,\infty ),$ we denote $H_{p}^{\pi _{0}}\left( \mathbf{R}^{d}\right) 
$ (resp. $\mathcal{H}_{p}^{\pi _{0}}\left( E\right) $) the closure in $%
L_{p}\left( \mathbf{R}^{d}\right) $ (resp. $L_{p}\left( E\right) $) of $%
C_{0}^{\infty }\left( \mathbf{R}^{d}\right) $ (resp. $C_{0}^{\infty }\left(
E\right) $) with respect to the norm 
\begin{equation*}
\left\vert f\right\vert _{\pi _{0},p}=\left\vert f\right\vert _{L_{p}\left( 
\mathbf{R}^{d}\right) }+\left\vert L^{\pi _{0}}f\right\vert _{L_{p}\left( 
\mathbf{R}^{d}\right) }\,,\text{ resp. }\left\vert g\right\vert _{\pi
_{0},p}=\left\vert g\right\vert _{L_{p}\left( E\right) }+\left\vert L^{\pi
_{0}}g\right\vert _{L_{p}\left( E\right) }.\,
\end{equation*}%
If $f_{n}\in C_{0}^{\infty }\left( \mathbf{R}^{d}\right) ,f_{n}\rightarrow f$
and $L^{\pi }f_{n}\rightarrow g$ in $L_{p}\left( \mathbf{R}^{d}\right) $ we
denote $g=L^{\pi }f$.

In this note, under certain "scalability" assumptions (see Assumption 
\textbf{D}$\left( \kappa ,l\right) $ below), we prove the existence and
uniqueness of (\ref{1'}) and (\ref{2'}) in $H_{p}^{\pi _{0}}\left( \mathbf{R}%
^{d}\right) $ (resp. $\mathcal{H}_{p}^{\pi _{0}}\left( E\right) $). Moreover
the following estimates hold:%
\begin{equation}
\left\vert u\right\vert _{\mathcal{H}_{p}^{\pi _{0}}}\leq C\left\vert
f\right\vert _{L_{p}\left( E\right) },\left\vert u\right\vert _{H_{p}^{\pi
_{0}}}\leq C\left\vert f\right\vert _{L_{p}\left( \mathbf{R}^{d}\right) }.
\label{4}
\end{equation}%
The symbol $\psi ^{\pi }\left( \xi \right) $ is not smooth in $\xi $ and the
standard Fourier multiplier results do not apply in this case. In order to
prove (\ref{4}), we associate to $L^{\pi }$ a family of balls and verify H%
\"{o}rmander condition (see Theorem \ref{t3} and (\ref{f23}) below) for it,
and apply Calderon-Zygmund theorem. As an example, we consider $\pi \in 
\mathfrak{A}^{\sigma }$ defined in radial and angular coordinates $%
r=\left\vert y\right\vert ,w=y/r,$ as 
\begin{equation}
\pi \left( \Gamma \right) =\int_{0}^{\infty }\int_{\left\vert w\right\vert
=1}\chi _{\Gamma }\left( rw\right) a\left( r,w\right) j\left( r\right)
r^{d-1}S\left( dw\right) dr,\Gamma \in \mathcal{B}\left( \mathbf{R}%
_{0}^{d}\right) ,  \label{5}
\end{equation}%
where $S\left( dw\right) $ is a finite measure on the unit sphere on $%
\mathbf{R}^{d}$. In \cite{xz}, the parabolic equation (\ref{1'}) was
considered, with $\pi $ in the form (\ref{5}) with $a=1,j\left( r\right)
=r^{-d-\sigma },$ and such that 
\begin{eqnarray*}
&&\int_{0}^{\infty }\int_{\left\vert w\right\vert =1}\chi _{\Gamma }\left(
rw\right) r^{-1-\sigma }\rho _{0}\left( w\right) S\left( dw\right) dr \\
&\leq &\pi \left( \Gamma \right) =\int_{0}^{\infty }\int_{\left\vert
w\right\vert =1}\chi _{\Gamma }\left( rw\right) r^{-1-\sigma }a\left(
r,w\right) S\left( dw\right) dr \\
&\leq &\int_{0}^{\infty }\int_{\left\vert w\right\vert =1}\chi _{\Gamma
}\left( rw\right) r^{-1-\sigma }S\left( dw\right) dr,\Gamma \in \mathcal{B}%
\left( \mathbf{R}_{0}^{d}\right) ,
\end{eqnarray*}%
and (\ref{3'}) holds with $\psi ^{\pi _{0}}\left( \xi \right) =\left\vert
\xi \right\vert ^{\sigma },\xi \in \mathbf{R}^{d}$. In this case, $\mathcal{H%
}_{p}^{\pi _{0}}\left( E\right) =H_{p}^{\sigma }\left( E\right) $ is the
fractional Sobolev space. The solution estimate (\ref{4}) for (\ref{1'}) was
derived in \cite{xz}, using $L^{\infty }$-$BMO$ type estimate. In \cite{kk},
the elliptic problem (\ref{2'}) was studied for $\pi $ in the form (\ref{5})
with $S\left( dw\right) =dw$ being a Lebesgue measure on the unit sphere in $%
\mathbf{R}^{d}$, with $0<c_{1}\leq a\leq c_{2}$, and a set of technical
assumptions on $j\left( r\right) $. The inequality (\ref{4}) for (\ref{2'})
was obtained using sharp function estimate based on the solution H\"{o}lder
norm estimate (following the idea in \cite{KimDong}, where (\ref{2'}) was
considered in $H_{p}^{\sigma }\left( \mathbf{R}^{d}\right) $ with $\pi $ as
in (\ref{5}) with $j\left( r\right) =r^{-d-\sigma }$ and $0<c_{1}\leq a\leq
c_{2}$.

The note is organized as follows. In Section 2, the main theorem is stated,
and an example of the form (\ref{5}) considered. In Section 3, the essential
technical results are presented. The main theorem is proved in Section 4.

\section{Notation and Main Results}

Denote $E=\left[ 0,T\right] \times \mathbf{R}^{d},\mathbf{N=}\left\{
0,1,2,\ldots \right\} ,\mathbf{R}_{0}^{d}=\mathbf{R}^{d}\backslash \left\{
0\right\} $. If $x,y\in \mathbf{R}^{d}$, we write%
\begin{equation*}
x\cdot y=\sum_{i=1}^{d}x_{i}y_{i},\left\vert x\right\vert =\left( x\cdot
x\right) ^{1/2}.
\end{equation*}

For a function $u\left( t,x\right) $ on $E$, we denote its partial
derivatives by $\partial _{t}u\left( t,x\right) =\partial u/\partial t$,$%
\partial _{i}u=\partial u/\partial x_{i}$, and $D^{\gamma }u=\partial
^{\left\vert \gamma \right\vert }u/\partial x_{1}^{\gamma _{1}}\ldots
\partial x_{d}^{\gamma _{d}}$, where multiindex $\gamma =\left( \gamma
_{1},\ldots ,\gamma _{d}\right) \in \mathbf{N}^{d},\nabla u=\left( \partial
_{1}u,\ldots ,\partial _{d}u\right) $ denotes the gradient of $u$ with
respect to $x$. For $k\in \mathbf{N}$, we denote $D^{k}u=\left( \partial
^{\gamma }u\right) _{\left\vert \gamma \right\vert =k}$.

Let $L_{p}(T)=L_{p}(E)$ is the space of $p-$integrable functions with norm, $%
p\geq 1$,

\begin{equation*}
\left\vert f\right\vert _{L_{p}(T)}=\left( \int_{0}^{T}\int \left\vert
f(t,x)\right\vert ^{p}dxdt\right) ^{1/p},\left\vert f\right\vert _{L_{\infty
}(T)}=\text{ess}\sup_{\left( t,x\right) \in E}\left\vert f(t,x)\right\vert .
\end{equation*}%
Similar space of functions on $\mathbf{R}^{d}$ is denoted $L_{p}\left( 
\mathbf{R}^{d}\right) .$

Let $\mathcal{S}\left( \mathbf{R}^{d}\right) $ be the Schwartz space of
smooth real valued rapidly decreasing functions. For $s\in \mathbf{N}$, we
define the Sobolev space $H_{p}^{n}\left( \mathbf{R}^{d}\right) $ (resp. $%
H_{p}^{n}\left( E\right) $) as closure of $C_{0}^{\infty }\left( \mathbf{R}%
^{d}\right) $ (resp. $C_{0}^{\infty }\left( E\right) $) with respect to the
norm 
\begin{equation*}
\left\vert f\right\vert _{n,p}=\sum_{\left\vert \beta \right\vert \leq
n}\left\vert D^{\beta }f\right\vert _{L_{p}\left( \mathbf{R}^{d}\right) },%
\text{ resp. }\left\vert g\right\vert _{n,p}=\sum_{\left\vert \beta
\right\vert \leq n}\left\vert D^{\beta }f\right\vert _{L_{p}\left( E\right)
}.
\end{equation*}%
For $\sigma \in \left( 0,2\right) $ and $v\in C_{0}^{\infty }\left( \mathbf{R%
}^{d}\right) $, we define the fractional Laplacian%
\begin{equation*}
\partial ^{\sigma }v\left( x\right) =\int \nabla _{y}^{\sigma }v\left(
x\right) \frac{dy}{\left\vert y\right\vert ^{d+\sigma }},x\in \mathbf{R}^{d},
\end{equation*}%
where%
\begin{equation*}
\nabla _{y}^{\sigma }v\left( x\right) =v\left( x+y\right) -v\left( x\right)
-\left( \nabla v\left( x\right) ,y\right) \chi _{\sigma }\left( y\right)
\end{equation*}%
with $\chi _{\sigma }\left( y\right) =1_{\left\{ \left\vert y\right\vert
\leq 1\right\} }1_{\sigma =1}+1_{\{\sigma \in \left( 1,2\right) \}}$ is the
integrand in the definition of $L^{\pi }$.

Given $\pi _{0}\in \mathfrak{A}^{\sigma },p\in \lbrack 1,\infty ),$ we
denote $H_{p}^{\pi _{0}}\left( \mathbf{R}^{d}\right) $ (resp. $\mathcal{H}%
_{p}^{\pi _{0}}\left( E\right) $) the closure in $L_{p}\left( \mathbf{R}%
^{d}\right) $ (resp. $L_{p}\left( E\right) $) of $C_{0}^{\infty }\left( 
\mathbf{R}^{d}\right) $ (resp. $C_{0}^{\infty }\left( E\right) $) with
respect to the norm 
\begin{equation*}
\left\vert f\right\vert _{\pi _{0},p}=\left\vert f\right\vert _{L_{p}\left( 
\mathbf{R}^{d}\right) }+\left\vert L^{\pi _{0}}f\right\vert _{L_{p}\left( 
\mathbf{R}^{d}\right) }\,,\text{ resp. }\left\vert g\right\vert _{\pi
_{0},p}=\left\vert g\right\vert _{L_{p}\left( E\right) }+\left\vert L^{\pi
_{0}}g\right\vert _{L_{p}\left( E\right) }.\,
\end{equation*}%
If $f_{n}\in C_{0}^{\infty }\left( \mathbf{R}^{d}\right) ,f_{n}\rightarrow f$
and $L^{\pi }f_{n}\rightarrow g$ in $L_{p}\left( \mathbf{R}^{d}\right) $ we
denote $g=L^{\pi }f$. Notice that $f_{n}\rightarrow 0$, $L^{\pi
}f_{n}\rightarrow h$ in $L_{p}\left( \mathbf{R}^{d}\right) $ implies that $%
h=0.$ Indeed, 
\begin{equation*}
\int \varphi L^{\pi }f_{n}=\int f_{n}L^{\pi ^{\ast }}\varphi \rightarrow
0,\varphi \in C_{0}^{\infty }\left( \mathbf{R}^{d}\right) ,
\end{equation*}%
where $\pi ^{\ast }\left( \Gamma \right) =\pi \left( -\Gamma \right) ,\Gamma
\in \mathcal{B}\left( \mathbf{R}_{0}^{d}\right) $, i.e. $\pi ^{\ast }\in 
\mathfrak{A}^{\sigma }$ as well. Note that if $\pi \in \mathfrak{A}^{\sigma }
$, then for any $f\in C_{0}^{\infty }\left( \mathbf{R}^{d}\right) ,$ 
\begin{equation*}
\left\vert L^{\pi }f\right\vert _{L_{p}\left( \mathbf{R}^{d}\right) }\leq
\left\vert \int_{\left\vert y\right\vert \leq 1}...\right\vert _{L_{p}\left( 
\mathbf{R}^{d}\right) }+\left\vert \int_{\left\vert y\right\vert
>1}...\right\vert _{L_{p}\left( \mathbf{R}^{d}\right) }\leq C\left\vert
f\right\vert _{2,p},
\end{equation*}%
that is $H_{p}^{2}\left( \mathbf{R}^{d}\right) \subseteq H_{p}^{\pi }\left( 
\mathbf{R}^{d}\right) $ and the embedding is continuous. The same holds for $%
H_{p}^{2}\left( E\right) \subseteq H_{p}^{\pi }\left( E\right) .$

We denote $\mathfrak{A=\cup }_{\sigma \in (0,2)}\mathfrak{A}^{\sigma }.$

We denote Fourier transform and its inverse%
\begin{eqnarray*}
\mathcal{F}v\left( \xi \right) &=&\hat{v}\left( \xi \right) =\int v\left(
x\right) e^{-i2\pi x\cdot \xi }dx,\xi \in \mathbf{R}^{d}, \\
\mathcal{F}^{-1}v\left( x\right) &=&\int v\left( \xi \right) e^{i2\pi x\cdot
\xi }d\xi ,x\in \mathbf{R}^{d},v\in \mathcal{S}\left( \mathbf{R}^{d}\right) .
\end{eqnarray*}

We denote $C_{b}^{\infty }\left( E\right) $ the space of bounded infinitely
differentiable in $x$ functions whose derivatives are bounded.

$C=C\left( \cdot ,\ldots ,\cdot \right) $ denotes constants depending only
on quantities appearing in parentheses. In a given context the same letter \
is (generally) used to denote different constants depending on the same set
of arguments.

We also introduce an auxiliary Levy measure $\mu _{0}$ on $\mathbf{R}%
_{0}^{d} $ such that the following assumption holds.

\textbf{Assumption A}$_{0}$. \emph{Let} $\mu ^{0}\in \mathfrak{A,}\chi
_{\left\{ \left\vert y\right\vert \leq 1\right\} }\mu ^{0}\left( dy\right)
=\mu ^{0}\left( dy\right) $, \emph{and}%
\begin{eqnarray*}
\int \left\vert y\right\vert \mu ^{0}\left( dy\right) +\int \left\vert \xi
\right\vert ^{2}[1+\zeta \left( \xi \right) ]^{d+3}\exp \left\{ -\phi
_{0}\left( \xi \right) \right\} d\xi &\leq &N_{0}\text{ if }\sigma \in
\left( 0,1\right) , \\
\int \left\vert y\right\vert ^{2}\mu ^{0}\left( dy\right) +\int \left\vert
\xi \right\vert ^{4}[1+\zeta \left( \xi \right) ]^{d+3}\exp \left\{ -\phi
_{0}\left( \xi \right) \right\} d\xi &\leq &N_{0}\text{ if }\sigma \in
\lbrack 1,2),
\end{eqnarray*}

\emph{where} 
\begin{eqnarray*}
\phi _{0}\left( \xi \right) &=&\int_{\left\vert y\right\vert \leq 1}\left[
1-\cos \left( 2\pi \xi \cdot y\right) \right] \mu ^{0}\left( dy\right) , \\
\zeta \left( \xi \right) &=&\int_{\left\vert y\right\vert \leq 1}\chi
_{\sigma }\left( y\right) \left\vert y\right\vert [\left( \left\vert \xi
\right\vert \left\vert y\right\vert \right) \wedge 1]\mu ^{0}\left(
dy\right) ,\xi \in \mathbf{R}^{d}.
\end{eqnarray*}%
\emph{In addition, we assume that for any} $\xi \in S_{d-1}=\left\{ \xi \in 
\mathbf{R}^{d}:\left\vert \xi \right\vert =1\right\} ,$ 
\begin{equation*}
\int_{\left\vert y\right\vert \leq 1}\left\vert \xi \cdot y\right\vert
^{2}\mu ^{0}\left( dy\right) \geq c_{1}>0.
\end{equation*}

For $\pi \in \mathfrak{A}=\cup _{\sigma \in \left( 0,2\right) }\mathfrak{A}%
^{\sigma }$ and $R>0$, we denote%
\begin{equation*}
\pi _{R}\left( \Gamma \right) =\int \chi _{\Gamma }\left( y/R\right) \pi
\left( dy\right) ,\Gamma \in \mathcal{B}\left( \mathbf{R}_{0}^{d}\right) .
\end{equation*}

\begin{definition}
We say that a continuous function $\kappa :(0,\infty )\rightarrow (0,\infty
) $ is a scaling function if $\lim_{R\rightarrow 0}\kappa \left( R\right)
=0,\lim_{R\rightarrow \infty }\kappa \left( R\right) =\infty $ and there is
a nondecreasing continuous function $l\left( \varepsilon \right)
,\varepsilon >0,$ such that $\lim_{\varepsilon \rightarrow 0}l\left(
\varepsilon \right) =0$ and 
\begin{equation*}
\kappa \left( \varepsilon r\right) \leq l\left( \varepsilon \right) \kappa
(r),r>0,\varepsilon >0.
\end{equation*}%
We call $l\left( \varepsilon \right) ,\varepsilon >0,$ a scaling factor of $%
\kappa $.
\end{definition}

For a scaling function $\kappa $ with a scaling factor $l$ and $\pi \in 
\mathfrak{A}^{\sigma }$ we introduce the following

\textbf{Assumption D}$\left( \kappa ,l\right) $\textbf{. }\emph{(i) For every%
} $R>0,$ 
\begin{equation*}
\tilde{\pi}_{R}\left( dy\right) =\kappa \left( R\right) \pi _{R}\left(
dy\right) \geq 1_{\left\{ \left\vert y\right\vert \leq 1\right\} }\mu
^{0}\left( dy\right) ,
\end{equation*}%
\emph{with }$\mu ^{0}=\mu ^{0;\pi }$\emph{\ satisfying Assumption }\textbf{A}%
$_{0}$. \emph{If }$\sigma =1$ \emph{we, in addition assume that }$%
\int_{R<\left\vert y\right\vert \leq R^{\prime }}y\mu ^{0}\left( dy\right)
=0 $ \emph{for any} $0<R<R^{\prime }\leq 1.$ \emph{\ Here }$\tilde{\pi}%
_{R}\left( dy\right) =\kappa \left( R\right) \pi _{R}\left( dy\right) .$

\emph{(ii) There exist }$\alpha _{1}$\emph{\ and }$\alpha _{2}$\emph{\ and a
constant }$N>0$ \emph{such that}%
\begin{equation*}
\int_{\left\vert z\right\vert \leq 1}\left\vert z\right\vert ^{\alpha _{1}}%
\tilde{\pi}_{R}(dz)+\int_{\left\vert z\right\vert >1}\left\vert z\right\vert
^{\alpha _{2}}\tilde{\pi}_{R}(dz)\leq N\text{ }\forall R>0,
\end{equation*}%
\emph{where }$\alpha _{1},\alpha _{2}\in (0,1]\text{ \emph{if} }\sigma \in
(0,1)\text{; }\alpha _{1},\alpha _{2}\in (1,2]\text{ if }\sigma \in (1,2)$%
\emph{; }$\alpha _{1}\in (1,2]$\emph{\ and }$\alpha _{2}\in \lbrack 0,1)$%
\emph{\ for }$\sigma =1$\emph{.}

(iii) \emph{Let }$\gamma \left( t\right) =\inf \left( s>0:l\left( s\right)
>t\right) ,t>0$. With $I_{1}=\left\{ t>0:\gamma \left( t\right) \leq
1\right\} ,I_{2}=\left\{ t>0:\gamma \left( t\right) >1\right\} $ we have 
\begin{equation*}
\int_{I_{1}}[t\gamma (t)^{-\alpha _{1}}+1_{\sigma \in \left( 1,2\right)
}\gamma \left( t\right) ^{-1}]dt\leq N_{1}<\infty ,
\end{equation*}%
and%
\begin{equation*}
\int_{I_{2}}[\gamma \left( t\right) ^{-(1+\alpha _{2})}+\gamma \left(
t\right) ^{-2\alpha _{2}}]dt\leq N_{1}<\infty .
\end{equation*}

The main result of this paper for (\ref{2'}) is

\begin{theorem}
\label{thm:main}Let \thinspace $p>1,\pi _{0},\pi \in \mathfrak{A}^{\sigma
},\lambda >0$. Assume there is a scaling function $\kappa $ such that 
\textbf{D}$\left( \kappa ,l\right) $\textbf{\ }hold for both, $\pi $ and $%
\pi _{0}$.

Then for each $f\in L_{p}(\mathbf{R}^{d})$ there is a unique $u\in
H_{p}^{\pi _{0}}\left( \mathbf{R}^{d}\right) $ solving (\ref{1'}). Moreover,
there is $C=C\left( d,p,\kappa ,l,N_{0},N,N_{1},c_{1}\right) $ such that 
\begin{eqnarray*}
\left\vert L^{\pi _{0}}u\right\vert _{L_{p}\left( \mathbf{R}^{d}\right) }
&\leq &C\left\vert f\right\vert _{L_{p}\left( \mathbf{R}^{d}\right) }, \\
\left\vert u\right\vert _{L_{p}\left( \mathbf{R}^{d}\right) } &\leq &\frac{1%
}{\lambda }\left\vert f\right\vert _{L_{p}\left( \mathbf{R}^{d}\right) }.
\end{eqnarray*}
\end{theorem}

The main result for (\ref{1'}) is

\begin{theorem}
\label{t1}Let \thinspace $p\in \left( 1,\infty \right) ,\pi _{0},\pi \in 
\mathfrak{A}^{\sigma }$. Assume there is a scaling function $\kappa $ such
that \textbf{D}$\left( \kappa ,l\right) $\textbf{\ }hold for both, $\pi $
and $\pi _{0}$.

Then for each $f\in L_{p}(E)$ there is a unique $u\in \mathcal{H}_{p}^{\pi
_{0}}\left( E\right) $ solving (\ref{1'}). Moreover, there is $C=C\left(
d,p,\kappa ,l,N_{0},N,N_{1},c_{1}\right) $ such that%
\begin{eqnarray*}
\left\vert L^{\pi _{0}}u\right\vert _{L_{p}\left( E\right) } &\leq
&C\left\vert f\right\vert _{L_{p}\left( E\right) }, \\
\left\vert u\right\vert _{L_{p}\left( E\right) } &\leq &\left( \frac{1}{%
\lambda }\wedge T\right) \left\vert f\right\vert _{L_{p}\left( E\right) }.
\end{eqnarray*}
\end{theorem}

\begin{remark}
Assumption \textbf{D}$\left( \kappa ,l\right) $ holds for both, $\pi ,\pi
_{0}$, means that $\kappa ,l,$ and the parameters $\alpha _{1},\alpha
_{2},N,N_{1},N_{0},c_{1}$ are the same.
\end{remark}

\subsection{Example}

Let $\mu \left( dt\right) $ be a measure on $\left( 0,\infty \right) $ such
that $\int_{0}^{\infty }\left( 1\wedge t\right) \mu \left( dt\right) <\infty 
$, and let 
\begin{equation*}
\phi \left( r\right) =\int_{0}^{\infty }\left( 1-e^{-rt}\right) \mu \left(
dt\right) ,r\geq 0,
\end{equation*}%
be Bernstein function (see \cite{vo}, \cite{kk}). Let%
\begin{equation*}
j\left( r\right) =\int_{0}^{\infty }\left( 4\pi t\right) ^{-\frac{d}{2}}\exp
\left( -\frac{r^{2}}{4t}\right) \mu \left( dt\right) ,r>0.
\end{equation*}%
We consider $\pi \in \mathfrak{A}=\cup _{\sigma \in \left( 0,2\right) }%
\mathfrak{A}^{\sigma }\,\ $defined in radial and angular coordinates $%
r=\left\vert y\right\vert ,w=y/r,$ as%
\begin{equation}
\pi \left( \Gamma \right) =\int_{0}^{\infty }\int_{\left\vert w\right\vert
=1}\chi _{\Gamma }\left( rw\right) a\left( r,w\right) j\left( r\right)
r^{d-1}S\left( dw\right) dr,\Gamma \in \mathcal{B}\left( \mathbf{R}%
_{0}^{d}\right) ,  \label{fe1}
\end{equation}%
where $S\left( dw\right) $ is a finite measure on the unite sphere on $%
\mathbf{R}^{d}$. If $S\left( dw\right) =dw$ is the Lebesgue measure on the
unit sphere,then%
\begin{equation*}
\pi \left( \Gamma \right) =\pi ^{J,a}\left( \Gamma \right) =\int_{\mathbf{R}%
^{d}}\chi _{\Gamma }\left( y\right) a\left( \left\vert y\right\vert
,y/\left\vert y\right\vert \right) J\left( y\right) dy,\Gamma \in \mathcal{B}%
\left( \mathbf{R}_{0}^{d}\right) ,
\end{equation*}%
where $J\left( y\right) =j\left( \left\vert y\right\vert \right) ,y\in 
\mathbf{R}^{d}.$ Let $\pi _{0}=\pi ^{J,1},$ i.e.,%
\begin{equation}
\pi _{0}\left( \Gamma \right) =\int_{\mathbf{R}^{d}}\chi _{\Gamma }\left(
y\right) J\left( y\right) dy,\Gamma \in \mathcal{B}\left( \mathbf{R}%
_{0}^{d}\right) .  \label{fe2}
\end{equation}

We assume

\textbf{H. }(i) There is $N>0$ so that%
\begin{equation*}
N^{-1}\phi \left( r^{-2}\right) r^{-d}\leq j\left( r\right) \leq N\phi
\left( r^{-2}\right) r^{-d},r>0.
\end{equation*}

(ii) There are $0<\delta _{1}\leq \delta _{2}<1$ and $N>0$ so that for $%
0<r\leq R$%
\begin{equation*}
N^{-1}\left( \frac{R}{r}\right) ^{\delta _{1}}\leq \frac{\phi \left(
R\right) }{\phi \left( r\right) }\leq N\left( \frac{R}{r}\right) ^{\delta
_{2}}.
\end{equation*}

\textbf{G. }There is $\rho _{0}\left( w\right) ,\left\vert w\right\vert =1,$
such that $\rho _{0}\left( w\right) \leq a\left( r,w\right) \leq
1,r>0,\left\vert w\right\vert =1,$ and for every $\left\vert \xi \right\vert
=1,$ 
\begin{equation*}
\int_{\left\vert w\right\vert =1}\left\vert \xi \cdot w\right\vert ^{2}\rho
_{0}\left( w\right) S\left( dw\right) \neq 0.
\end{equation*}

For example, in \cite{vo} and \cite{kk} among others the following specific
Bernstein functions satisfying \textbf{H} are listed:

(0) $\phi \left( r\right) =\sum_{i=1}^{n}r^{\alpha _{i}},\alpha _{i}\in
\left( 0,1\right) ,i=1,\ldots ,n;$

(1) $\phi \left( r\right) =\left( r+r^{\alpha }\right) ^{\beta },\alpha
,\beta \in \left( 0,1\right) ;$

(2) $\phi \left( r\right) =r^{\alpha }\left( \ln \left( 1+r\right) \right)
^{\beta },\alpha \in \left( 0,1\right) ,\beta \in \left( 0,1-\alpha \right)
; $

(3) $\phi \left( r\right) =\left[ \ln \left( \cosh \sqrt{r}\right) \right]
^{\alpha },\alpha \in \left( 0,1\right) .$

The following statement holds.

\begin{remark}
\label{exre}Let $\pi ,\pi _{0}$ be given by (\ref{fe1}) and (\ref{fe2})$.$
Assume \textbf{H} and \textbf{G} hold.

a) If $2\delta _{1}>1$, then Theorems \ref{t1} (resp. \ref{thm:main}) hold
in $H_{p}^{\pi _{0}}\left( E\right) $ (resp. \thinspace $H_{p}^{\pi
_{0}}\left( \mathbf{R}^{d}\right) \,$).

b) If $2\delta _{2}<1$ and $2\delta _{1}>\delta _{2}$, then Theorems \ref{t1}
(resp. \ref{thm:main}) hold in $H_{p}^{\pi _{0}}\left( E\right) $ (resp.
\thinspace $H_{p}^{\pi _{0}}\left( \mathbf{R}^{d}\right) \,$).
\end{remark}

\begin{proof}
We verify that the assumptions of Theorems \ref{t1} and \ref{thm:main} hold.
Indeed, \textbf{H} implies that there are $0<c\leq C$ so that 
\begin{eqnarray*}
cr^{-d-2\delta _{1}} &\leq &j\left( r\right) \leq Cr^{-d-2\delta _{2}},r\leq
1, \\
cr^{-d-2\delta _{2}} &\leq &j\left( r\right) \leq Cr^{-d-2\delta _{1}},r>1.
\end{eqnarray*}%
Hence $2\delta _{1}\leq \sigma \leq 2\delta _{2}.$ In this case $\kappa
\left( R\right) =j\left( R\right) ^{-1}R^{-d},R>0,$ is a scaling function: $%
\kappa \left( \varepsilon R\right) \leq l\left( \varepsilon \right) \kappa
\left( R\right) ,\varepsilon ,R>0,$ with%
\begin{equation*}
l\left( \varepsilon \right) =\left\{ 
\begin{array}{cc}
C_{1}\varepsilon ^{2\delta _{1}} & \text{if }\varepsilon \leq 1, \\ 
C_{1}\varepsilon ^{2\delta _{2}} & \text{if }\varepsilon >1%
\end{array}%
\right.
\end{equation*}%
for some $C_{1}>0$. Hence 
\begin{equation*}
\gamma \left( t\right) =l^{-1}\left( t\right) =\left\{ 
\begin{array}{c}
C_{1}^{-1/2\delta _{1}}t^{1/2\delta _{1}}\text{ if }t\leq C_{1}, \\ 
C_{1}^{-1/2\delta _{2}}t^{1/2\delta _{2}}\text{ if }t>C_{1}.%
\end{array}%
\right.
\end{equation*}%
We see easily that $\alpha _{1}$ is any number $>2\delta _{2}$ and $\alpha
_{2}$ is any number $<2\delta _{1}.$ The measure $\mu ^{0}$ for $\pi $ is%
\begin{equation*}
\mu ^{0}\left( dy\right) =\mu ^{0,\pi }\left( dy\right) =c_{1}\int \chi
_{dy}\left( rw\right) \chi _{\left\{ r\leq 1\right\} }r^{-1-2\delta
_{1}}\rho _{0}\left( w\right) S\left( dw\right) dr;
\end{equation*}%
and $\mu ^{0}$ for $\pi _{0}$ is%
\begin{equation*}
\mu ^{0}\left( dy\right) =\mu ^{0,\pi _{0}}\left( dy\right) =c_{1}^{\prime
}\int \chi _{dy}\left( rw\right) \chi _{\left\{ r\leq 1\right\}
}r^{-1-2\delta _{1}}dwdr.
\end{equation*}%
Integrability conditions \textbf{D}$\left( \kappa ,l\right) $(iii) easily
follow from a) or b).
\end{proof}

\section{Auxiliary results}

In this section we present some auxiliary results.

\subsection{Some \thinspace $L_{p}$ estimates}

We start with the following observation.

\begin{remark}
\label{re2 copy(1)}If $\pi \in \mathfrak{A}^{\sigma }$, then for any $f\in
C_{0}^{\infty }\left( \mathbf{R}^{d}\right) ,$ 
\begin{equation*}
\left\vert L^{\pi }f\right\vert _{L_{p}\left( \mathbf{R}^{d}\right) }\leq
\left\vert \int_{\left\vert y\right\vert \leq 1}...\right\vert _{L_{p}\left( 
\mathbf{R}^{d}\right) }+\left\vert \int_{\left\vert y\right\vert
>1}...\right\vert _{L_{p}\left( \mathbf{R}^{d}\right) }\leq C\left\vert
f\right\vert _{2,p}.
\end{equation*}%
Hence $H_{p}^{2}\left( \mathbf{R}^{d}\right) \subseteq H_{p}^{\pi }\left( 
\mathbf{R}^{d}\right) $ and the embedding is continuous. The same holds for $%
H_{p}^{2}\left( E\right) \subseteq H_{p}^{\pi }\left( E\right) .$
\end{remark}

We will use the following equality for Sobolev norm estimates.

\begin{lemma}
\label{lr1}$($Lemma 2.1 in \cite{ko}$)$ For $\alpha \in (0,1)$ and $u\in 
\mathcal{S}(\mathbf{R}^{d})$, 
\begin{equation}
u\left( x+y\right) -u(x)=C\int k^{(\alpha )}(y,z)\partial ^{\alpha }u(x-z)dz,
\label{ff0}
\end{equation}%
where the constant $C=C(\alpha ,d)$ and 
\begin{equation*}
k^{(\alpha )}(z,y)=|z+y|^{-d+\alpha }-|z|^{-d+\alpha }.
\end{equation*}%
Moreover, there is a constant $C=C(\alpha ,d)$ such that for each $y\in 
\mathbf{R}^{d}$ 
\begin{equation*}
\int |k^{(\alpha )}(z,y)|dz\leq C|y|^{\alpha }.
\end{equation*}
\end{lemma}

\begin{corollary}
\label{cor1}Let $\alpha\in(0,1),p\geq1$. Then

(i) for $y\in \mathbf{R}^{d},$%
\begin{equation}
\left\vert \partial ^{\alpha }u\right\vert _{L_{p}\left( \mathbf{R}%
^{d}\right) }\leq C\left\vert u\right\vert _{H_{p}^{1}\left( \mathbf{R}%
^{d}\right) },u\in \mathcal{S}\left( \mathbf{R}^{d}\right) ;  \label{ff1}
\end{equation}%
\begin{eqnarray}
\left\vert u\left( \cdot +y\right) -u\right\vert _{L_{p}\left( \mathbf{R}%
^{d}\right) } &\leq &C\left\vert \partial ^{\alpha }u\right\vert
_{L_{p}\left( \mathbf{R}^{d}\right) }\left\vert y\right\vert ^{\alpha },
\label{ff11} \\
\left\vert u\left( \cdot +y\right) -u-y\cdot \nabla u\right\vert
_{L_{p}\left( \mathbf{R}^{d}\right) } &\leq &C\left\vert \partial ^{\alpha
}\nabla u\right\vert _{L_{p}\left( \mathbf{R}^{d}\right) }\left\vert
y\right\vert ^{1+\alpha },  \label{ff22}
\end{eqnarray}%
$u\in \mathcal{S}\left( \mathbf{R}^{d}\right) .$

(ii) for any $\varepsilon >0,$ 
\begin{equation*}
\partial ^{\alpha }\left[ u\left( \varepsilon \cdot \right) \right]
=\varepsilon ^{\alpha }(\partial ^{\alpha }u)\left( \varepsilon x\right)
,\partial ^{\alpha }\nabla \left[ u\left( \varepsilon \cdot \right) \right]
=\varepsilon ^{1+\alpha }(\partial ^{\alpha }\nabla u)\left( \varepsilon
x\right) ,x\in \mathbf{R}^{d},
\end{equation*}
\end{corollary}

\begin{proof}
Let $u\in \mathcal{S}(\mathbf{R}^{d})$. Since for $x\in \mathbf{R}^{d},$ 
\begin{eqnarray*}
|\partial ^{\alpha }u\left( x\right) | &\leq &\int_{\left\vert y\right\vert
\leq 1}\int_{0}^{1}|y\cdot \nabla u\left( x+sy\right) |ds\frac{dy}{%
\left\vert y\right\vert ^{d+\alpha }} \\
&&+\int_{\left\vert y\right\vert >1}\left[ \left\vert u\left( x+y\right)
\right\vert +u\left( x\right) \right] \frac{dy}{\left\vert y\right\vert
^{d+\alpha }}
\end{eqnarray*}%
(\ref{ff1}) follows. Applying generalized Minkowski inequality to (\ref{ff0}%
), we derive easily (\ref{ff11}). Similarly, using%
\begin{equation*}
u\left( x+y\right) -u\left( x\right) -y\cdot \nabla u\left( x\right)
=\int_{0}^{1}y\cdot \left[ \nabla u\left( x+sy\right) -\nabla u\left(
x\right) \right] ds
\end{equation*}%
and (\ref{ff11}) we derive (\ref{ff22}).

Changing the variable of integration, 
\begin{equation*}
\partial ^{\alpha }[u\left( \varepsilon \cdot \right) ]\left( x\right)
=\varepsilon ^{\alpha }\int \left[ u\left( \varepsilon x+y\right) -u\left(
\varepsilon x\right) \right] \frac{dy}{\left\vert y\right\vert ^{d+\alpha }}%
=\varepsilon ^{\alpha }(\partial ^{\alpha }u)\left( \varepsilon x\right)
,x\in \mathbf{R}^{d}.
\end{equation*}
\end{proof}

\begin{corollary}
\label{cor2}Let $\pi \in \mathfrak{A}^{\sigma }$ and 
\begin{equation*}
\int_{\left\vert z\right\vert \leq 1}\left\vert z\right\vert ^{\alpha
_{1}}\pi (dz)+\int_{\left\vert z\right\vert >1}\left\vert z\right\vert
^{\alpha _{2}}\pi (dz)\leq N,
\end{equation*}%
where\emph{\ }$\alpha _{1},\alpha _{2}\in (0,1]\text{ if }\sigma \in (0,1)%
\text{; }\alpha _{1},\alpha _{2}\in (1,2]\text{ if }\sigma \in (1,2)$\emph{; 
}$\alpha _{1}\in (1,2]$\emph{\ }and\emph{\ }$\alpha _{2}\in (0,1]$\emph{\ }%
for\emph{\ }$\sigma =1$\emph{. }

Then there is a constant $C=C\left( N\right) $ such that for any $v\in 
\mathcal{S}\left( \mathbf{R}^{d}\right) ,$ (assuming $\partial ^{\gamma
}=\nabla $ if $\gamma =1$),%
\begin{eqnarray*}
\left\vert L^{\pi }v\right\vert _{L_{1}} &\leq &C\left( \left\vert \partial
^{\alpha _{1}}v\right\vert _{L_{1}}+\left\vert v\right\vert _{L_{1}}\right) ,
\\
\left\vert L^{\pi }v\right\vert _{L_{1}} &\leq &C\left( \left\vert \nabla
v\right\vert _{L_{1}}+\left\vert \partial ^{\alpha _{2}}v\right\vert
_{L_{1}}\right) ,
\end{eqnarray*}%
if $\sigma \in \left( 0,1\right) $; 
\begin{eqnarray*}
\left\vert L^{\pi }v\right\vert _{L_{1}} &\leq &C\left( \left\vert \partial
^{\alpha _{1}-1}\nabla v\right\vert _{L_{1}}+\left\vert v\right\vert
_{L_{1}}\right) , \\
\left\vert L^{\pi }v\right\vert _{L_{1}} &\leq &C\left( \left\vert
D^{2}v\right\vert _{L_{1}}+\left\vert \partial ^{\alpha _{2}}v\right\vert
_{L_{1}}\right)
\end{eqnarray*}%
if $\sigma =1$ ; 
\begin{eqnarray*}
\left\vert L^{\pi }v\right\vert _{L_{1}} &\leq &C\left( \left\vert \partial
^{\alpha _{1}-1}\nabla v\right\vert _{L_{1}}+\left\vert \nabla v\right\vert
_{L_{1}}\right) , \\
\left\vert L^{\pi }v\right\vert _{L_{1}} &\leq &C\left( \left\vert
D^{2}v\right\vert _{L_{1}}+\left\vert \partial ^{\alpha _{2}-1}\nabla
v\right\vert _{L_{1}}\right)
\end{eqnarray*}%
if $\sigma \in \left( 1,2\right) $.
\end{corollary}

\begin{proof}
By Corollary \ref{cor1}, for $\left\vert y\right\vert \leq 1$%
\begin{equation*}
|v\left( \cdot +y\right) -v-\chi _{\sigma }\left( y\right) y\cdot \nabla
u|_{L_{1}}\leq \left\{ 
\begin{array}{cc}
\left\vert \partial ^{\alpha _{1}}v\right\vert _{L_{1}}\left\vert
y\right\vert ^{\alpha _{1}} & \text{if }\sigma \in \left( 0,1\right) \\ 
\left\vert \partial ^{\alpha _{1}-1}\nabla v\right\vert _{L_{1}}\left\vert
y\right\vert ^{\alpha _{1}} & \text{if }\sigma \in \lbrack 1,2)%
\end{array}%
\right.
\end{equation*}%
and for $\left\vert y\right\vert >1,$%
\begin{equation*}
|v\left( \cdot +y\right) -v-\chi _{\sigma }\left( y\right) y\cdot \nabla
u|_{L_{1}}\leq \left\{ 
\begin{array}{cc}
2\left\vert v\right\vert _{L_{1}} & \text{if }\sigma \in (0,1], \\ 
2\left\vert \nabla v\right\vert _{L_{1}}\left\vert y\right\vert ^{\alpha
_{2}} & \text{if }\sigma \in (1,2).%
\end{array}%
\right.
\end{equation*}

On the other hand, for $\left\vert y\right\vert \leq 1$%
\begin{equation*}
|v\left( \cdot +y\right) -v-\chi _{\sigma }\left( y\right) y\cdot \nabla
u|_{L_{1}}\leq \left\{ 
\begin{array}{cc}
\left\vert \nabla v\right\vert _{L_{1}}\left\vert y\right\vert ^{\alpha _{1}}
& \text{if }\sigma \in \left( 0,1\right) , \\ 
\left\vert D^{2}v\right\vert _{L_{1}}\left\vert y\right\vert ^{\alpha _{1}}
& \text{if }\sigma \in \lbrack 1,2),%
\end{array}%
\right.
\end{equation*}%
and for $\left\vert y\right\vert >1,$%
\begin{equation*}
|v\left( \cdot +y\right) -v-\chi _{\sigma }\left( y\right) y\cdot \nabla
u|_{L_{1}}\leq \left\{ 
\begin{array}{cc}
\left\vert \partial ^{\alpha _{2}}v\right\vert _{L_{1}} & \text{if }\sigma
\in (0,1], \\ 
2\left\vert \partial ^{\alpha _{2}-1}\nabla v\right\vert _{L_{1}}\left\vert
y\right\vert ^{\alpha _{2}} & \text{if }\sigma \in (1,2).%
\end{array}%
\right.
\end{equation*}%
The statement follows.
\end{proof}

In addition, the following holds.

\begin{lemma}
\label{le3}For any $\beta \in \left[ 0,1\right] ,a\geq 0,\left\vert
z\right\vert \leq 1$ and $u\in \mathcal{S}\left( \mathbf{R}^{d}\right) ,$ 
\begin{eqnarray*}
&&\int_{\left\vert x\right\vert \geq a}\left\vert u\left( x+z\right)
-u\left( x\right) \right\vert dx \\
&\leq &2^{1-\beta }\left( \int_{\left\vert x\right\vert \geq (a-1)\vee
0}\left\vert u\left( x\right) \right\vert dx\right) ^{1-\beta }\left(
\int_{\left\vert x\right\vert \geq (a-1)\vee 0}\left\vert \nabla u\left(
x\right) \right\vert dx\right) ^{\beta }\left\vert z\right\vert ^{\beta },
\end{eqnarray*}
\end{lemma}

\begin{proof}
Let $u\in \mathcal{S}\left( \mathbf{R}^{d}\right) $. For $\beta \in \left[
0,1\right] ,x,z\in \mathbf{R}^{d},$ 
\begin{equation*}
\left\vert u\left( x+z\right) -u\left( x\right) \right\vert \leq \left\vert
u\left( x+z\right) -u\left( x\right) \right\vert ^{1-\beta }\left(
\int_{0}^{1}\left\vert \nabla u\left( x+sz\right) \right\vert ds\right)
^{\beta }\left\vert z\right\vert ^{\beta },
\end{equation*}%
and%
\begin{eqnarray*}
&&\left\vert u\left( x+z\right) -u(x)-z\cdot \nabla u(x)\right\vert \\
&\leq &\int_{0}^{1}\left\vert \nabla u\left( x+sz\right) -\nabla u\left(
x\right) \right\vert ds\left\vert z\right\vert
\end{eqnarray*}%
By H\"{o}lder inequality, for $\left\vert z\right\vert \leq 1,$%
\begin{eqnarray*}
&&\int_{\left\vert x\right\vert \geq a}\left\vert u\left( x+z\right)
-u\left( x\right) \right\vert dx \\
&\leq &\int_{\left\vert x\right\vert \geq a}\left\vert u\left( x+z\right)
-u\left( x\right) \right\vert ^{1-\beta }\left( \int_{0}^{1}\left\vert
\nabla u\left( x+sz\right) \right\vert ds\right) ^{\beta }dx\left\vert
z\right\vert ^{\beta } \\
&\leq &\left( \int_{\left\vert x\right\vert \geq a}\left\vert u\left(
x+z\right) -u\left( x\right) \right\vert dx\right) ^{1-\beta }\left(
\int_{0}^{1}\int_{\left\vert x\right\vert \geq a}\left\vert \nabla u\left(
x+sz\right) \right\vert dsdx\right) ^{\beta }\left\vert z\right\vert ^{\beta
} \\
&\leq &\left( 2\int_{\left\vert x\right\vert \geq (a-1)\vee 0}\left\vert
u\left( x\right) \right\vert dx\right) ^{1-\beta }\left( \int_{\left\vert
x\right\vert \geq (a-1)\vee 0}\left\vert \nabla u\left( x\right) \right\vert
dx\right) ^{\beta }\left\vert z\right\vert ^{\beta }.
\end{eqnarray*}
\end{proof}

\begin{corollary}
\label{c2}For any $\beta \in \left[ 0,1\right] ,a\geq 0,\left\vert
z\right\vert \leq 1$ and $u\in \mathcal{S}\left( \mathbf{R}^{d}\right) ,$ 
\begin{eqnarray*}
&&\int_{\left\vert x\right\vert \geq a}\left\vert u\left( x+z\right)
-u(x)-z\cdot \nabla u(x)\right\vert dx \\
&\leq &2^{1-\beta }\left( \int_{\left\vert x\right\vert \geq (a-1)\vee
0}\left\vert \nabla u\left( x\right) \right\vert dx\right) ^{1-\beta }\left(
\int_{\left\vert x\right\vert \geq (a-1)\vee 0}\left\vert D^{2}u\left(
x\right) \right\vert dx\right) ^{\beta }\left\vert z\right\vert ^{1+\beta }.
\end{eqnarray*}
\end{corollary}

\begin{proof}
For $\beta \in \left[ 0,1\right] ,x,z\in \mathbf{R}^{d},\left\vert
z\right\vert \leq 1,$ 
\begin{eqnarray*}
&&\left\vert u\left( x+z\right) -u(x)-z\cdot \nabla u(x)\right\vert \\
&\leq &\int_{0}^{1}\left\vert \nabla u\left( x+sz\right) -\nabla u\left(
x\right) \right\vert ds\left\vert z\right\vert ,
\end{eqnarray*}%
and the claim follows by Lemma \ref{le3}.
\end{proof}

\subsection{Density estimates}

We start with the following simple statement about the existence of a
probability density function (pdf).

\begin{lemma}
\label{l1}Let $\mu ^{0}$ be a nonnegative measure on $\mathbf{R}_{0}^{d}$
such that $\chi _{\left\vert y\right\vert \leq 1}\mu ^{0}\left( dy\right)
=\mu ^{0}\left( dy\right) $ and 
\begin{eqnarray*}
\int \left\vert y\right\vert d\mu ^{0} &\leq &K_{0}\text{ if }\sigma \in
\left( 0,1\right) , \\
\int \left\vert y\right\vert ^{2}d\mu ^{0} &\leq &K_{0}\text{ if }\sigma \in
\{1,2).
\end{eqnarray*}%
Let $\eta $ be a r.v. such that 
\begin{equation}
\mathbf{E}e^{i2\pi \xi \cdot \eta }=\exp \left\{ \psi _{0}\left( \xi \right)
\right\} ,\xi \in \mathbf{R}^{d},  \label{a0}
\end{equation}%
where 
\begin{equation*}
\psi _{0}\left( \xi \right) =\int \left[ e^{-i2\pi \xi \cdot y}-1-\chi
_{\sigma }\left( y\right) i2\pi \xi \cdot y\right] \mu ^{0}\left( dy\right)
,\xi \in \mathbf{R}^{d}.
\end{equation*}%
Assume $n\geq 0$ and 
\begin{equation}
\int \left\vert \xi \right\vert ^{n}[1+1_{n\geq 1}\zeta \left( \xi \right)
]^{d+3}\exp \left\{ \phi _{0}\left( \xi \right) \right\} d\xi \leq K_{0},
\label{b}
\end{equation}%
where $\phi _{0}\left( \xi \right) =\func{Re}\psi _{0}\left( \xi \right)
,\xi \in \mathbf{R}^{d}$ and%
\begin{equation*}
\zeta \left( \xi \right) =\int_{\left\vert y\right\vert \leq 1}\chi _{\sigma
}\left( y\right) \left\vert y\right\vert [\left( \left\vert \xi \right\vert
\left\vert y\right\vert \right) \wedge 1]\mu ^{0}\left( dy\right) ,\xi \in 
\mathbf{R}^{d}.
\end{equation*}%
Then $\eta $ has a pdf $p_{0}\left( x\right) ,x\in \mathbf{R}^{d},$ such
that 
\begin{equation*}
\sup_{x}\left\vert \partial ^{\beta }p_{0}\left( x\right) \right\vert +\int
(1+\left\vert x\right\vert ^{2})\left\vert \partial ^{\beta }p_{0}\left(
x\right) \right\vert dx\leq C\text{ }\forall \left\vert \beta \right\vert
\leq n
\end{equation*}%
for some $C=C\left( d,K_{0}\right) .$
\end{lemma}

\begin{proof}
By Proposition I.2.5 in \cite{sa}, $\eta $ has a continuos bounded density 
\begin{equation}
p_{0}\left( x\right) =\int e^{-i2\pi x\cdot \xi }\exp \left\{ \psi
_{0}\left( \xi \right) \right\} d\xi  \label{d2}
\end{equation}%
if 
\begin{equation*}
\int \exp \left\{ -\phi _{0}\left( \xi \right) \right\} d\xi <\infty \text{.}
\end{equation*}%
The assumption (\ref{b}) implies that for any multiindex $\left\vert \beta
\right\vert \leq n,$ 
\begin{equation*}
\partial ^{\beta }p_{0}\left( x\right) =\int e^{-i2\pi x\cdot \xi }\left(
-i2\pi \xi \right) ^{\beta }\exp \left\{ \psi _{0}\left( \xi \right)
\right\} d\xi ,x\in \mathbf{R}^{d},
\end{equation*}%
is a bounded continuous function. The function $\left( 1+\left\vert
x\right\vert ^{2}\right) $ $\partial ^{\beta }p_{0}$ is integrable if 
\begin{eqnarray}
&&(-i2\pi x_{j})^{d+1}\left( -i2\pi x_{k}\right) ^{2}\partial ^{\beta
}p_{0}\left( x\right)  \label{d3} \\
&=&\int \partial _{\xi _{j}}^{d+1}\partial _{\xi _{k}}^{2}[e^{-i2\pi x\cdot
\xi }]\left( -i2\pi \xi \right) ^{\beta }\exp \left\{ \psi _{0}\left( \xi
\right) \right\} d\xi  \notag \\
&=&\left( -1\right) ^{d+3}\int e^{-ix\cdot \xi }\partial _{\xi
_{j}}^{d+1}\partial _{\xi _{k}}^{2}[\left( -i2\pi \xi \right) ^{\beta }\exp
\left\{ \psi _{0}\left( \xi \right) \right\} ]d\xi  \notag
\end{eqnarray}%
is bounded for all $j,k$. Since $\partial ^{\mu }\psi _{0}\left( \xi \right) 
$ is bounded for $\left\vert \mu \right\vert \geq 2$ and%
\begin{equation*}
|\nabla \psi _{0}\left( \xi \right) |\leq C\left( 1+\zeta \left( \xi \right)
\right) ,\xi \in \mathbf{R}^{d},
\end{equation*}%
the boundedness of (\ref{d3}) follows from assumption (\ref{b}). Therefore
\thinspace $p_{0}\left( x\right) $ has $n$ bounded continuous derivatives
and for any multiindex $\left\vert \beta \right\vert \leq n,$ 
\begin{equation}
\int \left( 1+\left\vert x\right\vert ^{2}\right) \left\vert \partial
^{\beta }p_{0}\left( x\right) \right\vert dx\leq C  \label{d4}
\end{equation}%
with $C=C\left( d,K_{0}\right) .$
\end{proof}

We will need the following tail estimate.

\begin{lemma}
\label{cor10}Let $\pi \in \mathfrak{A}.$ Assume 
\begin{equation}
\int_{\left\vert z\right\vert \leq 1}\left\vert z\right\vert ^{\alpha
_{1}}\pi (dz)+\int_{\left\vert z\right\vert >1}\left\vert z\right\vert
^{\alpha _{2}}\pi (dz)\leq N,  \label{fa5}
\end{equation}%
where $\alpha _{1},\alpha _{2}\in (0,1]\text{ if }\sigma \in (0,1)\text{; }%
\alpha _{1},\alpha _{2}\in (1,2]\text{ if }\sigma \in (1,2)$; $\alpha
_{1}\in (1,2]$\ and $\alpha _{2}\in \lbrack 0,1)$\ if $\sigma =1$. Let $%
\zeta _{t}$ be the associated Levy process, that is%
\begin{equation*}
\mathbf{E}e^{i2\pi \xi \cdot \zeta _{t}}=\exp \{\psi \left( \xi \right)
t\},t\geq 0,
\end{equation*}%
with%
\begin{equation*}
\psi \left( \xi \right) =\int \left[ e^{i2\pi \xi \cdot y}-1-i2\pi \chi
_{\sigma }\left( y\right) y\cdot \xi \right] d\pi ,\xi \in \mathbf{R}^{d}.
\end{equation*}%
Let $t>0$ and $\mathcal{L}_{t}\left( dy\right) $ be the distribution measure
of $\zeta _{t}$ on $\mathbf{R}^{d}$. Then for each $\delta >0$ there is a
constant $C=C\left( \delta ,N\right) $ such that%
\begin{equation*}
\mathcal{L}_{t}(\{\left\vert y\right\vert >\delta \})\leq Ct.
\end{equation*}
\end{lemma}

\begin{proof}
Recall 
\begin{equation}
\zeta _{t}=\int_{0}^{t}\int \chi _{\sigma }(y)yq(ds,dy)+\int_{0}^{t}\int
(1-\chi _{\sigma }(y))yp(ds,dy),t\geq 0,  \label{fa3}
\end{equation}%
$p(ds,dy)$ is Poisson point measure with 
\begin{equation*}
\mathbf{E}p\left( ds,dy\right) =\pi \left( dy\right) ds,q\left( ds,dy\right)
=p\left( ds,dy\right) -\pi \left( dy\right) ds.
\end{equation*}%
Now, $\zeta _{t}=\bar{\zeta}_{t}+\tilde{\zeta}_{t}$ with%
\begin{eqnarray*}
\bar{\zeta}_{t} &=&\int_{0}^{t}\int_{\left\vert y\right\vert \leq 1}\chi
_{\sigma }(y)yq(ds,dy)+\int_{0}^{t}\int_{\left\vert y\right\vert \leq
1}(1-\chi _{\sigma }(y))yp(ds,dy), \\
\tilde{\zeta}_{t} &=&\int_{0}^{t}\int_{\left\vert y\right\vert >1}\chi
_{\sigma }(y)yq(ds,dy)+\int_{0}^{t}\int_{\left\vert y\right\vert >1}(1-\chi
_{\sigma }(y))yp(ds,dy),
\end{eqnarray*}

$t\geq 0$.

\emph{Case 1: }$\sigma \in \left( 0,1\right) $. In this case (\ref{fa5})
holds with $\alpha _{1},\alpha _{2}\in (0,1]$. Then 
\begin{equation*}
\left\vert \bar{\zeta}_{t}\right\vert ^{\alpha _{1}}=\sum_{s\leq t}\left[
\left\vert \bar{\zeta}_{s-}+\Delta \bar{\zeta}_{s}\right\vert ^{\alpha
_{1}}-\left\vert \bar{\zeta}_{s-}\right\vert ^{\alpha _{1}}\right] \leq
\sum_{s\leq t}\left\vert \Delta \bar{\zeta}_{s}\right\vert ^{\alpha _{1}},
\end{equation*}%
and%
\begin{equation*}
\mathbf{E}\left\vert \bar{\zeta}_{t}\right\vert ^{\alpha _{1}}\leq
t\int_{\left\vert y\right\vert \leq 1}\left\vert y\right\vert ^{\alpha
_{1}}\pi \left( dy\right) \leq Nt.
\end{equation*}%
Similarly, $\mathbf{E}\left\vert \tilde{\zeta}_{t}\right\vert ^{\alpha
_{2}}\leq Nt.$

\emph{Case 2: }$\sigma \in \left( 1,2\right) $.\ In this case, $\alpha
_{1},\alpha _{2}\in (1,2]$. Then%
\begin{equation*}
\mathbf{E[}\bar{\zeta}_{t}^{2}]=\int_{\left\vert y\right\vert \leq
1}\left\vert y\right\vert ^{2}\pi \left( dy\right) t\leq Nt,
\end{equation*}%
and%
\begin{equation*}
\mathbf{E[}\tilde{\zeta}_{t}]\leq 2t\int_{\left\vert y\right\vert
>1}\left\vert y\right\vert ^{\alpha _{2}}\pi \left( dy\right) \leq Nt
\end{equation*}

\emph{Case 3: }$\sigma =1$. In this case, $\alpha _{1}\in (1,2]$\emph{\ and }%
$\alpha _{2}\in \lbrack 0,1)$. Similarly as above, we find that%
\begin{eqnarray*}
\mathbf{E[}\bar{\zeta}_{t}^{2}] &=&t\int_{\left\vert y\right\vert \leq
1}\left\vert y\right\vert ^{2}\pi \left( dy\right) \leq Nt, \\
\mathbf{E[}\tilde{\zeta}_{t}^{\alpha _{2}}] &\leq &Nt.
\end{eqnarray*}%
The statement is proved.
\end{proof}

Let $\pi \in \mathfrak{A}^{\sigma }$ and $p\left( dt,dy\right) $ be a
Poisson point measure on $[0,\infty )\times \mathbf{R}_{0}^{d}$ such that $%
\mathbf{E}p\left( dt,dy\right) =\pi \left( dy\right) dt.$ Let $%
q(dt,dy)=p\left( dt,dy\right) -\pi \left( dy\right) dt$. We associate to $%
L^{\pi }$ the stochastic process with independent increments 
\begin{equation}
Z_{t}=Z_{t}^{\pi }=\int_{0}^{t}\int \chi _{\sigma
}(y)yq(ds,dy)+\int_{0}^{t}\int (1-\chi _{\sigma }(y))yp(ds,dy),t\geq 0.
\label{f10}
\end{equation}%
By Ito formula, 
\begin{equation}
\mathbf{E}e^{i2\pi \xi \cdot Z_{t}^{\pi }}=\exp \left\{ \psi ^{\pi }\left(
\xi \right) t\right\} ,t\geq 0,\xi \in \mathbf{R}^{d},  \label{f12}
\end{equation}%
where%
\begin{equation*}
\psi ^{\pi }(\xi ):=\int \left[ \exp (i2\pi \xi \cdot y)-1-i2\pi y\cdot \xi
\chi _{\sigma }\left( y\right) \right] \pi (dy).
\end{equation*}

Let $\kappa \left( R\right) ,R>0,$ be a scaling function, $Z_{t}=Z_{t}^{R}$
be the stochastic process with independent increments associated with $%
\tilde{\pi}_{R}=\kappa \left( R\right) \pi _{R}$, i.e., 
\begin{equation*}
\mathbf{E}e^{i2\pi \xi \cdot Z_{t}^{R}}=\exp \left\{ \psi ^{\tilde{\pi}%
_{R}}\left( \xi \right) t\right\}
\end{equation*}%
with 
\begin{equation*}
\psi ^{\tilde{\pi}_{R}}\left( \xi \right) =\int \left[ e^{i2\pi \xi \cdot
y}-1-i2\pi \chi _{\sigma }\left( y\right) y\cdot \xi \right] d\tilde{\pi}%
_{R},\xi \in \mathbf{R}^{d}.
\end{equation*}%
Note $Z_{t}^{R}$ and $R^{-1}Z_{\kappa \left( R\right) t}^{\pi },t>0,$ have
the same distribution.

\begin{lemma}
\label{led}Let $\pi \in \mathfrak{A}^{\sigma },\kappa $ be a scaling
function with scaling factor $l$. Assume 
\begin{equation*}
\tilde{\pi}_{R}\left( dy\right) =\kappa \left( R\right) \pi _{R}\left(
dy\right) \geq 1_{\left\{ \left\vert y\right\vert \leq 1\right\} }\mu
^{0}\left( dy\right)
\end{equation*}%
with $\mu ^{0}$ satisfying the assumptions of Lemma \ref{l1} (in particular,
(\ref{b}) with $n\geq 0$ and the constant $K_{0}$), and Let 
\begin{equation*}
\psi _{0}\left( \xi \right) =\int \left[ e^{-2\pi \xi \cdot y}-1-\chi
_{\sigma }\left( y\right) \xi \cdot y\right] \mu ^{0}\left( dy\right) ,\xi
\in \mathbf{R}^{d}.
\end{equation*}

a) For each $t>0,R>0,$ we have $Z_{t}^{R}=\eta _{t}+\eta _{t}^{\prime }$ (in
distribution)$,$ $\eta _{t}$ and $\tilde{\eta}_{t}$ are independent with%
\begin{equation}
\mathbf{E}e^{i2\pi \xi \cdot \eta _{t}}=\exp \{\psi _{0}\left( \xi \gamma
\left( t\right) \right) \},\xi \in \mathbf{R}^{d},  \label{a}
\end{equation}%
and $\mu _{\gamma \left( t\right) ^{-1}}^{0}\leq t\tilde{\pi}_{R}$, where $%
\gamma \left( t\right) =l^{-1}\left( t\right) =\inf \left( s:l\left(
s\right) \geq t\right) .$ Moreover, $\eta _{t}=\gamma \left( t\right) \eta $
(in distribution), where $\eta $ is a r.v. in Lemma \ref{l1}.

b)For every $t>0,R>0,\,$the process $Z_{t}^{R}$ (equivalently $%
R^{-1}Z_{\kappa \left( R\right) t}^{\pi }$) has a bounded continuous
probability density function 
\begin{equation*}
p^{R}\left( t,x\right) =\gamma \left( t\right) ^{-d}\int p_{0}\left( \frac{%
x-y}{\gamma \left( t\right) }\right) P_{t,R}\left( dy\right) ,x\in \mathbf{R}%
^{d},
\end{equation*}%
where $P_{t,R}\left( dy\right) $ is the distribution measure of $\eta
_{t}^{\prime }$ on $\mathbf{R}^{d}$ and $p_{0}$ is pdf of $\eta $. Moreover, 
$p^{R}\left( t,x\right) $ has $n$ bounded continuous derivatives such that
for any multiindex $\left\vert \beta \right\vert \leq n,$ 
\begin{eqnarray*}
\int \left\vert \partial ^{\beta }p^{R}\left( t,x\right) \right\vert dx
&\leq &\gamma (t)^{-\left\vert \beta \right\vert }\int \left\vert \partial
^{\beta }p_{0}\left( x\right) \right\vert dx, \\
\sup_{x\in \mathbf{R}^{d}}\left\vert \partial ^{\beta }p^{R}\left(
t,x\right) \right\vert &\leq &\gamma \left( t\right) ^{-d-\left\vert \beta
\right\vert }\sup_{x}\left\vert \partial ^{\beta }p_{0}\left( x\right)
\right\vert ,
\end{eqnarray*}%
and for any $\alpha \in \left( 0,1\right) $ such that $\left\vert \beta
\right\vert +\alpha <n$%
\begin{equation}
\int \left\vert \partial ^{\alpha }\partial ^{\beta }p^{R}\left( t,x\right)
\right\vert dx\leq \gamma (t)^{-\left\vert \beta \right\vert -\alpha }\int
\left\vert \partial ^{\alpha }\partial ^{\beta }p_{0}\left( x\right)
\right\vert dx.  \label{d11}
\end{equation}

c) Assume, in addition, that there exist $\alpha _{1}$\ and $\alpha _{2}$\
and a constant $N>0$ such that%
\begin{equation*}
\int_{\left\vert z\right\vert \leq 1}\left\vert z\right\vert ^{\alpha _{1}}%
\tilde{\pi}_{R}(dz)+\int_{\left\vert z\right\vert >1}\left\vert z\right\vert
^{\alpha _{2}}\tilde{\pi}_{R}(dz)\leq N\text{ }\forall R>0,
\end{equation*}%
where $\alpha _{1},\alpha _{2}\in (0,1]\text{ if }\sigma \in (0,1)\text{; }%
\alpha _{1},\alpha _{2}\in (1,2]\text{ if }\sigma \in (1,2)$; $\alpha
_{1}\in (1,2]$\ and $\alpha _{2}\in \lbrack 0,1)$\ if $\sigma =1$. Then for
each $a>0$ there is $C=C\left( d,a,N,K_{0},n\right) $ such that for any
multiindex $\left\vert \beta \right\vert \leq n,R>0,t>0,$%
\begin{equation*}
\int_{\left\vert x\right\vert >a}\left\vert \partial ^{\beta }p^{R}\left(
t,x\right) \right\vert dx\leq C\left( \gamma \left( t\right) ^{2-\left\vert
\beta \right\vert }+t\gamma \left( t\right) ^{-\left\vert \beta \right\vert
}\right) .
\end{equation*}
\end{lemma}

\begin{proof}
a) Let $R>0,t>0$. Since $l\left( l^{-1}\left( t\right) \right) =t$, we have $%
\kappa (R)t\geq \kappa \left( Rl^{-1}\left( t\right) \right) =\kappa \left(
R\gamma \left( t\right) \right) $. Hence 
\begin{equation}
\tilde{\pi}_{R}t\geq \kappa \left( R\gamma \left( t\right) \right) \pi
_{R\gamma \left( t\right) /\gamma \left( t\right) }\geq \mu _{\gamma \left(
t\right) ^{-1}}^{0},  \label{d50}
\end{equation}%
and $\mu _{\gamma \left( t\right) ^{-1}}^{0}\left( dy\right) =\mu ^{0}\left(
\gamma \left( t\right) ^{-1}dy\right) $ is the Levy measure of a random
variable, denoted $\eta _{t},$ such that (\ref{a}) holds. Let 
\begin{equation*}
\psi ^{t\tilde{\pi}_{R}}\left( \xi \right) =\psi _{0}\left( \xi \gamma
\left( t\right) \right) +\psi ^{\prime }\left( \xi \right) ,\xi \in \mathbf{R%
}^{d}.
\end{equation*}%
The inequality (\ref{d50}) implies that (see e.g. \cite{fa}) $\psi ^{\prime
}=\psi ^{\Pi _{t}}$ with $\Pi _{t}=\tilde{\pi}_{R}t-\mu _{\gamma \left(
t\right) ^{-1}}^{0}$ and\ $\exp \{\psi ^{\prime }\left( \xi \right) \}$ is
characteristic function of a random variable $\eta _{t}^{\prime }$
independent of $\eta _{t}$. Obviously the distribution of $Z_{t}^{R}$
coincides with the distribution of the sum $\eta _{t}+\eta _{t}^{\prime }$.
If $\eta $ is a r.v. with characteristic function (\ref{a0}), then $\eta
_{t}=\gamma \left( t\right) \eta $ in distribution.

b) First we prove the existence of the probability density function of $\eta 
$ whose characteristic function is $\exp \left\{ \psi _{0}\left( \xi \right)
\right\} $. Note that $\phi _{0}\left( \xi \right) =\func{Re}\psi _{0}\left(
\xi \right) ,\xi \in \mathbf{R}^{d}.$ Let $t>0$. By part a), $%
Z_{t}^{R}=\gamma \left( t\right) \eta +\eta _{t}^{\prime }$ (in distribution)%
$,$ $\eta $ and $\eta _{t}^{\prime }$ are independent. The pdf of $\eta
\gamma \left( t\right) $ is 
\begin{equation}
p_{0}\left( t,x\right) =\gamma \left( t\right) ^{-d}p_{0}\left( x/\gamma
\left( t\right) \right) ,x\in \mathbf{R}^{d}.  \label{d70}
\end{equation}%
Let $P_{t,R}\left( dy\right) $ be the distribution measure of $\eta
_{t}^{\prime }$ on $\mathbf{R}^{d}$. Since $\eta \gamma \left( t\right) $
and $\eta _{t}^{\prime }$ are independent, $Z_{t}^{R}$ has a density 
\begin{equation*}
p^{R}\left( t,x\right) =\int p_{0}\left( t,x-y\right) P_{t,R}\left(
dy\right) ,x\in \mathbf{R}^{d}.
\end{equation*}%
According to (\ref{d70}) (see (\ref{d4}) as well) , for any $\left\vert
\beta \right\vert \leq n,$ 
\begin{equation*}
\partial ^{\beta }p^{R}\left( t,x\right) =\int \partial ^{\beta }p_{0}\left(
t,x-y\right) P_{t,R}\left( dy\right) ,
\end{equation*}%
and, according to Corollary \ref{cor1}(ii), 
\begin{eqnarray}
\sup_{x,R}\left\vert \partial ^{\beta }p^{R}\left( t,x\right) \right\vert
&\leq &\gamma \left( t\right) ^{-d-\left\vert \beta \right\vert
}\sup_{x}\left\vert \partial ^{\beta }p_{0}\left( x\right) \right\vert
<\infty ,  \label{d5} \\
\int \left\vert \partial ^{\beta }p^{R}\left( t,x\right) \right\vert dx
&\leq &\gamma \left( t\right) ^{-\left\vert \beta \right\vert }\int
\left\vert \partial ^{\beta }p_{0}\left( x\right) \right\vert dx<\infty . 
\notag
\end{eqnarray}%
Similarly, see Corollary \ref{cor1}, (\ref{d11}) follows.

c) Let $a>0,\left\vert \beta \right\vert \leq n$.

Then 
\begin{eqnarray*}
\int_{\left\vert x\right\vert >a}\left\vert \partial ^{\beta }p^{R}\left(
t,x\right) \right\vert dx &=&\gamma \left( t\right) ^{-d-\left\vert \beta
\right\vert }\int_{\left\vert x\right\vert >a}\left\vert \int (\partial
^{\beta }p_{0})\left( \frac{x-y}{\gamma \left( t\right) }\right)
P_{t,R}\left( dy\right) \right\vert dx \\
&\leq &\int \int_{\left\vert x-y\right\vert >a/2}...+\int \int_{\left\vert
y\right\vert >a/2}... \\
&\leq &C[\gamma \left( t\right) ^{2-\left\vert \beta \right\vert }\int
\left\vert x\right\vert ^{2}\left\vert \partial ^{\beta }p_{0}\left(
x\right) \right\vert dx+t\gamma \left( t\right) ^{-\left\vert \beta
\right\vert }\left\vert \partial ^{\beta }p_{0}\right\vert _{L_{1}}],
\end{eqnarray*}%
because by Lemma \ref{cor10} there is $C=C(N,a)$ such that%
\begin{equation*}
P_{t,R}\left( \left\vert y\right\vert >a/2\right) \leq Ct.
\end{equation*}
\end{proof}

We will need some estimates involving the operators $L^{\pi }$.

\begin{lemma}
\label{ld1}Let $\pi _{0}\in \mathfrak{A}^{\sigma },\kappa $ be a scaling
function with scaling factor $l,$ and \textbf{D}$\left( \kappa ,l\right) $%
(i)-(ii) hold for $\pi _{0}$. Let $\pi \in \mathfrak{A}^{\sigma }$ be such
that%
\begin{equation*}
\int_{\left\vert z\right\vert \leq 1}\left\vert z\right\vert ^{\alpha _{1}}%
\tilde{\pi}_{R}(dz)+\int_{\left\vert z\right\vert >1}\left\vert z\right\vert
^{\alpha _{2}}\tilde{\pi}_{R}(dz)\leq N\text{ }\forall R>0,
\end{equation*}%
for some $N_{2}>0$, \emph{where }$\tilde{\pi}_{R}\left( dy\right) =\kappa
\left( R\right) \pi _{R}\left( dy\right) $, and $\alpha _{1},\alpha _{2}$
are exponents in assumption \textbf{D}$\left( \kappa ,l\right) $ for $\pi
_{0}$. Let $R>0$ and $p^{R}\left( t,x\right) ,x\in \mathbf{R}^{d},$ be pdf
of $R^{-1}Z_{\kappa \left( R\right) t}^{\pi _{0}}$ (see Lemma \ref{led}), $%
\gamma \left( t\right) =l^{-1}\left( t\right) ,t>0$, and $L^{0,R}$ be the
operator corresponding to Levy measure $\kappa \left( R\right) \pi
_{0}\left( Rdy\right) .$ Then there is $C=C\left( d,N_{0},N\right) $ such
that

a) 
\begin{eqnarray}
&&\int_{\left\vert x\right\vert >2}\left\vert L^{\tilde{\pi}_{R}}p^{R}\left(
t,x\right) \right\vert dx  \label{fd1} \\
&\leq &C\left[ 1+1_{\sigma \in \left( 1,2\right) }\gamma \left( t\right)
^{-1}+\gamma \left( t\right) ^{-\alpha _{1}}\left( \gamma \left( t\right)
^{2}+t\right) \right] ,  \notag
\end{eqnarray}%
for all $t>0;$

b) denoting $I_{2}=\left\{ t>0:\gamma \left( t\right) >1\right\} $,%
\begin{equation}
\left\vert L^{\tilde{\pi}_{R}}\nabla p^{R}\left( t,\cdot \right) \right\vert
_{L_{1}}\leq C\gamma \left( t\right) ^{-(1+\alpha _{2})},t\in I_{2}
\label{fd20}
\end{equation}

c)%
\begin{equation}
\left\vert L^{\tilde{\pi}_{R}}L^{0,R}p^{R}\left( t,\cdot \right) \right\vert
_{L_{1}}\leq C\gamma \left( t\right) ^{-2\alpha _{2}},t\in I_{2}.
\label{fd30}
\end{equation}
\end{lemma}

\begin{proof}
a) For any $t>0,$%
\begin{equation*}
\int_{\left\vert x\right\vert >2}\left\vert L^{\tilde{\pi}_{R}}p^{R}\left(
t,x\right) \right\vert dx\leq \int \int_{\left\vert x\right\vert
>2}\left\vert \nabla _{z}^{\sigma }p^{R}\left( t,x\right) \right\vert dx\pi
\left( dz\right) .
\end{equation*}%
Now, for $\left\vert z\right\vert \leq 1,$ by Lemmas \ref{le3} and \ref{led}%
, 
\begin{eqnarray*}
&&\int_{\left\vert x\right\vert >2}\left\vert \nabla _{z}^{\sigma
}p^{R}\left( t,x\right) \right\vert dx \\
&\leq &C\left\vert z\right\vert ^{\alpha _{1}}\left( \int_{\left\vert
x\right\vert >1}\left\vert p^{R}\left( t,x\right) \right\vert dx\right)
^{1-\alpha _{1}}\left( \int_{\left\vert x\right\vert >1}\left\vert \nabla
p^{R}\left( t,x\right) \right\vert dx\right) ^{\alpha _{1}} \\
&\leq &C\left\vert z\right\vert ^{\alpha _{1}}\left[ \gamma \left( t\right)
^{2}+t\right] \gamma \left( t\right) ^{-\alpha _{1}}\text{ if }\sigma \in
\left( 0,1\right) ;
\end{eqnarray*}%
and by Corollary \ref{c2} and Lemma \ref{led},%
\begin{eqnarray*}
&&\int_{\left\vert x\right\vert >2}\left\vert \nabla _{z}^{\sigma
}p^{R}\left( t,x\right) \right\vert dx \\
&\leq &C\left\vert z\right\vert ^{\alpha _{1}}\left( \int_{\left\vert
x\right\vert >1}\left\vert \nabla p^{R}\left( t,x\right) \right\vert
dx\right) ^{2-\alpha _{1}}\left( \int_{\left\vert x\right\vert >1}\left\vert
D^{2}p^{R}\left( t,x\right) \right\vert dx\right) ^{\alpha _{1}-1} \\
&\leq &C\left\vert z\right\vert ^{\alpha _{1}}[t+\gamma \left( t\right)
^{2}]\gamma \left( t\right) ^{\alpha _{1}}\text{ if }\sigma \in \lbrack 1,2).
\end{eqnarray*}%
For $\left\vert z\right\vert >1,$%
\begin{equation*}
\int_{\left\vert x\right\vert >2}\left\vert \nabla _{z}^{\sigma }p^{R}\left(
t,x\right) \right\vert dx\leq 2\text{ if }\sigma \in (0,1],
\end{equation*}%
and%
\begin{equation*}
\int_{\left\vert x\right\vert >2}\left\vert \nabla _{z}^{\sigma }p^{R}\left(
t,x\right) \right\vert dx\leq C(1+\left\vert z\right\vert \int_{\left\vert
x\right\vert >2}\left\vert \nabla p^{R}(t,x)\right\vert dx)\text{ if }\sigma
\in (1,2).
\end{equation*}%
Hence by Lemma \ref{led} c),%
\begin{equation*}
\int \int_{\left\vert x\right\vert >2}\left\vert \nabla _{z}^{\sigma
}p^{R}\left( t,x\right) \right\vert dx\pi \left( dz\right) \leq C\left[
1+\gamma \left( t\right) ^{-\alpha _{1}}\left( \gamma \left( t\right)
^{2}+t\right) \right]
\end{equation*}%
if $\sigma \in (0,1]$, and%
\begin{eqnarray*}
\int \int_{\left\vert x\right\vert >2}\left\vert \nabla _{z}^{\sigma
}p^{\ast R}\left( t,x\right) \right\vert dx\pi \left( dz\right) &\leq
&C\gamma \left( t\right) ^{-\alpha _{1}}\left( \gamma \left( t\right)
^{2}+t\right) \\
&&+C\left( 1+\gamma \left( t\right) ^{-1}\right)
\end{eqnarray*}%
if $\sigma \in (1,2)$.

b)\ Let $t\in I_{2}$. By Corollary \ref{cor2} and Lemma \ref{led} b),%
\begin{eqnarray*}
\left\vert L^{\tilde{\pi}_{R}}\nabla p^{R}\left( t,\cdot \right) \right\vert
_{L_{1}} &\leq &C\left( \left\vert D^{2}p^{R}\left( t,\cdot \right)
\right\vert _{L_{1}}+\left\vert \partial ^{\alpha _{2}}\nabla p^{R}\left(
t,\cdot \right) \right\vert _{L_{1}}\right) \\
&\leq &C\left[ \gamma \left( t\right) ^{-2}+\gamma \left( t\right)
^{-(1+\alpha _{2})}\right] \leq C\gamma \left( t\right) ^{-(1+\alpha _{2})}
\end{eqnarray*}%
if $\sigma \in (0,1)$. Similarly,%
\begin{equation*}
\left\vert L^{R}\nabla p^{R}\left( t,\cdot \right) \right\vert _{L_{1}}\leq
C \left[ \gamma \left( t\right) ^{-3}+\gamma \left( t\right) ^{-(1+\alpha
_{2})}\right] \leq C\gamma \left( t\right) ^{-(1+\alpha _{2})}
\end{equation*}%
if $\sigma =1$, and%
\begin{eqnarray*}
\left\vert L^{R}\nabla p^{R}\left( t,\cdot \right) \right\vert _{L_{1}}
&\leq &C\left( \left\vert D^{3}p^{R}\left( t,\cdot \right) \right\vert
_{L_{1}}+\left\vert \partial ^{\alpha _{2}-1}D^{2}p^{R}\left( t,\cdot
\right) \right\vert _{L_{1}}\right) \\
&\leq &C\left[ \gamma \left( t\right) ^{-3}+\gamma \left( t\right)
^{-(1+\alpha _{2})}\right] \leq C\gamma \left( t\right) ^{-(1+\alpha _{2})}
\end{eqnarray*}%
if $\sigma \in \left( 1,2\right) $.

c) Let $t\in I_{2}$, i.e., $\gamma \left( t\right) >1.$ By Corollary \ref%
{cor2},

\begin{eqnarray*}
&&\left\vert L^{\tilde{\pi}_{R}}L^{0,R}p^{R}\left( t,\cdot \right)
\right\vert _{L_{1}} \\
&\leq &C\left( \left\vert L^{0,R}\nabla p^{R}\left( t,\cdot \right)
\right\vert _{L_{1}}+\left\vert L^{0,R}\partial ^{\alpha _{2}}p^{R}\left(
t,\cdot \right) v\right\vert _{L_{1}}\right) \\
&\leq &C(\left\vert D^{2}p^{R}\left( t,\cdot \right) \right\vert
_{L_{1}}+\left\vert \partial ^{\alpha _{2}}\nabla p^{R}\left( t,\cdot
\right) \right\vert _{L_{1}}+\left\vert \partial ^{\alpha _{2}}\partial
^{\alpha _{2}}p^{R}\left( t,\cdot \right) v\right\vert _{L_{1}})
\end{eqnarray*}%
if $\sigma \in \left( 0,1\right) $;%
\begin{eqnarray*}
&&\left\vert L^{\tilde{\pi}_{R}}L^{0,R}p^{R}\left( t,\cdot \right)
\right\vert _{L_{1}} \\
&\leq &C\left( \left\vert L^{0,R}D^{2}p^{R}\left( t,\cdot \right)
\right\vert _{L_{1}}+\left\vert L^{0,R}\partial ^{\alpha _{2}}p^{R}\left(
t,\cdot \right) v\right\vert _{L_{1}}\right) \\
&\leq &C(\left\vert D^{4}p^{R}\left( t,\cdot \right) \right\vert
_{L_{1}}+\left\vert \partial ^{\alpha _{2}}D^{2}p^{R}\left( t,\cdot \right)
\right\vert _{L_{1}}+\left\vert \partial ^{\alpha _{2}}\partial ^{\alpha
_{2}}p^{R}\left( t,\cdot \right) v\right\vert _{L_{1}})
\end{eqnarray*}%
if $\sigma =1$ ;%
\begin{eqnarray*}
&&\left\vert L^{\tilde{\pi}_{R}}L^{0,R}p^{R}\left( t,\cdot \right)
\right\vert _{L_{1}} \\
&\leq &C\left( \left\vert L^{0,R}D^{2}p^{R}\left( t,\cdot \right)
\right\vert _{L_{1}}+\left\vert L^{0,R}\partial ^{\alpha _{2}-1}\nabla
p^{R}\left( t,\cdot \right) v\right\vert _{L_{1}}\right) \\
&\leq &C(\left\vert D^{4}p^{R}\left( t,\cdot \right) \right\vert
_{L_{1}}+\left\vert \partial ^{\alpha _{2}-1}D^{3}p^{R}\left( t,\cdot
\right) \right\vert _{L_{1}}+\left\vert \partial ^{\alpha _{2}-1}\partial
^{\alpha _{2}-1}D^{2}p^{R}\left( t,\cdot \right) v\right\vert _{L_{1}})
\end{eqnarray*}%
if $\sigma \in \left( 1,2\right) $. The estimate (\ref{fd30}) follows by
Lemma \ref{led}.
\end{proof}

\subsection{Estimates of $\protect\psi ^{\protect\pi }$}

We present now some properties of the functions $\psi ^{\pi }\left( \xi
\right) ,\xi \in \mathbf{R}^{d}$, with $\pi \in \mathfrak{A}^{\sigma }$.

\begin{lemma}
\label{le5}Let $\pi \in \mathfrak{A}^{\sigma }$ and $\kappa \left( R\right)
,R>0,$ be a scaling function, and $\tilde{\pi}_{R}\left( dy\right) =\kappa
\left( R\right) \pi \left( Rdy\right) $.

a) Assume there is $N_{1}>0$ so that 
\begin{eqnarray}
\int \left( \left\vert y\right\vert \wedge 1\right) \tilde{\pi}_{R}\left(
dy\right) &\leq &N_{2}\text{ if }\sigma \in (0,1),  \label{fa0} \\
\int \left( \left\vert y\right\vert ^{2}\wedge 1\right) \tilde{\pi}%
_{R}\left( dy\right) &\leq &N_{2}\text{ if }\sigma =1,  \notag \\
\int_{\left\vert y\right\vert \leq 1}\left\vert y\right\vert ^{2}\tilde{\pi}%
_{R}\left( dy\right) +\int_{\left\vert y\right\vert >1}\left\vert
y\right\vert \tilde{\pi}_{R}\left( dy\right) &\leq &N_{2}\text{ if }\sigma
\in \left( 1,2\right)  \notag
\end{eqnarray}%
for any $R>0.$ Then there is a constant $C_{1}$ so that for all $\xi \in 
\mathbf{R}^{d},$%
\begin{eqnarray*}
\int \left[ 1-\cos \left( 2\pi \xi y\right) \right] \pi \left( dy\right)
&\leq &C_{1}N_{2}\kappa \left( \left\vert \xi \right\vert ^{-1}\right) ^{-1},
\\
\int \left\vert \sin \left( 2\pi \xi \cdot y\right) -2\pi \chi _{\sigma
}\left( y\right) \xi \cdot y\right\vert \pi \left( dy\right) &\leq
&C_{1}N_{2}\kappa \left( \left\vert \xi \right\vert ^{-1}\right) ^{-1},
\end{eqnarray*}%
assuming $\kappa \left( \left\vert \xi \right\vert ^{-1}\right) ^{-1}=0$ if $%
\xi =0.$

b) Assume there is a $n_{1}>0$ such that%
\begin{equation}
\int_{\left\vert y\right\vert \leq 1}\left\vert \xi \cdot y\right\vert ^{2}%
\tilde{\pi}_{R}\left( dy\right) \geq n_{1},  \label{fa00}
\end{equation}%
for all $R>0$ and $\xi \in S_{d-1}=\left\{ \xi \in \mathbf{R}^{d}:\left\vert
\xi \right\vert =1\right\} .$ Then there is a constant $c_{2}=c_{2}\left(
l\right) >0$ such that%
\begin{equation*}
\int \left[ 1-\cos \left( 2\pi \xi y\right) \right] \pi \left( dy\right)
\geq c_{2}n_{1}\kappa \left( \left\vert \xi \right\vert ^{-1}\right) ^{-1}
\end{equation*}%
for all $\xi \in \mathbf{R}^{d},$ assuming $\kappa \left( \left\vert \xi
\right\vert ^{-1}\right) ^{-1}=0$ if $\xi =0.$
\end{lemma}

\begin{proof}
The following simple trigonometric estimates hold:%
\begin{eqnarray}
\left\vert \sin x-x\right\vert &\leq &\frac{\left\vert x\right\vert ^{3}}{6}%
,1-\cos x\leq \frac{1}{2}x^{2},x\in \mathbf{R,}  \label{fa1} \\
1-\cos x &\geq &\frac{x^{2}}{\pi }\text{ if }\left\vert x\right\vert \leq
\pi /2.  \notag
\end{eqnarray}%
a)\ Let $\xi \neq 0$. Denoting $\hat{\xi}=\xi /\left\vert \xi \right\vert ,$
and using (\ref{fa1}), 
\begin{eqnarray*}
\int \left\vert 1-\cos \left( 2\pi \hat{\xi}\left\vert \xi \right\vert
y\right) \right\vert \pi \left( dy\right) &=&\kappa \left( \left\vert \xi
\right\vert ^{-1}\right) ^{-1}\int \left\vert 1-\cos \left( 2\pi \hat{\xi}%
\cdot y\right) \right\vert \tilde{\pi}_{\left\vert \xi \right\vert
^{-1}}\left( dy\right) \\
&\leq &\kappa \left( \left\vert \xi \right\vert ^{-1}\right) ^{-1}2\pi
^{2}\int \left( \left\vert y\right\vert ^{2}\wedge 1\right) \tilde{\pi}%
_{\left\vert \xi \right\vert ^{-1}}\left( dy\right) ,
\end{eqnarray*}%
and there is $C_{1}$ so that%
\begin{eqnarray*}
&&\int \left\vert \sin \left( 2\pi \xi \cdot y\right) -2\pi \chi _{\sigma
}\left( y\right) \xi \cdot y\right\vert \pi \left( dy\right) \\
&=&\kappa \left( \left\vert \xi \right\vert ^{-1}\right) ^{-1}\int
\left\vert \sin \left( 2\pi \hat{\xi}\cdot y\right) -2\pi \chi _{\sigma
}\left( y\right) \hat{\xi}\cdot y\right\vert \tilde{\pi}_{\left\vert \xi
\right\vert ^{-1}}\left( dy\right) \\
&\leq &C_{1}N_{2}\kappa \left( \left\vert \xi \right\vert ^{-1}\right) ^{-1}
\end{eqnarray*}%
for all $\xi \in \mathbf{R}^{d}$.

b) By (\ref{fa1}), for all $\xi \in \mathbf{R}^{d},$ 
\begin{eqnarray*}
&&\int [1-\cos \left( 2\pi \xi \cdot y\right) ]\pi \left( dy\right) \\
&=&\int [1-\cos \left( 2\pi \hat{\xi}\cdot y\right) ]\pi _{\left\vert \xi
\right\vert ^{-1}}\left( dy\right) \geq \int_{\left\vert y\right\vert \leq 
\frac{1}{4}}4\pi \left\vert \hat{\xi}\cdot y\right\vert ^{2}\pi _{\left\vert
\xi \right\vert ^{-1}}\left( dy\right) \\
&=&4^{-1}\int_{\left\vert 4y\right\vert \leq 1}\pi \left\vert \hat{\xi}\cdot
4y\right\vert ^{2}\pi _{\left\vert \xi \right\vert ^{-1}}\left( dy\right)
=4^{-1}\kappa \left( \left\vert 4\xi \right\vert ^{-1}\right)
^{-1}\int_{\left\vert y\right\vert \leq 1}\pi \left\vert \hat{\xi}\cdot
y\right\vert ^{2}\tilde{\pi}_{\left\vert 4\xi \right\vert ^{-1}}\left(
dy\right) \\
&\geq &n_{1}4^{-1}\pi \kappa \left( \left\vert \xi \right\vert ^{-1}\right)
^{-1}\frac{\kappa \left( \left\vert \xi \right\vert ^{-1}\right) }{\kappa
\left( \left\vert 4\xi \right\vert ^{-1}\right) }.
\end{eqnarray*}%
Let $l$ be a scaling factor of $\kappa $, i.e., $\kappa \left( \varepsilon
R\right) \leq l\left( \varepsilon \right) \kappa \left( R\right)
,\varepsilon ,R>0$. Then 
\begin{equation*}
\frac{\kappa \left( \left\vert \xi \right\vert ^{-1}\right) }{\kappa \left(
\left\vert 4\xi \right\vert ^{-1}\right) }\geq \frac{1}{l\left(
4^{-1}\right) }>0\text{ }\forall \xi \in \mathbf{R}^{d}.
\end{equation*}%
The claim follows.
\end{proof}

For $\pi \in \mathfrak{A}^{\sigma },$ let 
\begin{eqnarray*}
\tilde{\psi}\left( \xi \right) &=&\tilde{\psi}^{\pi }\left( \xi \right)
=\int \left[ \cos \left( 2\pi \xi y\right) -1\right] \pi \left( dy\right) =%
\func{Re}\psi ^{\pi }\left( \xi \right) , \\
\varphi \left( \xi \right) &=&\varphi ^{\pi }\left( \xi \right) =\int \left[
\sin \left( 2\pi \xi \cdot y\right) -2\pi \chi _{\sigma }\left( y\right) \xi
\cdot y\right] \pi \left( dy\right) =\func{Im}\psi ^{\pi }\left( \xi \right)
,
\end{eqnarray*}%
$\xi \in \mathbf{R}^{d}$. An obvious consequence of Lemma \ref{le5} (note $%
\psi ^{\pi ^{\ast }}\left( \xi \right) =\psi ^{\pi }\left( -\xi \right) ,\xi
\in \mathbf{R}^{d}$) is the following

\begin{corollary}
\label{c1}Let $\kappa $ be a scaling function with scaling factor $l$ and
both assumptions, (\ref{fa0}) and (\ref{fa00}), of Lemma \ref{le5} hold for $%
\pi \in \mathfrak{A}^{\sigma }$. Then there is constant $%
c=c(n_{1},N_{1},l)>0 $ such that%
\begin{equation*}
c\left\vert \psi ^{\pi }\left( \xi \right) \right\vert \leq \left\vert 
\tilde{\psi}^{\pi }\left( \xi \right) \right\vert \leq \left\vert \psi ^{\pi
}\left( \xi \right) \right\vert ,\xi \in \mathbf{R}^{d},
\end{equation*}%
and $\left\vert \varphi ^{\pi }\left( \xi \right) \right\vert \leq
c^{-1}\left\vert \tilde{\psi}^{\pi }\left( \xi \right) \right\vert ,\xi \in 
\mathbf{R}^{d}.$

Note that it implies%
\begin{equation*}
c\left\vert \psi ^{\pi ^{\ast }}\left( \xi \right) \right\vert \leq
\left\vert \tilde{\psi}^{\pi }\left( \xi \right) \right\vert =\left\vert 
\tilde{\psi}^{\pi }\left( -\xi \right) \right\vert \leq \left\vert \psi
^{\pi ^{\ast }}\left( \xi \right) \right\vert ,\xi \in \mathbf{R}^{d},
\end{equation*}%
where $\pi ^{\ast }\left( dy\right) =\pi \left( -dy\right) .$
\end{corollary}

\section{Proof of the main results}

In this section we prove the main results in three steps. First we prove the
existence and uniqueness of classical solutions for smooth input functions.
Then we derive $\left\vert u\right\vert _{L^{\pi _{0}}}$-norm estimates with
constants independent of the regularity of the input function. Finally,
continuity estimate of $L^{\pi }$ with respect to $\left\vert u\right\vert
_{L^{\pi _{0}}}$-norm allows to pass to the limit and derive the results for
the input function $f\in L_{p}$.

\subsection{Existence and uniqueness for smooth input functions}

For $E=[0,T]\times \mathbf{R}^{d}$, we denote by $\tilde{C}^{\infty }(E)$
the space of all measurable functions $f$ on $E$ such that for any
multiindex $\gamma \in \mathbf{N}_{0}^{d}$ and for all $1\leq p<\infty $ 
\begin{equation*}
\sup_{\left( t,x\right) \in E}\left\vert D^{\gamma }f\left( t,x\right)
\right\vert +\sup_{t\in \lbrack 0,T]}\left\vert D^{\gamma }f(t,\cdot
)\right\vert _{L_{p}\left( R^{d}\right) }<\infty .
\end{equation*}%
Similarly, let $\tilde{C}^{\infty }(\mathbf{R}^{d})$ be the space of all
measurable functions $f$ on $\mathbf{R}^{d}$ such that for any multiindex $%
\gamma \in \mathbf{N}_{0}^{d}$ and for all $1\leq p<\infty $%
\begin{equation*}
\sup_{x\in \mathbf{R}^{d}}\left\vert D^{\gamma }f\left( x\right) \right\vert
+\left\vert D^{\gamma }f\right\vert _{L_{p}\left( R^{d}\right) }<\infty .
\end{equation*}

Next we suppose that $f\in \tilde{C}^{\infty }(E)$ and derive some estimates
for the solution.

\begin{lemma}
\label{le0}Let $f\in \tilde{C}^{\infty }(E)$ then there is unique $u\in 
\tilde{C}^{\infty }\left( E\right) $ solving (\ref{1'}). Moreover, 
\begin{equation}
u(t,x)=\int_{0}^{t}e^{-\lambda (t-s)}\mathbf{E}f\left( s,x+Z_{t-s}^{\pi
}\right) ds,\left( t,x\right) \in E,  \label{h4}
\end{equation}%
and for $p\in \lbrack 1,\infty ]$ and any mutiindex $\gamma \in \mathbf{N}%
_{0}^{d},$%
\begin{eqnarray}
\left\vert D^{\gamma }u\right\vert _{L_{p}\left( E\right) } &\leq &\rho
_{\lambda }\left\vert D^{\gamma }f\right\vert _{L_{p}\left( E\right) },
\label{h5} \\
\left\vert D^{\gamma }u\left( t\right) \right\vert _{L_{p}\left( \mathbf{R}%
^{d}\right) } &\leq &\int_{0}^{t}\left\vert D^{\gamma }f\left( s\right)
\right\vert _{L_{p}\left( \mathbf{R}^{d}\right) }ds,t\geq 0.  \label{h40}
\end{eqnarray}%
where $\rho _{\lambda }=\left( 1/\lambda \right) \wedge T$.
\end{lemma}

\begin{proof}
Denote $Z_{t}=Z_{t}^{\pi },t\geq 0.$ \emph{Uniqueness. }Let $u_{1},u_{2}\in 
\tilde{C}^{\infty }\left( E\right) $ solve (\ref{1'}) and $u=u_{1}-u_{2}$.
Then $u$ solve (\ref{1'}) with $f=0$. Let $\left( t,x\right) \in E$. By Ito
formula for $e^{\lambda \left( t-s\right) }u\left( t-s,x+Z_{s}\right) ,0\leq
s\leq t$, we have%
\begin{equation*}
-u\left( t,x\right) =\mathbf{E}\int_{0}^{t}e^{\lambda \left( t-s\right)
}[-\partial _{t}+L^{\pi }-\lambda ]u\left( t-s,x+Z_{s}\right) ds=0.
\end{equation*}%
Hence $u\left( t,x\right) =0$ for all $\left( t,x\right) \in E$.

\emph{Existence. }Let $f\in \tilde{C}^{\infty }\left( E\right) $. Set 
\begin{equation*}
u(t,x)=\int_{0}^{t}e^{-\lambda (t-s)}\mathbf{E}f\left( s,x+Z_{t-s}^{\pi
}\right) ds,\left( t,x\right) \in E.
\end{equation*}

Then%
\begin{equation*}
D^{\gamma }u(t,x)=\int_{0}^{t}e^{-\lambda (t-s)}\mathbf{E}D^{\gamma }f\left(
s,x+Z_{t-s}^{\pi }\right) ds,\left( t,x\right) \in E,
\end{equation*}%
and (\ref{h40}) follows.

By H\"{o}lders inequality for $\lambda >0,$ 
\begin{equation*}
|D^{\gamma }u|_{L_{p}(T)}^{p}\leq \lambda
^{-p}\int_{0}^{T}\int_{0}^{t}\lambda e^{-\lambda (t-s)}\mathbf{|}D^{\gamma
}f\left( s,\cdot \right) |_{L_{p}\left( \mathbf{R}^{d}\right) }^{p}dsdt\leq
\lambda ^{-p}\mathbf{|}D^{\gamma }f|_{L_{p}\left( T\right) }^{p}.
\end{equation*}%
By H\"{o}lder inequality for $\lambda \geq 0,$%
\begin{equation*}
|D^{\gamma }u|_{L_{p}\left( T\right) }^{p}\leq
\int_{0}^{T}t^{p-1}\int_{0}^{t}\mathbf{|}D^{\gamma }f\left( s,\cdot \right)
|_{L_{p}\left( \mathbf{R}^{d}\right) }^{p}dsdt\leq \frac{T^{p}}{p}\left\vert
D^{\gamma }f\right\vert _{L_{p}\left( T\right) }^{p}.
\end{equation*}

We fix $s\in \left[ 0,T\right] ,x\in \mathbf{R}^{d}$, and applying Ito
formula with $e^{-\lambda r}f(s,x+\cdot )$ and $Z_{r},\ 0\leq r\leq t-s,$ we
have 
\begin{eqnarray*}
&&e^{-\lambda \left( t-s\right) }f(s,x+Z_{t-s}) \\
&=&f(s,x)+\int_{0}^{t-s}\int_{\mathbf{R}_{0}}e^{-\lambda r}\left[
f(s,x+Z_{r}+y)-f(s,x+Z_{r-})\right] q(dr,dy) \\
&&+\int_{0}^{t-s}e^{-\lambda r}\left( L^{\pi }-\lambda \right) f\left(
s,x+Z_{r}\right) dr
\end{eqnarray*}%
Taking expectation on both sides, and integrating with respect to $s$, we
obtain by Fubini's theorem for each $\left( t,x\right) \in E,$ 
\begin{eqnarray*}
&&\int_{0}^{t}e^{-\lambda (t-s)}\mathbf{E}f(s,x+Z_{t-s})ds \\
&=&\int_{0}^{t}f(s,x)ds+\int_{0}^{t}\int_{0}^{t-s}e^{-\lambda r}\mathbf{E}%
\left[ L^{\pi }-\lambda \right] f\left( s,x+Z_{r}\right) drds \\
&=&\int_{0}^{t}f(s,x)ds+\int_{0}^{t}\int_{s}^{t}e^{-\lambda (r-s)}\mathbf{E[}%
L^{\pi }-\lambda ]f\left( s,x+Z_{r-s}\right) drds \\
&=&\int_{0}^{t}f(s,x)ds+\int_{0}^{t}\int_{0}^{r}e^{-\lambda (r-s)}\mathbf{E[}%
L^{\pi }-\lambda ]f\left( s,x+Z_{r-s}\right) dsdr.
\end{eqnarray*}%
Since for each $\left( r,x\right) \in E$%
\begin{equation*}
\int_{0}^{r}e^{-\lambda (r-s)}\mathbf{E}L^{\pi }f\left( s,x+Z_{r-s}\right)
ds=L^{\pi }u\left( r,x\right) ,
\end{equation*}%
it follows that for each $\left( t,x\right) \in E,$%
\begin{equation*}
u(t,x)=\int_{0}^{t}f(s,x)ds+\int_{0}^{t}\left[ L^{\pi }u\left( r,x\right)
-\lambda u\left( r,x\right) \right] dr.
\end{equation*}%
The statement follows.
\end{proof}

Similarly we can handle the problem (\ref{2'}).

\begin{lemma}
\label{le00}Let $f\in \tilde{C}^{\infty }(\mathbf{R}^{d})$ then there is
unique $u\in \tilde{C}^{\infty }\left( \mathbf{R}^{d}\right) $ solving (\ref%
{2'}). Moreover, 
\begin{equation}
u(x)=\int_{0}^{\infty }e^{-\lambda t}\mathbf{E}f\left( x+Z_{t}^{\pi }\right)
dt,x\in \mathbf{R}^{d},  \label{h60}
\end{equation}%
and for $p\in \left[ 1,\infty \right] $ and any mutiindex $\gamma \in 
\mathbf{N}_{0}^{d},$%
\begin{equation}
\left\vert D^{\gamma }u\right\vert _{L_{p}\left( \mathbf{R}^{d}\right) }\leq
\left( 1/\lambda \right) \left\vert D^{\gamma }f\right\vert _{L_{p}\left( 
\mathbf{R}^{d}\right) }.  \label{h6}
\end{equation}
\end{lemma}

\begin{proof}
Denote $Z_{t}=Z_{t}^{\pi },t\geq 0.$ \emph{Uniqueness. }Let $u_{1},u_{2}\in 
\tilde{C}^{\infty }\left( E\right) $ solve (\ref{1'}) and $u=u_{1}-u_{2}$.
Then $u$ solve (\ref{1'}) with $f=0$. Let $x\in \mathbf{R}^{d}$. By Ito
formula for $e^{-\lambda t}u\left( x+Z_{t}\right) ,0\leq s\leq t$, we have%
\begin{equation*}
e^{-\lambda t}\mathbf{E}u\left( x+Z_{t}\right) -u\left( x\right) =\mathbf{E}%
\int_{0}^{t}e^{-\lambda s}[L^{\pi }u\left( x+Z_{s}\right) -\lambda u\left(
x+Z_{s}\right) ]ds=0.
\end{equation*}%
Passing to the limit as $t\rightarrow \infty $ we obtain that $u\left(
x\right) =0$ for all $x\in \mathbf{R}^{d}$.

\emph{Existence. }Let $f\in \tilde{C}^{\infty }\left( \mathbf{R}^{d}\right) $%
. Set%
\begin{equation*}
u(x)=\int_{0}^{\infty }e^{-\lambda t}\mathbf{E}f\left( x+Z_{t}^{\pi }\right)
dt,x\in \mathbf{R}^{d}.
\end{equation*}%
By direct estimate using H\"{o}lder inequality, as in Lemma \ref{le0}, (\ref%
{h6}) readily follows, i.e. $u\in \tilde{C}^{\infty }\left( \mathbf{R}%
^{d}\right) $. We fix $x\in \mathbf{R}^{d}$, and applying Ito formula with $%
e^{-\lambda t}f(x+Z_{t}),\ 0\leq t,$ we have for all $t>0,x\in \mathbf{R}%
^{d},$ 
\begin{equation*}
e^{-\lambda t}\mathbf{E}f(x+Z_{t})=f(x)+\int_{0}^{t}e^{-\lambda s}\mathbf{E}%
\left( L^{\pi }-\lambda \right) f\left( s,x+Z_{s}\right) ds
\end{equation*}%
Passing to the limit as $t\rightarrow \infty $ we have%
\begin{equation*}
L^{\pi }u(x)-\lambda u\left( x\right) +f\left( x\right) =0,x\in \mathbf{R}%
^{d}.
\end{equation*}%
The statement follows.
\end{proof}

An obvious consequence of Lemma \ref{le00} is

\begin{corollary}
\label{coro1}Let $f\in \tilde{C}^{\infty }(\mathbf{R}^{d})$ and 
\begin{equation*}
u(x)=\int_{0}^{\infty }e^{-\lambda t}\mathbf{E}f\left( x+Z_{t}^{\pi }\right)
dt,x\in \mathbf{R}^{d}.
\end{equation*}%
Then $u\in \tilde{C}^{\infty }(\mathbf{R}^{d})$, and%
\begin{eqnarray*}
L^{\pi }u(x) &=&L^{\pi }\int_{0}^{\infty }e^{-\lambda t}\mathbf{E}f\left(
x+Z_{t}^{\pi }\right) dt=\int_{0}^{\infty }e^{-\lambda t}\mathbf{E}L^{\pi
}f\left( x+Z_{t}^{\pi }\right) dt, \\
\left( L^{\pi }-\lambda \right) u(x) &=&(L^{\pi }-\lambda )\int_{0}^{\infty
}e^{-\lambda t}\mathbf{E}f\left( x+Z_{t}^{\pi }\right) dt \\
&=&\int_{0}^{\infty }e^{-\lambda t}\mathbf{E}(L^{\pi }-\lambda )f\left(
x+Z_{t}^{\pi }\right) dt=-f\left( x\right) ,x\in \mathbf{R}^{d}.
\end{eqnarray*}
\end{corollary}

\subsection{$L_{2}$-estimates}

We derive first some $L_{2}$-estimates independent of the regularity of $f$.
Given $\pi \in \mathfrak{A}^{\sigma }$, let 
\begin{equation*}
\tilde{\psi}=\tilde{\psi}^{\pi }=\func{Re}\psi ^{\pi }\left( \xi \right)
=\int \left[ \cos \left( 2\pi y\cdot \xi \right) -1\right] \pi \left(
dy\right) ,\xi \in \mathbf{R}^{d}.
\end{equation*}%
Let for $v\in \tilde{C}^{\infty }\left( \mathbf{R}^{d}\right) $ or $v\in 
\tilde{C}^{\infty }\left( E\right) $ we define 
\begin{equation*}
\left\vert v\right\vert _{\tilde{\psi},2}=\left\vert \mathcal{F}^{-1}\tilde{%
\psi}^{\pi }\hat{v}\right\vert _{L_{2}\left( \mathbf{R}^{d}\right) }\text{
or }\left\vert v\right\vert _{\tilde{\psi},2;E}=\left\vert \mathcal{F}^{-1}%
\tilde{\psi}^{\pi }\hat{v}\right\vert _{L_{2}\left( E\right) };
\end{equation*}%
in the case $v\in \tilde{C}^{\infty }\left( E\right) $, $\hat{v}$ denotes
Fourier transform in $x$.

\begin{lemma}
\label{le10}Let $\pi \in \mathfrak{A}^{\sigma }.$

a) Let $f\in \tilde{C}^{\infty }(E)$ and $u\in \tilde{C}^{\infty }(E)$ be
the unique solution to (\ref{1}), defined in Lemma \ref{le0}. Then 
\begin{equation}
\left\vert u\right\vert _{\tilde{\psi},2;E}\leq \left\vert f\right\vert
_{L_{2}(T)}.  \label{eq:h5}
\end{equation}

Moreover, 
\begin{equation}
\left\vert u\right\vert _{L_{2}(T)}\leq \left( \lambda ^{-1}\wedge T\right)
\left\vert f\right\vert _{L_{2}(T)},\left\vert u\left( t\right) \right\vert
_{L_{2}\left( \mathbf{R}^{d}\right) }\leq \int_{0}^{t}\left\vert f\left(
s\right) \right\vert _{L_{2}\left( \mathbf{R}^{d}\right) }ds,t\geq 0,
\label{eq:h6}
\end{equation}%
and for $t\geq 0,$%
\begin{equation}
\int_{0}^{t}\int \int \left[ u\left( s,x+y\right) -u\left( s,x\right) \right]
^{2}\pi \left( dy\right) dxds\leq \left( \int_{0}^{t}\left\vert f\left(
s\right) \right\vert _{L_{2}\left( \mathbf{R}^{d}\right) }ds\right) ^{2}.
\label{h80}
\end{equation}

b) Let $f\in \tilde{C}^{\infty }(\mathbf{R}^{d})$ and $v\in \tilde{C}%
^{\infty }(\mathbf{R}^{d})$ be the unique solution to (\ref{2'}), defined in
Lemma \ref{le00}. Then 
\begin{equation}
\left\vert v\right\vert _{\tilde{\psi},2}\leq \left\vert f\right\vert
_{L_{2}(\mathbf{R}^{d})}.
\end{equation}

Moreover, 
\begin{equation}
\left\vert v\right\vert _{L_{2}(\mathbf{R}^{d})}\leq \left( 1/\lambda
\right) \left\vert f\right\vert _{L_{2}(\mathbf{R}^{d})},  \label{h70}
\end{equation}%
and%
\begin{equation}
\int \int \left[ v\left( x+y\right) -v\left( x\right) \right] ^{2}\pi \left(
dy\right) dx\leq \frac{2}{\lambda }\left\vert f\right\vert _{L_{2}\left( 
\mathbf{R}^{d}\right) }^{2}.  \label{h71}
\end{equation}
\end{lemma}

\begin{proof}
a) Let $f\in \tilde{C}^{\infty }(E)$. Taking Fourier transform in $x$ in the
representation (\ref{h4}) we find that%
\begin{equation*}
\hat{u}\left( t,\xi \right) =\int_{0}^{t}e^{-\lambda (t-s)}\exp \left\{ \psi
^{\pi }\left( \xi \right) \left( t-s\right) \right\} \hat{f}\left( s,\xi
\right) ds,\left( t,\xi \right) \in E.
\end{equation*}%
Hence, by H\"{o}lder inequality, for any $\left( t,\xi \right) \in E,$%
\begin{equation*}
\left\vert \tilde{\psi}^{\pi }\left( \xi \right) \hat{u}\left( t,\xi \right)
\right\vert ^{2}\leq \left\vert \tilde{\psi}^{\pi }\left( \xi \right)
\right\vert \int_{0}^{t}\exp \left\{ \tilde{\psi}^{\pi }\left( \xi \right)
\left( t-s\right) \right\} \left\vert \hat{f}\left( s,\xi \right)
\right\vert ^{2}ds,
\end{equation*}%
and, by Fubini theorem, 
\begin{equation*}
\left\vert \tilde{\psi}^{\pi }\hat{u}\right\vert _{L_{2}\left( T\right)
}^{2}\leq \left\vert \hat{f}\right\vert _{L_{2}\left( T\right) }.
\end{equation*}

The inequality (\ref{eq:h6}) is derived in Lemma \ref{le0}. We derive the
remaining inequalities using chain rule and integrating. For any $\left(
t,x\right) \in E,$%
\begin{equation}
u(t,x)^{2}=2\int_{0}^{t}u(s,x)\left[ L^{\pi }u(s,x)-\lambda u(s,x)+f(s,x)%
\right] ds.  \label{g1}
\end{equation}%
Now, for any $\left( s,x\right) \in E,y\in \mathbf{R}^{d},$ 
\begin{eqnarray}
&&2u\left( s,x\right) \left[ u(s,x+y\right] -u\left( s,x\right) -\chi
_{\sigma }\left( y\right) y\cdot \nabla u\left( s,x\right) ]  \notag \\
&=&-\left[ u\left( s,x+y\right) -u\left( s,x\right) \right] ^{2}  \label{g2}
\\
&&+u\left( s,x+y\right) ^{2}-u\left( s,x\right) ^{2}-\chi _{\sigma }\left(
y\right) y\cdot \nabla \lbrack u\left( s,x\right) ^{2}].  \notag
\end{eqnarray}%
Using (\ref{g2}) and integrating both sides of (\ref{g1}) in $x,$ we have
for any $t\in \left[ 0,T\right] ,$%
\begin{eqnarray*}
&&\left\vert u\left( t\right) \right\vert _{L_{2}\left( \mathbf{R}%
^{d}\right) }^{2}+\int_{0}^{t}\int \int \left[ u\left( s,x+y\right) -u\left(
s,x\right) \right] ^{2}\pi \left( dy\right) dxds \\
&\leq &2\int_{0}^{t}\int f\left( s,x\right) u\left( s,x\right) dxds\leq
2\int_{0}^{t}\left\vert f\left( s\right) \right\vert _{L_{2}\left( \mathbf{R}%
^{d}\right) }\left\vert u\left( s\right) \right\vert _{L_{2}\left( \mathbf{R}%
^{d}\right) } \\
&\leq &\left( \int_{0}^{t}\left\vert f\left( s\right) \right\vert
_{L_{2}\left( \mathbf{R}^{d}\right) }ds\right) ^{2}
\end{eqnarray*}%
and (\ref{h80}) follows.

b) Let $f\in \tilde{C}^{\infty }(\mathbf{R}^{d})$. Taking Fourier transform
in (\ref{h60}) we find that%
\begin{equation*}
\hat{v}(\xi )=\int_{0}^{\infty }e^{-\lambda t}\exp \left\{ \psi ^{\pi
}\left( \xi \right) t\right\} \hat{f}\left( \xi \right) dt,\xi \in \mathbf{R}%
^{d}.
\end{equation*}%
Hence for any $\xi \in \mathbf{R}^{d},$%
\begin{equation*}
\left\vert \tilde{\psi}^{\pi }\left( \xi \right) \hat{v}\left( \xi \right)
\right\vert ^{2}\leq \left\vert \tilde{\psi}^{\pi }\left( \xi \right)
\right\vert \int_{0}^{\infty }\exp \left\{ \tilde{\psi}^{\pi }\left( \xi
\right) t\right\} \left\vert \hat{f}\left( \xi \right) \right\vert
^{2}dt\leq \left\vert \hat{f}\left( \xi \right) \right\vert ^{2},
\end{equation*}%
and%
\begin{equation*}
\left\vert \tilde{\psi}^{\pi }\hat{v}\right\vert _{L_{2}\left( \mathbf{R}%
^{d}\right) }^{2}\leq \left\vert \hat{f}\right\vert _{L_{2}\left( \mathbf{R}%
^{d}\right) }.
\end{equation*}

Inequality (\ref{h70}) was derived in Lemma \ref{le00}. Multiplying both
sides of (\ref{2}) by $2v$ and integrating as in the part a), we have%
\begin{eqnarray*}
&&\int \int \left[ v\left( x+y\right) -v\left( x\right) \right] ^{2}\pi
\left( dy\right) dx \\
&\leq &2\left\vert v\right\vert _{L_{2}\left( \mathbf{R}^{d}\right)
}\left\vert f\right\vert _{L_{2}\left( \mathbf{R}^{d}\right) }\leq \frac{2}{%
\lambda }\left\vert f\right\vert _{L_{2}\left( \mathbf{R}^{d}\right) }^{2}.
\end{eqnarray*}
\end{proof}

\begin{remark}
\label{rm1}For $v\in \tilde{C}^{\infty }\left( \mathbf{R}^{d}\right) ,$ by
Plancherel's equality,%
\begin{eqnarray*}
\int \int \left[ v\left( x+y\right) -v\left( x\right) \right] ^{2}\pi \left(
dy\right) dx &=&\int \int |e^{i2\pi \xi \cdot y}-1|^{2}\pi \left( dy\right) |%
\hat{v}\left( \xi \right) |^{2}d\xi \\
&=&2\int \left( -\tilde{\psi}^{\pi }\left( \xi \right) \right) \left\vert 
\hat{v}\left( \xi \right) \right\vert ^{2}d\xi \\
&=&2\left\vert \sqrt{-\tilde{\psi}}\hat{v}\right\vert _{L_{2}\left( \mathbf{R%
}^{d}\right) }^{2}=2\left\vert v\right\vert _{\left( -\tilde{\psi}\right)
^{1/2},2}^{2},
\end{eqnarray*}%
where $\left\vert v\right\vert _{\left( -\tilde{\psi}\right)
^{1/2},2}=\left\vert \mathcal{F}^{-1}\sqrt{-\tilde{\psi}}\hat{v}\right\vert
_{L_{2}\left( \mathbf{R}^{d}\right) }.$
\end{remark}

\subsection{Continuity of $L^{\protect\pi }$ and proof of Theorem \protect
\ref{thm:main}}

We show first the $L_{2}$ continuity as follows.

\begin{lemma}
\label{le2}Let $\pi ,\pi _{0}\in \mathfrak{A}^{\sigma }$. Assume there is $%
C_{0}>0$ so that 
\begin{equation*}
\left\vert \psi ^{\pi }\left( \xi \right) \right\vert \leq C_{0}\left\vert
\psi ^{\pi _{0}}\left( \xi \right) \right\vert ,\xi \in \mathbf{R}^{d}.
\end{equation*}%
Then%
\begin{equation}
\left\vert L^{\pi }\varphi \right\vert _{L_{2}(\mathbf{R}^{d})}\leq
C_{0}\left\vert L^{\pi _{0}}\varphi \right\vert _{L_{2}(\mathbf{R}%
^{d})},\varphi \in \tilde{C}^{\infty }(\mathbf{R}^{d}).
\end{equation}
\end{lemma}

\begin{proof}
Let $\varphi \in \tilde{C}^{\infty }\left( \mathbf{R}^{d}\right) $. By
Plancherel equality,%
\begin{equation*}
\left\vert L^{\pi }\varphi \right\vert _{L_{2}(\mathbf{R}^{d})}=\left\vert
\psi ^{\pi }\hat{\varphi}\right\vert _{L_{2}\left( \mathbf{R}^{d}\right)
}\leq C_{0}\left\vert \psi ^{\pi _{0}}\hat{\varphi}\right\vert _{L_{2}\left( 
\mathbf{R}^{d}\right) }=C_{0}\left\vert L^{\pi _{0}}\varphi \right\vert
_{L_{2}\left( \mathbf{R}^{d}\right) }.
\end{equation*}
\end{proof}

Let $\pi _{0}\in \mathfrak{A}^{\sigma },\kappa $ be a scaling function with
a scaling factor $l$ and \textbf{D}$\left( \kappa ,l\right) $(i)-(ii) hold
for $\pi _{0}$. Let $\gamma \left( t\right) =\inf \left( s>0:l\left(
s\right) >t\right) ,t>0.$ According to Lemma \ref{led}, for each $t>0$, the
associated Levy process $Z_{t}^{\pi _{0}}$ has a bounded probability density
function $p^{\pi _{0}}\left( t,x\right) ,x\in \mathbf{R}^{d}.$ Also, $\int
\left\vert \nabla p^{\pi _{0}}\left( t,x\right) \right\vert dx\leq C\gamma
\left( t\right) ^{-1},t>0,$ if $\sigma \in \left( 1,2\right) $.

Define for $v\in \tilde{C}^{\infty }\left( \mathbf{R}^{d}\right) ,\lambda
>0,\varepsilon \geq 0,$ 
\begin{eqnarray*}
K_{\lambda }^{\varepsilon }v\left( x\right) &=&\int_{\varepsilon }^{\infty
}e^{-\lambda t}\mathbf{E}v\left( x+Z_{t}^{\pi _{0}}\right)
dt=\int_{\varepsilon }^{\infty }\int e^{-\lambda t}v\left( x+y\right) p^{\pi
_{0}}\left( t,y\right) dydt \\
&=&\int v(x-y)\int_{\varepsilon }^{\infty }e^{-\lambda t}p^{\pi _{0}^{\ast
}}\left( t,y\right) dtdy \\
&=&\int v(y)\int_{\varepsilon }^{\infty }e^{-\lambda t}p^{\pi _{0}^{\ast
}}\left( t,x-y\right) dtdy,x\in \mathbf{R}^{d}.
\end{eqnarray*}%
Let%
\begin{eqnarray}
T_{\lambda }^{\varepsilon }v\left( x\right) &=&L^{\pi }K_{\lambda
}^{\varepsilon }v\left( x\right) =\int L^{\pi }v(x-y)\int_{\varepsilon
}^{\infty }e^{-\lambda t}p^{\pi _{0}^{\ast }}\left( t,y\right) dtdy  \notag
\\
&=&\int_{\varepsilon }^{\infty }e^{-\lambda t}\int L^{\pi }p^{\pi _{0}^{\ast
}}(t,x-y)v\left( y\right) dydt  \label{far1} \\
&=&\int m_{\lambda }^{\varepsilon }\left( x-y\right) v\left( y\right)
dy,v\in \tilde{C}^{\infty }\left( \mathbf{R}^{d}\right) ,  \notag
\end{eqnarray}%
with 
\begin{equation*}
m^{\varepsilon }\left( x\right) =m_{\lambda }^{\varepsilon }\left( x\right)
=\int_{\varepsilon }^{\infty }e^{-\lambda t}L^{\pi }p^{\pi _{0}^{\ast
}}(t,x)dt,x\in \mathbf{R}^{d}.
\end{equation*}%
Note that according to Lemma \ref{led}b) and Corollaries \ref{cor2} and \ref%
{cor1}, 
\begin{equation*}
\int \left\vert m_{\lambda }^{\varepsilon }\left( x\right) \right\vert
dx<\infty .
\end{equation*}

\begin{lemma}
\label{pro1}Let $\pi _{0}\in \mathfrak{A}^{\sigma },\kappa $ be a scaling
function with scaling factor $l,$ and \textbf{D}$\left( \kappa ,l\right) $
hold for $\pi _{0}$. Let $\pi \in \mathfrak{A}^{\sigma }$ be such that%
\begin{equation*}
\int_{\left\vert z\right\vert \leq 1}\left\vert z\right\vert ^{\alpha _{1}}%
\tilde{\pi}_{R}(dz)+\int_{\left\vert z\right\vert >1}\left\vert z\right\vert
^{\alpha _{2}}\tilde{\pi}_{R}(dz)\leq N\text{ }\forall R>0,
\end{equation*}
\emph{where }$\tilde{\pi}_{R}\left( dy\right) =\kappa \left( R\right) \pi
_{R}\left( dy\right) $, and $\alpha _{1},\alpha _{2}$ are exponents in
assumption \textbf{D}$\left( \kappa ,l\right) $ for $\pi _{0}$.

Then for each $p\in \left( 1,\infty \right) $ there is a constant $C=C\left(
d,p,\kappa ,l,N_{0},N,N_{1},c_{1}\right) $ such that%
\begin{equation*}
\left\vert T_{\lambda }^{\varepsilon }v\right\vert _{L_{p}}\leq C\left\vert
v\right\vert _{L_{p}},v\in \tilde{C}^{\infty }\left( \mathbf{R}^{d}\right) .
\end{equation*}
\end{lemma}

\begin{proof}
1. First we prove the statement for $p=2$. Observe that%
\begin{equation*}
\widehat{m_{\lambda }^{\varepsilon }}\left( \xi \right) =\psi ^{\pi }\left(
\xi \right) \int_{\varepsilon }^{\infty }\exp \left\{ \psi ^{\pi _{0}^{\ast
}}\left( \xi \right) t-\lambda t\right\} dt,\xi \in \mathbf{R}^{d}.
\end{equation*}%
Hence by Lemma \ref{le5} and Corollary \ref{c1} there is $C=C\left(
N,c_{1}\right) $ such that 
\begin{equation*}
\left\vert \widehat{m_{\lambda }^{\varepsilon }}\left( \xi \right)
\right\vert \leq CN_{2}\left\vert \psi ^{\pi _{0}^{\ast }}\left( \xi \right)
\right\vert \int_{0}^{\infty }\exp \left\{ \tilde{\psi}^{\pi _{0}^{\ast
}}\left( \xi \right) t-\lambda t\right\} dt\leq CN_{2}
\end{equation*}%
for all $\xi \in \mathbf{R}^{d},$ and 
\begin{equation*}
\left\vert T_{\lambda }^{\varepsilon }v\right\vert _{L_{2}}\leq C\left\vert
v\right\vert _{L_{2}},v\in \tilde{C}^{\infty }\left( \mathbf{R}^{d}\right) .
\end{equation*}

2. Since we already have an $L_{2}$ -estimate, according to Theorem 3 of
Chapter I in \cite{stein1}, it suffices to show that 
\begin{equation}
\int_{|x|\geq 3|s|}\left\vert m_{\lambda }^{\varepsilon }(x-s)-m_{\lambda
}^{\varepsilon }(x)\right\vert dx\leq C,\text{ }\forall s\neq 0.
\label{far2}
\end{equation}

Let $s\neq 0,R=\left\vert s\right\vert $. Changing the variable in (\ref%
{far2}), we see that we have to prove that%
\begin{equation}
R^{d}\int_{|x|\geq 3}\left\vert m_{\lambda }^{\varepsilon }(R\left( x-\hat{s}%
\right) )-m_{\lambda }^{\varepsilon }(Rx)\right\vert dx\leq C,\text{ }%
\left\vert \hat{s}\right\vert =1,R>0.  \label{far3}
\end{equation}

Let $p^{\ast R}$ be the pdf corresponding to the Levy measure $\kappa \left(
R\right) \pi _{0}^{\ast }\left( Rdy\right) $, and let $L^{R}=L^{\tilde{\pi}%
_{R}}$ with $\tilde{\pi}_{R}=\kappa \left( R\right) \pi \left( Rdy\right) $.
For $R>0$, changing the variables of integration, we have%
\begin{eqnarray*}
m_{\lambda }^{\varepsilon }(Rx) &=&\int_{\varepsilon }^{\infty }e^{-\lambda
t}(\nabla _{y}^{\sigma }p^{\pi _{0}^{\ast }})\left( t,Rx\right) \pi
^{\varepsilon }\left( dy\right) dt \\
&=&\int_{\varepsilon /\kappa \left( R\right) }^{\infty }e^{-\lambda \kappa
\left( R\right) t}(\nabla _{Ry}^{\sigma }p^{\pi _{0}^{\ast }})\left( \kappa
\left( R\right) t,Rx\right) \widetilde{\pi }_{R}\left( dy\right) dt \\
&=&\int_{\varepsilon /\kappa \left( R\right) }^{\infty }e^{-\lambda \kappa
\left( R\right) t}R^{-d}(L^{R}p^{\ast R})\left( t,x\right) dt,x\in \mathbf{R}%
^{d}.
\end{eqnarray*}%
In order to prove (\ref{fa3}) it is enough to show that%
\begin{eqnarray}
&&\int_{\left\vert x\right\vert >3}\int_{0}^{\infty }\left\vert L^{R}p^{\ast
R}\left( t,x-\hat{s}\right) -L^{R}p^{\ast R}\left( t,x\right) \right\vert
dtdx  \label{far4} \\
&\leq &C,\left\vert \hat{s}\right\vert =1,R>0,  \notag
\end{eqnarray}%
with $C=C=C\left( d,p,\kappa ,l,N_{0},N,N_{1}\right) .$

Let $I_{1}=\{t>0:\gamma \left( t\right) \leq 1\},I_{2}=\left\{ t>0:\gamma
\left( t\right) >1\right\} $. Now,%
\begin{eqnarray*}
&&\int_{\left\vert x\right\vert >3}\int_{0}^{\infty }\left\vert L^{R}p^{\ast
R}\left( t,x-\hat{s}\right) -L^{R}p^{\ast R}\left( t,x\right) \right\vert
dtdx \\
&\leq &\int_{\left\vert x\right\vert >3}\int_{I_{1}}\left\vert L^{R}p^{\ast
R}\left( t,x-\hat{s}\right) -L^{R}p^{\ast R}\left( t,x\right) \right\vert
dtdx \\
&&+\int_{\left\vert x\right\vert >3}\int_{I_{2}}\left\vert L^{R}p^{\ast
R}\left( t,x-\hat{s}\right) -L^{R}p^{\ast R}\left( t,x\right) \right\vert
dtdx \\
&=&A_{1}+A_{2}.
\end{eqnarray*}

By Lemma \ref{ld1}, 
\begin{eqnarray*}
A_{1} &=&\int_{\left\vert x\right\vert >3}\int_{I_{1}}\left\vert
L^{R}p^{\ast R}\left( t,x-\hat{s}\right) -L^{R}p^{\ast R}\left( t,x\right)
\right\vert dtdx \\
&\leq &2\int_{\left\vert x\right\vert >2}\int_{I_{1}}\left\vert L^{R}p^{\ast
R}\left( t,x\right) \right\vert dtdx \\
&\leq &C\int_{I_{1}}\left( 1+t\gamma \left( t\right) ^{-\alpha
_{1}}+1_{\sigma \in \left( 1,2\right) }\gamma \left( t\right) ^{-1}\right)
dt\leq C.
\end{eqnarray*}%
By Fubini theorem, 
\begin{eqnarray*}
A_{2} &=&\int_{\left\vert x\right\vert >3}\int_{I_{2}}\left\vert
L^{R}p^{\ast R}\left( t,x-\hat{s}\right) -L^{R}p^{\ast R}\left( t,x\right)
\right\vert dtdx \\
&\leq &\int_{\left\vert x\right\vert >3}\int_{I_{2}}\int_{0}^{1}\left\vert
L^{R}\nabla p^{\ast R}\left( t,x-r\hat{s}\right) \right\vert drdtdx \\
&\leq &2\int_{I_{2}}\int_{\left\vert x\right\vert >2}\left\vert L^{R}\nabla
p^{\ast R}\left( t,x\right) \right\vert dxdt\leq 2\int_{I_{2}}\left\vert
L^{R}\nabla p^{\ast R}\left( t,\cdot \right) \right\vert _{L_{1}}dt.
\end{eqnarray*}%
By Lemma \ref{ld1},%
\begin{equation*}
\int_{I_{2}}\left\vert L^{R}\nabla p^{\ast R}\left( t,\cdot \right)
\right\vert _{L_{1}}dt\leq C\int_{I_{2}}\gamma \left( t\right) ^{-(1+\alpha
_{2})}dt\leq C.
\end{equation*}

The statement is proved.
\end{proof}

\begin{corollary}
\label{coro2}Let $\pi _{0}\in \mathfrak{A}^{\sigma },\kappa $ be a scaling
function with scaling factor $l,$ and \textbf{D}$\left( \kappa ,l\right) $
hold for $\pi _{0}$. Let $\pi \in \mathfrak{A}^{\sigma }$ be such that%
\begin{equation*}
\int_{\left\vert z\right\vert \leq 1}\left\vert z\right\vert ^{\alpha _{1}}%
\tilde{\pi}_{R}(dz)+\int_{\left\vert z\right\vert >1}\left\vert z\right\vert
^{\alpha _{2}}\tilde{\pi}_{R}(dz)\leq N\text{ }\forall R>0,
\end{equation*}
\emph{where }$\tilde{\pi}_{R}\left( dy\right) =\kappa \left( R\right) \pi
_{R}\left( dy\right) $, and $\alpha _{1},\alpha _{2}$ are exponents in
assumption \textbf{D}$\left( \kappa \right) $ for $\pi _{0}$. Let $v\in 
\tilde{C}^{\infty }\left( \mathbf{R}^{d}\right) $ and $u\in \tilde{C}%
^{\infty }\left( \mathbf{R}^{d}\right) $ be the unique solution to 
\begin{equation}
\left( L^{\pi _{0}}-\lambda \right) u=v\text{ in }\mathbf{R}^{d}.
\label{far9}
\end{equation}

Then for each $p\in \left( 1,\infty \right) $ there is a constant $C=C\left(
d,p,\kappa ,l,N_{0},N,N_{1},c_{1}\right) $ such that%
\begin{equation*}
\left\vert L^{\pi }u\right\vert _{L_{p}}\leq C\left\vert v\right\vert
_{L_{p}}.
\end{equation*}
\end{corollary}

\begin{proof}
Let $v\in \tilde{C}^{\infty }\left( \mathbf{R}^{d}\right) ,\lambda >0$.
There is a unique $u\in \tilde{C}^{\infty }\left( \mathbf{R}^{d}\right) $
solving (\ref{far9})$.$ According to Lemma \ref{le00},%
\begin{equation*}
L^{\pi }u(x)=\int_{0}^{\infty }e^{-\lambda t}\mathbf{E}L^{\pi }v\left(
x+Z_{t}^{\pi _{0}}\right) dt,x\in \mathbf{R}^{d}.
\end{equation*}%
By (\ref{far1}), 
\begin{eqnarray}
T_{\lambda }^{\varepsilon }v\left( x\right) &=&\int L^{\pi
}v(x-y)\int_{\varepsilon }^{\infty }e^{-\lambda t}p^{\pi _{0}^{\ast }}\left(
t,y\right) dtdy  \label{fara1} \\
&=&\int_{\varepsilon }^{\infty }e^{-\lambda t}\mathbf{E}L^{\pi }v\left(
x+Z_{t}^{\pi _{0}}\right) dt,x\in \mathbf{R}^{d}.  \notag
\end{eqnarray}%
By Lemma \ref{pro1}, for each $p\in \left( 1,\infty \right) $ there is a
constant $C=C=C\left( d,p,\kappa ,l,N_{0},N,N_{1},c_{1}\right) $ such that 
\begin{equation}
\left\vert T_{\lambda }^{\varepsilon }v\right\vert _{L_{p}}\leq C\left\vert
v\right\vert _{L_{p}}.  \label{fara2}
\end{equation}%
Passing to the limit in (\ref{fara2}) and (\ref{fara1}) as $\varepsilon
\rightarrow 0,$ we have%
\begin{equation*}
\left\vert L^{\pi }u\right\vert _{L_{p}}\leq C\left\vert v\right\vert
_{L_{p}}.
\end{equation*}
\end{proof}

Now we prove the continuity of $L^{\pi }$-norm.

\begin{proposition}
\label{pro2}Let $\pi _{0}\in \mathfrak{A}^{\sigma },\kappa $ be a scaling
function with scaling factor $l,$ and \textbf{D}$\left( \kappa ,l\right) $
hold for $\pi _{0}$. Let $\pi \in \mathfrak{A}^{\sigma }$ be such that%
\begin{equation*}
\int_{\left\vert z\right\vert \leq 1}\left\vert z\right\vert ^{\alpha _{1}}%
\tilde{\pi}_{R}(dz)+\int_{\left\vert z\right\vert >1}\left\vert z\right\vert
^{\alpha _{2}}\tilde{\pi}_{R}(dz)\leq N\text{ }\forall R>0,
\end{equation*}
\emph{where }$\tilde{\pi}_{R}\left( dy\right) =\kappa \left( R\right) \pi
_{R}\left( dy\right) $, and $\alpha _{1},\alpha _{2}$ are exponents in
assumption \textbf{D}$\left( \kappa \right) $ for $\pi _{0}$.

Then for each $p\in \left( 1,\infty \right) $ there is a constant $%
C=C=C\left( d,p,\kappa ,l,N_{0},N,N_{1},c_{1}\right) $ such that%
\begin{equation*}
\left\vert L^{\pi }f\right\vert _{L_{p}}\leq C\left\vert L^{\pi
_{0}}f\right\vert _{L_{p}},f\in \tilde{C}^{\infty }\left( \mathbf{R}%
^{d}\right) .
\end{equation*}
\end{proposition}

\begin{proof}
Let $f\in \tilde{C}^{\infty }\left( \mathbf{R}^{d}\right) ,\lambda >0.$ Then 
$v=\left( L^{\pi _{0}}-\lambda \right) f\in \tilde{C}^{\infty }\left( 
\mathbf{R}^{d}\right) $, and, by Corollary \ref{coro2}, there is $C=C\left(
d,p,\kappa ,\mu ^{0},N_{2},N,C_{0}\right) $ such that%
\begin{equation*}
\left\vert L^{\pi }f\right\vert _{L_{p}}\leq C\left\vert v\right\vert
_{L_{p}}=C\left\vert \left( L^{\pi _{0}}-\lambda \right) f\right\vert
_{L_{p}}.
\end{equation*}%
Since $C$ does not depend on $\lambda >0$, the statement follows.
\end{proof}

\subsubsection{Proof of Theorem \protect\ref{thm:main}}

\emph{Existence. }Let $f\in L_{p}\left( \mathbf{R}^{d}\right) $. There is a
sequence $f_{n}\in C_{0}^{\infty }\left( \mathbf{R}^{d}\right) $ such that $%
f_{n}\rightarrow f$ in $L_{p}$. For each $n,$ there is unique $u_{n}\in 
\tilde{C}^{\infty }\left( \mathbf{R}^{d}\right) $ solving (\ref{2'}). Hence%
\begin{equation*}
\left( L^{\pi }-\lambda \right) \left( u_{n}-u_{m}\right) =f_{n}-f_{m}.
\end{equation*}%
By Corollary \ref{coro2} and Lemma \ref{le00},%
\begin{eqnarray*}
\left\vert L^{\pi _{0}}\left( u_{n}-u_{m}\right) \right\vert _{L_{p}} &\leq
&C\left\vert f_{n}-f_{m}\right\vert _{L_{p}}\rightarrow 0, \\
\left\vert u_{n}-u_{m}\right\vert _{L_{p}} &\leq &\frac{1}{\lambda }%
\left\vert f_{n}-f_{m}\right\vert _{L_{p}}\rightarrow 0,
\end{eqnarray*}%
as $n,m\rightarrow \infty $. Hence there is $u\in H_{p}^{\pi _{0}}$ so that $%
u_{n}\rightarrow u$ in $H_{p}^{\pi _{0}}$. Using Proposition \ref{pro2}, we
can pass to the limit in $\left( L^{\pi }-\lambda \right) u_{n}=f_{n}$ as $%
n\rightarrow \infty $. Obviously, $\left( L^{\pi }-\lambda \right) u=f$ in $%
L_{p}$.

\emph{Uniqueness. }Assume $u_{1},u_{2}\in H_{p}^{\pi _{0}}\left( \mathbf{R}%
^{d}\right) $ solve (\ref{2'})$.$ Then $u=u_{1}-u_{2}\in H_{p}^{\pi _{0}}$
solves $\left( L^{\pi }-\lambda \right) u=0$, i.e. $\forall \varphi \in 
\tilde{C}^{\infty }\left( \mathbf{R}^{d}\right) $ 
\begin{equation*}
\int \varphi \left( L^{\pi }-\lambda \right) u=\int u\left( L^{\pi ^{\ast
}}-\lambda \right) \varphi dx=0
\end{equation*}%
According to Lemma \ref{le00}, $\int ufdx=0$ $\forall f\in \tilde{C}^{\infty
}\left( \mathbf{R}^{d}\right) $. Hence $u=0$ a.e. The statement is proved.

\subsection{Proof of Theorem \protect\ref{t1}}

Let $f\in \tilde{C}^{\infty }\left( E\right) $ and $u\in \tilde{C}^{\infty
}\left( E\right) $ be the solution to 
\begin{eqnarray}
\partial _{t}u &=&L^{\pi }u-\lambda u+f\text{ in }E,  \label{pf1} \\
u\left( 0,\cdot \right) &=&0.  \notag
\end{eqnarray}%
By Lemma \ref{led}, the associated process $Z_{t}^{\pi }$ has a density
function $p^{\pi }\left( t,x\right) ,x\in \mathbf{R}^{d}$. Then $p^{\pi
^{\ast }}\left( t,x\right) =p^{\pi }\left( t,-x\right) ,x\in \mathbf{R}^{d},$
is pdf of $Z_{t}^{\pi ^{\ast }}$. By Lemma \ref{le0}, 
\begin{eqnarray*}
u\left( t,x\right) &=&\int_{0}^{t}e^{-\lambda \left( t-s\right) }\mathbf{E}%
f\left( s,x+Z_{t-s}^{\pi }\right) ds \\
&=&\int_{0}^{t}\int e^{-\lambda \left( t-s\right) }f\left( s,x-y\right)
p^{\pi ^{\ast }}\left( t-s,y\right) dyds \\
&=&\int_{0}^{t}\int e^{-\lambda \left( t-s\right) }p^{\pi ^{\ast }}\left(
t-s,x-y\right) f\left( s,y\right) dyds,
\end{eqnarray*}%
and%
\begin{eqnarray}
L^{\pi _{0}}u\left( t,x\right) &=&\int_{0}^{t}e^{-\lambda \left( t-s\right) }%
\mathbf{E}L^{\pi _{0}}f\left( s,x+Z_{t-s}^{\pi }\right) ds  \label{fo2} \\
&=&\int_{0}^{t}\int e^{-\lambda \left( t-s\right) }L^{\pi _{0}}f\left(
s,x-y\right) p^{\pi ^{\ast }}\left( t-s,y\right) dyds  \notag \\
&=&\int_{0}^{t}\int e^{-\lambda \left( t-s\right) }(L^{\pi _{0}}p^{\pi
^{\ast }})\left( t-s,x-y\right) f\left( s,y\right) dyds.  \notag
\end{eqnarray}%
Hence%
\begin{equation*}
L^{\pi _{0}}u\left( t,x\right) =\int_{-\infty }^{\infty }\int e^{-\lambda
\left( t-s\right) }K\left( t-s,x-y\right) \chi _{\lbrack 0,\infty )}\left(
t-s\right) f\left( s,y\right) dsdy,
\end{equation*}%
where $f(s,y)=\chi _{\lbrack 0,\infty )}\left( s\right) f\left( s,y\right)
,\left( s,y\right) \in \mathbf{R}^{d+1},$%
\begin{equation*}
K\left( t,x\right) =L^{\pi _{0}}p^{\pi ^{\ast }}\left( t,x\right) \chi
_{\lbrack 0,\infty )}\left( t\right) ,\left( t,x\right) \in \mathbf{R}^{d+1}.
\end{equation*}%
Let $\varepsilon \in \left( 0,1\right) $,%
\begin{equation*}
K^{\varepsilon }\left( t,x\right) =L^{\pi _{0}}p^{\pi ^{\ast }}\left(
t,x\right) \chi _{\lbrack \varepsilon ,\infty )}\left( t\right) ,\left(
t,x\right) \in \mathbf{R}^{d+1},
\end{equation*}%
and consider for $h\in C_{0}^{\infty }\left( \mathbf{R}^{d+1}\right) ,$%
\begin{eqnarray}
T_{\lambda }^{\varepsilon }h\left( t,x\right) &=&\int_{-\infty
}^{t-\varepsilon }\int e^{-\lambda \left( t-s\right) }p^{\pi ^{\ast }}\left(
t-s,x-y\right) L^{\pi _{0}}h\left( s,y\right) dyds  \label{fo1} \\
&=&\int_{-\infty }^{\infty }\int K_{\lambda }^{\varepsilon }\left(
t-s,x-y\right) h\left( s,y\right) dyds,\left( t,x\right) \in \mathbf{R}^{d},
\notag
\end{eqnarray}%
where $K_{\lambda }^{\varepsilon }\left( t,x\right) =e^{-\lambda
t}K^{\varepsilon }\left( t,x\right) ,\left( t,x\right) \in \mathbf{R\times R}%
^{d}=\mathbf{R}^{d+1}.$

\begin{claim}
\label{cl1}For each $p\in \left( 1,\infty \right) $ there is a constant $%
C=C\left( d,p,l,N,N_{0},N_{1},c_{1}\right) $ such that%
\begin{equation}
\left\vert T_{\lambda }^{\varepsilon }h\right\vert _{L_{p}}\leq C\left\vert
h\right\vert _{L_{p}},h\in C_{0}^{\infty }\left( \mathbf{R}^{d+1}\right) ,
\label{df2}
\end{equation}%
here $L_{p}=L_{p}\left( \mathbf{R}^{d+1}\right) .$
\end{claim}

\begin{proof}
We will apply Calderon-Zygmund theorem (see Theorem \ref{t3} in Appendix).
First we will prove the estimate in $L_{2}.$ Then, according to Theorem \ref%
{t3}, the proof reduces to verification of H\"{o}rmander condition (see \ \
\ below).

1. We start with $p=2$. Obviously,%
\begin{equation*}
\left\vert \widehat{T_{\lambda }^{\varepsilon }h}\left( t,\xi \right)
\right\vert \leq \int_{-\infty }^{t}e^{-\lambda \left( t-s\right) }|\psi
^{\pi _{0}}\left( \xi \right) |\exp \left\{ \func{Re}\psi ^{\pi ^{\ast
}}\left( \xi \right) \left( t-s\right) \right\} \left\vert \hat{h}\left(
s,\xi \right) \right\vert ds,
\end{equation*}%
$\left( t,\xi \right) \in E$. Hence by H\"{o}lder inequality, Fubini
theorem, Lemma \ref{le5} and Corollary \ref{c1}, we have $|\psi ^{\pi
_{0}}\left( \xi \right) |\leq C|\func{Re}\psi ^{\pi ^{\ast }}\left( \xi
\right) |,\xi \in \mathbf{R}^{d},$ for some $C=C\left( N,c_{1},l\right) $\
and 
\begin{equation*}
\left\vert \widehat{T_{\lambda }^{\varepsilon }h}\right\vert
_{L_{2}}^{2}\leq C\left\vert h\right\vert _{L_{2}\left( T\right) },
\end{equation*}
i.e., (\ref{df2}) follows (cf. the proof of Lemma \ref{le10}).

2. We prove (\ref{df2}) for $p\in \left( 1,2\right) $ using a version of
Calderon-Zygmund theorem (Theorem \ref{t3} in Appendix). Let $\mathbb{Q}$ be
the collection of sets $Q_{\delta }=Q_{\delta }\left( t,x\right) =\left(
t-\kappa \left( \delta \right) ,t+\kappa \left( \delta \right) \right)
\times B_{\delta }\left( x\right) ,\left( t,x\right) \in \mathbf{R}\times 
\mathbf{R}^{d}=\mathbf{R}^{d+1},\delta >0$.

Note

(i) $\left( t,x\right) \notin \left( s-\kappa \left( c\delta \right)
,s+\kappa \left( c\delta \right) \right) \times B_{c\delta }\left( y\right) $
$\Longleftrightarrow $ $\ \left( t-s,x-y\right) \notin Q_{c\delta }\left(
0\right) $ $\Longleftrightarrow $ $\left\vert t-s\right\vert \geq \kappa
\left( c\delta \right) $ or $\left\vert x-y\right\vert \geq c\delta ;$

(ii) $(\bar{s},\bar{y})\in Q_{\delta }\left( s,y\right) $ $%
\Longleftrightarrow $ $\left( \bar{s}-s,\bar{y}-y\right) \in Q_{\delta
}\left( 0\right) $ $\Longleftrightarrow $ $|s-\bar{s}|<\kappa \left( \delta
\right) $ and $\left\vert y-\bar{y}\right\vert <\delta .$

According to Theorem \ref{t3}, it is enough to show%
\begin{equation}
\int \chi _{Q_{c\delta }\left( 0\right) ^{c}}\left( t-s,x-y\right)
\left\vert K_{\lambda }^{\varepsilon }\left( t-\bar{s},x-\bar{y}\right)
-K_{\lambda }^{\varepsilon }\left( t-s,x-y\right) \right\vert dxdt\leq A%
\text{ }  \label{df3}
\end{equation}%
for all $(\bar{s},\bar{y})\in Q_{\delta }\left( s,y\right)
\Longleftrightarrow \left( \bar{s}-s,\bar{y}-y\right) \in Q_{\delta }\left(
0\right) $. Equivalently, we have to prove that 
\begin{eqnarray*}
&&\int \chi _{Q_{c\delta }\left( 0\right) ^{c}}\left( t,x\right) \left\vert
K_{\lambda }^{\varepsilon }\left( t-\tilde{s},x-\tilde{y})\right)
-K_{\lambda }^{\varepsilon }\left( t,x\right) \right\vert dxdt \\
&=&\delta ^{d}\kappa \left( \delta \right) \int \chi _{Q_{c\delta }\left(
0\right) ^{c}}\left( \kappa \left( \delta \right) t,\delta x\right)
\left\vert K_{\lambda }^{\varepsilon ,\delta }\left( t-\hat{s},x-\hat{y}%
)\right) -K_{\lambda }^{\varepsilon ,\delta }\left( t,x\right) \right\vert
dxdt \\
&\leq &A\text{ }\forall \left\vert \hat{s}\right\vert \leq 1,\left\vert \hat{%
y}\right\vert \leq 1,
\end{eqnarray*}%
where 
\begin{equation*}
K_{\lambda }^{\varepsilon ,\delta }\left( t,x\right) =K_{\lambda
}^{\varepsilon }\left( \kappa \left( \delta \right) t,\delta x\right)
,\left( t,x\right) \in \mathbf{R}^{d+1}.
\end{equation*}%
Fix $c>4$ such that $(\bar{c}-1,\infty )\subseteq I_{2}$ with $\bar{c}%
=l\left( 1/c\right) ^{-1}>3$. Let $G=\left( -\bar{c},\bar{c}\right) \times
B_{c}\left( 0\right) .$ Since $\chi _{G}\left( t,x\right) \leq \chi
_{Q_{c\delta }\left( 0\right) }\left( \kappa \left( \delta \right) t,\delta
x\right) $, it is enough to prove that%
\begin{equation}
\delta ^{d}\kappa \left( \delta \right) \int_{\mathbf{R}^{d+1}}\chi
_{G^{c}}\left( t,x\right) \left\vert K_{\lambda }^{\varepsilon ,\delta
}\left( t-\hat{s},x-\hat{y})\right) -K_{\lambda }^{\varepsilon ,\delta
}\left( t,x\right) \right\vert dxdt\leq A  \label{df4}
\end{equation}%
for all $\left\vert \hat{s}\right\vert \leq 1,\left\vert \hat{y}\right\vert
\leq 1.$ Since%
\begin{equation*}
\delta ^{d}\kappa \left( \delta \right) K_{\lambda }^{\varepsilon ,\delta
}\left( t,x\right) =e^{-\lambda t}L^{0,\delta }p^{\tilde{\pi}_{\delta
}^{\ast }}\left( t,x\right) \chi _{\left( \varepsilon /\kappa \left( \delta
\right) ,\infty \right) }\left( t\right) ,\left( t,x\right) \in \mathbf{R}%
^{d+1},
\end{equation*}%
with $L^{0,\delta }=L^{\tilde{\pi}_{0;\delta }}$, $\tilde{\pi}_{0;\delta
}\left( dy\right) =\kappa \left( \delta \right) \pi _{0}\left( \delta
dy\right) ,\tilde{\pi}_{\delta }^{\ast }\left( dy\right) =\kappa \left(
\delta \right) \pi ^{\ast }\left( \delta dy\right) ,$ we rewrite (\ref{df4})
as 
\begin{eqnarray}
B &=&\int_{\mathbf{R}^{d+1}}\chi _{G^{c}}\left( t,x\right) |e^{-\lambda (t-%
\hat{s})}L^{0,\delta }p^{\pi _{\delta }^{\ast }}\left( t-\hat{s},x-\hat{y}%
\right) \chi _{\left( \varepsilon /\kappa \left( \delta \right) ,\infty
\right) }\left( t-\hat{s}\right)  \notag \\
&&-e^{-\lambda t}L^{0,\delta }p^{\pi _{\delta }^{\ast }}\left( t,x\right)
\chi _{\left( \varepsilon /\kappa \left( \delta \right) ,\infty \right)
}\left( t\right) |dxdt  \label{df5} \\
&\leq &A  \notag
\end{eqnarray}

Since $G^{c}\subseteq \left\{ \left( t,x\right) :\left\vert t\right\vert
\leq \bar{c},\left\vert x\right\vert \geq 3\right\} \cup \left\{ \left(
t,x\right) :\left\vert t\right\vert \geq \bar{c}\right\} =G_{1}\cup G_{2}$, 
\begin{equation*}
B\leq \int_{G_{1}}...+\int_{G_{2}}...=B_{1}+B_{2}.
\end{equation*}

By Lemma \ref{ld1} a),%
\begin{eqnarray*}
B_{1}\leq &&2\int_{0}^{\bar{c}+1}\int_{\left\vert x\right\vert >2}\left\vert
L^{0,\delta }p^{\pi _{\delta }^{\ast }}\left( t,x\right) \right\vert dtdx \\
&\leq &C\int_{0}^{\bar{c}+1}\left( 1+t\gamma \left( t\right) ^{-\alpha
_{1}}+1_{\sigma \in \left( 1,2\right) }\gamma \left( t\right) ^{-1}\right)
dt.
\end{eqnarray*}

\emph{Estimate of }$B_{2}$. We have%
\begin{eqnarray*}
&&B_{2} \\
&\leq &\int_{\bar{c}}^{\infty }\int |L^{0,\delta }p^{\pi _{\delta }^{\ast
}}\left( t,x\right) |~\left\vert \chi _{\left( \varepsilon /\kappa \left(
\delta \right) ,\infty \right) }\left( t\right) -\chi _{\left( \varepsilon
/\kappa \left( \delta \right) ,\infty \right) }\left( t-\hat{s}\right)
\right\vert dtdx \\
&&+\int_{\bar{c}}^{\infty }\int |e^{-\lambda (t-\hat{s})}L^{0,\delta }p^{\pi
_{\delta }^{\ast }}(t-\hat{s},x-\hat{y})-e^{-\lambda t}L^{0,\delta }p^{\pi
_{\delta }^{\ast }}\left( t,x\right) |dtdx \\
&=&B_{21}+B_{22}.
\end{eqnarray*}%
Then $B_{21}\leq C$ by Lemma \ref{led} b). Since 
\begin{eqnarray*}
&&e^{-\lambda (t-\hat{s})}L^{0,\delta }p^{\pi _{\delta }^{\ast }}(t-\hat{s}%
,x-\hat{y})-e^{-\lambda t}L^{0,\delta }p^{\pi _{\delta }^{\ast }}\left(
t,x\right) \\
&=&\int_{0}^{1}[-\lambda e^{-\lambda (t-r\hat{s})}L^{0,\delta }p^{\pi
_{\delta }^{\ast }}\left( t-r\hat{s},x-r\hat{y}\right) \\
&&+e^{-\lambda (t-r\hat{s})}L^{0,\delta }\partial _{t}p^{\pi _{\delta
}^{\ast }}\left( t-r\hat{s},x-r\hat{y}\right) ](-\hat{s})dr \\
&&+\int_{0}^{1}e^{-\lambda (t-r\hat{s})}\hat{y}\cdot L^{0,\delta }\nabla
p^{\pi _{\delta }^{\ast }}\left( t-r\hat{s},x-r\hat{y}\right) dr
\end{eqnarray*}%
we have%
\begin{eqnarray*}
&&B_{22} \\
&\leq &\int_{\bar{c}}^{\infty }\int \int_{0}^{1}\lambda e^{-\lambda (t-r\hat{%
s})}|L^{0,\delta }p^{\pi _{\delta }^{\ast }}(t-r\hat{s},x-r\hat{y})|drdtdx \\
&&+\int_{\bar{c}}^{\infty }\int \int_{0}^{1}|L^{0,\delta }\partial
_{t}p^{\pi _{\delta }^{\ast }}\left( t-rs,x-r\hat{y}\right) |drdtdx \\
&&+\int_{\bar{c}}^{\infty }\int \int_{0}^{1}|L^{0,\delta }\nabla p^{\pi
_{\delta }^{\ast }}\left( t-r\hat{s},x-r\hat{y}\right) |drdtdx \\
&=&b_{1}+b_{2}+b_{3}.
\end{eqnarray*}%
By Lemma \ref{led} b), 
\begin{equation*}
b_{1}\leq \int_{\bar{c}-1}^{\infty }\int \lambda e^{-\lambda t}|L^{0,\delta
}p^{\pi _{\delta }^{\ast }}(t,x)|dtdx\leq C.
\end{equation*}

Since $\partial _{t}p^{\pi ^{\ast }}\left( t,x\right) =L^{\pi }p\left(
t,x\right) ,t>0$, we have by Lemma \ref{ld1} c),%
\begin{eqnarray*}
b_{2} &=&\int_{\bar{c}}^{\infty }\int \int_{0}^{1}|L^{0,\delta }L^{\pi
_{\delta }}p^{\pi _{\delta }^{\ast }}\left( t-rs,x-r\hat{y}\right) |drdtdx \\
&\leq &\int_{\bar{c}-1}^{\infty }\int |L^{0,\delta }L^{\pi _{\delta }}p^{\pi
_{\delta }^{\ast }}\left( t,x\right) |dtdx\leq C\int_{\bar{c}-1}^{\infty
}\gamma \left( t\right) ^{-2\alpha _{2}}dt.
\end{eqnarray*}%
Finally, by Lemma \ref{ld1} b),%
\begin{eqnarray*}
b_{3} &\leq &\int_{\bar{c}-1}^{\infty }\int |L^{0,\delta }\nabla p^{\pi
_{\delta }^{\ast }}\left( t,x\right) |dtdx \\
&\leq &C\int_{\bar{c}-1}^{\infty }\gamma \left( t\right) ^{-1-\alpha _{2}}dt.
\end{eqnarray*}%
Claim is proved for $p\in \left( 1,2\right) $ by Theorem \ref{t3}.

3. We prove the statement for $p>2$ in a standard way (by duality argument).
First note that%
\begin{equation*}
L^{\pi _{0}}p^{\pi ^{\ast }}\left( t,-x\right) =L^{\pi _{0}^{\ast }}p^{\pi
}\left( t,x\right) ,\left( t,x\right) \in \mathbf{R}^{d+1},
\end{equation*}%
and let%
\begin{equation*}
\tilde{K}^{\varepsilon }\left( t,x\right) =K^{\varepsilon }\left(
t,-x\right) =L^{\pi _{0}^{\ast }}p^{\pi }\left( t,x\right) \chi _{\lbrack
\varepsilon ,\infty )}\left( t\right) ,\left( t,x\right) \in \mathbf{R}%
^{d+1},
\end{equation*}%
and%
\begin{equation*}
\tilde{T}_{\lambda }^{\varepsilon }g\left( s,y\right) =\int e^{-\lambda
\left( s-t\right) }\tilde{K}^{\varepsilon }\left( s-t,y-x\right) g\left(
t,x\right) dtdx,\left( s,y\right) \in \mathbf{R}^{d+1}.
\end{equation*}%
Let $1/p+1/q=1,h,g\in C_{0}^{\infty }\left( \mathbf{R}^{d+1}\right) $. Then,
denoting $\tilde{g}\left( t,x\right) =g\left( -t,x\right) ,\left( t,x\right)
\in \mathbf{R}^{d+1},$ we have (by Fubini theorem and changing the variable
of integration)%
\begin{eqnarray*}
&&\int T_{\lambda }^{\varepsilon }h\left( t,x\right) g\left( t,x\right) dtdx
\\
&=&\int \int e^{-\lambda \left( t-s\right) }K^{\varepsilon }\left(
t-s,x-y\right) h\left( s,y\right) dsdyg\left( t,x\right) dtdx \\
&=&\int \int e^{-\lambda \left( s-t\right) }\tilde{K}^{\varepsilon }\left(
s-t,y-x\right) g\left( -t,x\right) dtdxh\left( -s,y\right) dsdy \\
&=&\int \tilde{T}_{\lambda }^{\varepsilon }\tilde{g}\left( s,y\right)
h\left( -s,y\right) dsdy,
\end{eqnarray*}%
and (\ref{df2}) holds for $\tilde{T}_{\lambda }^{\varepsilon }$ and $q\in
\left( 1,2\right) $ (see Corollary \ref{c1}). Hence by H\"{o}lder inequality,%
\begin{eqnarray*}
&&\left\vert \int T_{\lambda }^{\varepsilon }h\left( t,x\right) g\left(
t,x\right) dtdx\right\vert \\
&\leq &\left\vert \tilde{T}_{\lambda }^{\varepsilon }\tilde{g}\right\vert
_{L_{q}}\left\vert h\right\vert _{L_{p}}\leq C\left\vert g\right\vert
_{L_{q}}\left\vert h\right\vert _{L_{p}},
\end{eqnarray*}%
and (\ref{df2}) holds for $p>2$ as well, that is for all \thinspace $p\in
\left( 1,\infty \right) $. The claim is proved.
\end{proof}

Now, we see that for $h\left( t,x\right) =\chi _{\lbrack 0,T]}\left(
t\right) f\left( t,x\right) $ with $f\in \tilde{C}\left( E\right) ,$%
\begin{equation}
T_{\lambda }^{\varepsilon }h\left( t,x\right) =\int_{0}^{t-\varepsilon
}e^{-\lambda \left( t-s\right) }\mathbf{E}L^{\pi _{0}}f\left(
s,x+Z_{t-s}^{\pi }\right) ds,\left( t,x\right) \in E.  \label{fo3}
\end{equation}%
If $u\in \tilde{C}^{\infty }\left( E\right) $ solves (\ref{pf1}), then 
\begin{equation}
L^{\pi _{0}}u\left( t,x\right) =\int_{0}^{t}e^{-\lambda \left( t-s\right) }%
\mathbf{E}L^{\pi _{0}}f\left( s,x+Z_{t-s}^{\pi }\right) ds,  \label{fo4}
\end{equation}%
and $\left\vert T_{\lambda }^{\varepsilon }h-L^{\pi _{0}}u\right\vert
_{L_{p}}\rightarrow 0$ as $\varepsilon \rightarrow 0$. Since the constant $%
C=C\left( d,p,l,N,N_{0},N_{1},c_{1}\right) $ in the above Claim \ref{cl1}
does not depend on $\varepsilon $, passing to the limit in it we get the
estimate 
\begin{equation}
\left\vert L^{\pi _{0}}u\right\vert _{L_{p}}\leq C\left\vert f\right\vert
_{L_{p}}.  \label{fo5}
\end{equation}

We finish the proof of Theorem \ref{t1} the same way as the proof of Theorem %
\ref{thm:main}. Let $f\in L_{p}\left( \mathbf{R}^{d}\right) $. There is a
sequence $f_{n}\in \tilde{C}^{\infty }\left( E\right) $ such that $%
f_{n}\rightarrow f$ in $L_{p}\left( E\right) $. For each $n,$ there is
unique $u_{n}\in \tilde{C}^{\infty }\left( E\right) $ solving (\ref{1'}).
Hence%
\begin{equation*}
\partial _{t}\left( u_{n}-u_{m}\right) =\left( L^{\pi }-\lambda \right)
\left( u_{n}-u_{m}\right) +f_{n}-f_{m}.
\end{equation*}%
By (\ref{fo5}) and Lemma \ref{le0}, 
\begin{eqnarray}
\left\vert L^{\pi _{0}}\left( u_{n}-u_{m}\right) \right\vert _{L_{p}\left(
E\right) } &\leq &C\left\vert f_{n}-f_{m}\right\vert _{L_{p}\left( E\right)
}\rightarrow 0,  \label{fo6} \\
\left\vert u_{n}-u_{m}\right\vert _{L_{p}\left( E\right) } &\leq &\left( 
\frac{1}{\lambda }\wedge T\right) \left\vert f_{n}-f_{m}\right\vert
_{L_{p}\left( E\right) }\rightarrow 0,  \notag \\
\left\vert u_{n}\left( t\right) -u_{m}\left( t\right) \right\vert
_{L_{p}\left( \mathbf{R}^{d}\right) } &\leq &\left\vert
f_{n}-f_{m}\right\vert _{L_{p}\left( E\right) }\text{ }\forall t\in \left[
0,T\right] ,  \notag
\end{eqnarray}%
as $n,m\rightarrow \infty $. Hence there is $u\in \mathcal{H}_{p}^{\pi _{0}}$
so that $u_{n}\rightarrow u$ in $\mathcal{H}_{p}^{\pi _{0}}$. Moreover,%
\begin{equation}
\sup_{t\leq T}\left\vert u_{n}\left( t\right) -u\left( t\right) \right\vert
_{L_{p}\left( \mathbf{R}^{d}\right) }\rightarrow 0,  \label{fo7}
\end{equation}%
and, according to Proposition \ref{pro2}, 
\begin{equation}
\left\vert L^{\pi }f\right\vert _{L_{p}\left( E\right) }\leq C\left\vert
L^{\pi _{0}}f\right\vert _{L_{p}\left( E\right) },f\in \tilde{C}^{\infty
}\left( E\right) .  \label{fo8}
\end{equation}%
Hence (see (\ref{fo6})-(\ref{fo8})) we can pass to the limit in the equation%
\begin{equation}
u_{n}\left( t\right) =\int_{0}^{t}[L^{\pi }u_{n}\left( s\right) -\lambda
u_{n}\left( s\right) +f_{n}\left( s\right) ]ds,0\leq t\leq T.  \label{fo9}
\end{equation}%
Obviously, (\ref{fo9}) holds for $u$ and $f$. We proved the existence part
of Theorem \ref{t1}.

\emph{Uniqueness. }Assume $u_{1},u_{2}\in \mathcal{H}_{p}^{\pi _{0}}$ solve (%
\ref{1'})$.$ Then $u=u_{1}-u_{2}\in \mathcal{H}_{p}^{\pi _{0}}$ solves (\ref%
{pf1}) with $f=0$. Now, let $\varphi \in \tilde{C}^{\infty }\left( E\right) $%
, and $\tilde{\varphi}\left( t,x\right) =\varphi \left( T-t,x\right) ,\left(
t,x\right) \in E$. By Lemma \ref{le0}, there is unique $\tilde{v}\in \tilde{C%
}^{\infty }\left( E\right) $ solving (\ref{pf1}) with $f=\tilde{\varphi}$
and $\pi ^{\ast }$ instead of $\pi $. Let $v\left( t,x\right) =\tilde{v}%
\left( T-t,x\right) ,\left( t,x\right) \in E$. Then $\partial _{t}v+L^{\pi
^{\ast }}v-\lambda v+\varphi =0$ in $E$ and $v\left( T\right) =v\left(
T,\cdot \right) =0$. Integrating by parts,%
\begin{eqnarray*}
\int_{E}\varphi u &=&\int_{E}u\left( -\partial _{t}v-L^{\pi ^{\ast
}}v+\lambda v\right) \\
&=&\int_{E}v\left( \partial _{t}u-L^{\pi }u+\lambda u\right) =0.
\end{eqnarray*}%
Hence $\int_{E}u\varphi dx=0$ $\forall \varphi \in \tilde{C}^{\infty }\left(
E\right) $. Hence $u=0$ a.e. Theorem \ref{t1} is proved.

\section{Appendix}

Given a function $\kappa :\left( 0,\infty \right) \rightarrow \left(
0,\infty \right) $, consider the collection $\mathbb{Q}$ of sets $Q_{\delta
}=Q_{\delta }\left( t,x\right) =\left( t-\kappa \left( \delta \right)
,t+\kappa \left( \delta \right) \right) \times B_{\delta }\left( x\right)
,\left( t,x\right) \in \mathbf{R}\times \mathbf{R}^{d}=\mathbf{R}%
^{d+1},\delta >0$. The volume $\left\vert Q_{\delta }\left( t,x\right)
\right\vert =c_{0}\kappa \left( \delta \right) \delta ^{d}.$ We will need
the following assumptions.

\textbf{A1. }$\kappa $ is continuous, $\lim_{\delta \rightarrow 0}\kappa
\left( \delta \right) =0$ and $\lim_{\delta \rightarrow \infty }\kappa
\left( \delta \right) =\infty .$

\textbf{A2.} There is a constant $C_{1}$ and a nondecreasing continuous
function $l\left( \varepsilon \right) ,\varepsilon >0,$ such that $%
\lim_{\varepsilon \rightarrow 0}l\left( \varepsilon \right) =0$ and 
\begin{equation*}
\kappa \left( \varepsilon r\right) \leq l\left( \varepsilon \right) \kappa
(r),r>0,\varepsilon >0.
\end{equation*}

Since $Q_{\delta }\left( t,x\right) $ not "exactly" increases in $\delta $,
we present the basic estimates involving maximal functions based on the
system $\mathbb{Q=}\left\{ Q_{\delta }\right\} $.

\subsection{Vitali Lemma, maximal functions}

We start with engulfing property.

\begin{lemma}
\label{le1}Let \textbf{A}2\textbf{\ }hold. If $Q_{\delta }\left( t,x\right)
\cap Q_{\delta ^{\prime }}\left( r,z\right) \neq \emptyset $ with $\delta
^{\prime }\leq \delta $, then there is $K_{0}\geq 3$ such that $%
Q_{K_{0}\delta }\left( t,x\right) $ contains both, $Q_{\delta }\left(
t,x\right) $ and $Q_{\delta ^{\prime }}\left( r,z\right) ,$ and 
\begin{equation*}
\left\vert Q_{\delta }\left( t,x\right) \right\vert \leq \left\vert
Q_{K_{0}\delta }\left( t,x\right) \right\vert \leq K_{0}^{d}l\left(
K_{0}\right) \left\vert Q_{\delta }\left( t,x\right) \right\vert .
\end{equation*}
\end{lemma}

\begin{proof}
Let $\left( s,y\right) \in Q_{\delta }\left( t,x\right) \cap Q_{\delta
^{\prime }}\left( r,z\right) $ with $\delta ^{\prime }\leq \delta .$ If $%
\left( r^{\prime },z^{\prime }\right) \in Q_{\delta ^{\prime }}\left(
r,z\right) $, then $\left\vert z^{\prime }-x\right\vert \leq 3\delta $, and
using \textbf{A2},%
\begin{equation*}
\left\vert r^{\prime }-t\right\vert \leq \left\vert r^{\prime }-r\right\vert
+\left\vert r-s\right\vert +\left\vert s-t\right\vert \leq 2\kappa \left(
\delta ^{\prime }\right) +\kappa \left( \delta \right) \leq \lbrack 2l\left(
1\right) +1]\kappa \left( \delta \right) .
\end{equation*}%
We choose $K_{0}\geq 3$ so that $[2l\left( 1\right) +1]l(1/K_{0})\leq 1$. By 
\textbf{A2}, 
\begin{equation*}
\lbrack 2l\left( 1\right) +1]\kappa \left( \delta \right) \leq \lbrack
2l\left( 1\right) +1]l(1/K_{0})\kappa \left( K_{0}\delta \right) \leq \kappa
\left( K_{0}\delta \right) .
\end{equation*}
Hence $Q_{\delta ^{\prime }}\left( r,z\right) \subseteq Q_{K_{0}\delta
}\left( t,x\right) $ and, obviously, $Q_{\delta }\left( t,x\right) \subseteq
Q_{K_{0}\delta }\left( t,x\right) $. Also, 
\begin{equation*}
\left\vert Q_{K_{0}\delta }\left( t,x\right) \right\vert
=c_{0}K_{0}^{d}\delta ^{d}\kappa (K_{0}\delta )\leq c_{0}K_{0}^{d}\delta
^{d}l(K_{0})\kappa (\delta )=K_{0}^{d}l(K_{0})\left\vert Q_{\delta }\left(
t,x\right) \right\vert .
\end{equation*}
\end{proof}

Now, following 3.1.1 in \cite{stein1}, we prove Vitali covering lemma.

\begin{lemma}
\label{levi}Let $E\subseteq \mathbf{R}\times \mathbf{R}^{d}$ be a union of a
finite collection $\left\{ Q^{\prime }\right\} $ of sets from the system $%
\{Q_{\delta }\left( t,x):(t,x\right) \in \mathbf{R}^{d+1},\delta >0\}$ and 
\textbf{A2} hold.

There is a positive $c=\frac{1}{K_{0}^{d}l\left( K_{0}\right) }$ and a
disjoint subcollection \newline
$\left\{ Q^{k}=Q_{\delta _{k}}\left( t_{k},x_{k}\right) ,1\leq k\leq
m\right\} $ such that%
\begin{equation*}
\sum_{k=1}^{m}\left\vert Q^{k}\right\vert \geq c\left\vert E\right\vert .
\end{equation*}
\end{lemma}

\begin{proof}
Let $Q^{1}=Q_{\delta _{1}}\left( t_{1},x_{1}\right) $ be the set of the
collection $\left\{ Q^{\prime }\right\} $ with maximal $\delta $. Let $%
Q^{2}=Q_{\delta _{2}}\left( t_{2},x_{2}\right) $ be the set with maximal $%
\delta $ among remaining sets in $\left\{ Q^{\prime }\right\} $ that do not
intersect $Q^{1}$. According to Lemma \ref{le1}, $Q_{K_{0}\delta _{1}}\left(
t_{1},x_{1}\right) $ contains $Q^{1}$ and all $Q_{\delta }$ in $\left\{
Q^{\prime }\right\} $ that intersect $Q^{1}$ and such that $\delta \leq
\delta _{1}.$ Continuing we get $Q_{K_{0}\delta _{k}}\left(
t_{k},x_{k}\right) $ containing $Q^{k}=Q_{\delta _{k}}\left(
t_{k},x_{k}\right) $ and all $Q_{\delta }$ in $\left\{ Q^{\prime }\right\} $
that intersect $Q^{k}$ and such that $\delta \leq \delta _{k}.$ So we obtain
a finite disjoint subcollection $\left\{ Q^{k}=Q_{\delta _{k}}\left(
t_{k},x_{k}\right) ,1\leq k\leq m\right\} $ such that $\cup
_{k=1}^{m}Q_{K_{0}\delta _{k}}\left( t_{k},x_{k}\right) \supseteq Q_{\delta
} $ for any $Q^{\delta }$ in $\left\{ Q^{\prime }\right\} .$ Hence $\cup
_{k=1}^{m}Q_{K_{0}\delta _{k}}\left( t_{k},x_{k}\right) \supseteq E$, and by
Lemma \ref{le1},%
\begin{equation*}
\left\vert E\right\vert \leq \sum_{k=1}^{m}\left\vert Q_{K_{0}\delta
_{k}}\right\vert \leq K_{0}^{d}l\left( K_{0}\right) \sum_{k=1}^{m}\left\vert
Q^{k}\right\vert .
\end{equation*}
\end{proof}

\begin{remark}
\label{re1}The statement of the Lemma \ref{levi} still holds if instead of 
\textbf{A2} we assume that there is a constant $C$ so that $C\kappa \left(
\delta \right) \geq \kappa \left( \delta ^{\prime }\right) $ whenever $%
\delta \geq \delta ^{\prime }.$
\end{remark}

Following \cite{stein1}, for a locally integrable function $f\left(
t,x\right) $ on $\mathbf{R}^{d+1}$ we define 
\begin{equation*}
\left( A_{\delta }f\right) (t,x)=\frac{1}{\left\vert Q_{\delta }\left(
t,x\right) \right\vert }\int_{Q_{\delta }\left( t,x\right) }f\left(
s,y\right) dsdy,\left( t,x\right) \in \mathbf{R\times R}^{d},\delta >0
\end{equation*}%
and the maximal function of $f$ by%
\begin{equation*}
\mathcal{M}f\left( t,x\right) =\sup_{\delta >0}\left( A_{\delta }\left\vert
f\right\vert \right) (t,x),\left( t,x\right) \in \mathbf{R}^{d+1}.
\end{equation*}%
We use collection $\mathbb{Q}$ to define a larger, noncentered maximal
function of $f$, as%
\begin{equation*}
\widetilde{\mathcal{M}}f\left( t,x\right) =\sup_{\left( t,x\right) \in Q}%
\frac{1}{\left\vert Q\right\vert }\int_{Q}|f\left( s,y\right) |dsdy,\left(
t,x\right) \in \mathbf{R}^{d+1},
\end{equation*}%
where $\sup $ is taken over all $Q\in \mathbb{Q}$ that contain $(t,x).$

\begin{remark}
\label{re2}Let \textbf{A2} hold and $K_{0}$ be a constant in Lemma \ref{le1}%
. For a locally integrable $f$ on $\mathbf{R}^{d+1},$%
\begin{equation*}
\mathcal{M}f\leq \widetilde{\mathcal{M}}f\leq \frac{1}{K_{0}^{d}l\left(
K_{0}\right) }\mathcal{M}f.
\end{equation*}%
Indeed, if $(t,x)\in Q^{\prime }$=$Q_{\delta }\left( t^{\prime },x^{\prime
}\right) $, then by Lemma \ref{le1} 
\begin{equation*}
\frac{1}{\left\vert Q^{\prime }\right\vert }\int_{Q^{\prime }}\left\vert
f\right\vert \leq \frac{K_{0}^{d}l\left( K_{0}\right) }{\left\vert
Q_{K_{0}\delta }\left( t,x\right) \right\vert }\int_{Q_{K_{0}\delta }\left(
t,x\right) }\left\vert f\right\vert .
\end{equation*}

Note $\widetilde{\mathcal{M}}f$ is lower semicontinuous as a sup of lower
semicontinuous functions.
\end{remark}

Theorem 1.3.1 in \cite{stein1} holds for $\mathbb{Q}$ (we sketch its proof).

\begin{theorem}
\label{ft}Let \textbf{A2} hold and $f$ be measurable function on $\mathbf{R}%
^{d+1}=\mathbf{R\times R}^{d}.$

(a) If $f\in L_{p},1\leq p\leq \infty $, then $\mathcal{M}f$ is finite a.e.

(b) If $f\in L_{1}$, then for every $\alpha >0$, 
\begin{equation*}
\left\vert \left\{ \mathcal{M}f\left( t,x\right) >\alpha \right\}
\right\vert \leq \frac{c}{\alpha }\int |f|dm.
\end{equation*}

(c) If $f\in L_{p},1<p\leq \infty $, then $\mathcal{M}f\in L_{p}$ and%
\begin{equation*}
\left\vert \mathcal{M}f\right\vert _{L_{p}}\leq N_{p}\left\vert f\right\vert
_{L_{p}},
\end{equation*}%
where $N_{p}$ depends only on $p,l$ and $K_{0}.$
\end{theorem}

\begin{proof}
(b)\textbf{\ }Let $E_{\alpha }=\left\{ \widetilde{\mathcal{M}}f\left(
t,x\right) >\alpha \right\} $ and $E\subseteq E_{\alpha }$ be any compact
subset. Since $\widetilde{\mathcal{M}}f$ is lower semicontinuous, $E_{\alpha
}$ is open. By definition of $\widetilde{\mathcal{M}}f$ for each $\left(
t,x\right) \in E$, there is $Q\in \mathbb{Q}$ so that $\left( t,x\right) \in
Q$ and%
\begin{equation*}
\left\vert Q\right\vert \leq \frac{1}{\alpha }\int_{Q}\left\vert
f\right\vert .
\end{equation*}%
Since $E$ is compact there exist a finite number $Q_{\delta _{1}}\left(
t_{1},x_{1}\right) ,\ldots ,Q_{\delta _{n}}\left( t_{n},x_{n}\right) \in 
\mathbb{Q}$ so that $E\subseteq \cup _{j=1}^{n}Q_{\delta _{j}}\left(
t_{j},x_{j}\right) $. By Lemma \ref{levi}, there is a subcovering of
disjoint sets $Q^{1},\ldots ,Q^{m}$ so that%
\begin{equation*}
\left\vert E\right\vert \leq c\sum_{k=1}^{m}\left\vert Q^{k}\right\vert \leq 
\frac{c}{\alpha }\sum_{k=1}^{m}\int_{Q^{k}}\left\vert f\right\vert \leq 
\frac{c}{\alpha }\int \left\vert f\right\vert .
\end{equation*}%
with $c=K_{0}^{d}l\left( K_{0}\right) $. Taking $\sup $ over all such
compacts $E$ we get (b)\textbf{.}

\textbf{c) }Let $f_{1}=f\chi _{\left\{ \left\vert f\right\vert >\alpha
/2\right\} }$. Note that $\widetilde{\mathcal{M}}f\leq \widetilde{\mathcal{M}%
}f_{1}+\frac{\alpha }{2}.$ Hence by part (b) 
\begin{equation*}
\left\vert \left\{ \widetilde{\mathcal{M}}f>\alpha \right\} \right\vert \leq
\left\vert \left\{ \widetilde{\mathcal{M}}f_{1}>\alpha /2\right\}
\right\vert \leq \frac{2c}{\alpha }\int_{\left\vert f\right\vert >\alpha
/2}\left\vert f\right\vert dm.
\end{equation*}%
On the other hand,%
\begin{eqnarray*}
\int \left( \widetilde{\mathcal{M}}f\right) ^{p} &=&p\int_{0}^{\infty
}\left\vert \left\{ \widetilde{\mathcal{M}}f>\alpha \right\} \right\vert
\alpha ^{p-1}d\alpha \leq p\frac{2c}{\alpha }\int_{0}^{\infty
}\int_{\left\vert f\right\vert >\alpha /2}\left\vert f\right\vert \alpha
^{p-1}d\alpha \\
&=&2cp\int \int_{0}^{2\left\vert f\right\vert }\alpha ^{p-2}d\alpha
\left\vert f\right\vert =c\frac{p}{p-1}2^{p}\int \left\vert f\right\vert
^{p}.
\end{eqnarray*}
\end{proof}

\begin{corollary}
\label{co1}Let $f\in L_{1}$. Then%
\begin{equation*}
\lim_{\delta \rightarrow 0}A_{\delta }f\left( t,x\right) =f\left( t,x\right) 
\text{ a.e.}
\end{equation*}%
and $\left\vert f\left( t,x\right) \right\vert \leq \mathcal{M}f\left(
t,x\right) $ a.e. Moreover, for every $\alpha >0$, 
\begin{equation*}
\left\vert \left\{ \mathcal{\tilde{M}}f\left( t,x\right) >\alpha \right\}
\right\vert \leq \frac{2c}{\alpha }\int_{\left\{ \mathcal{\tilde{M}}f\left(
t,x\right) >\alpha /2\right\} }|f|dm,
\end{equation*}%
where $c$ is a constant in Theorem \ref{ft}.
\end{corollary}

\begin{proof}
Let $f\in L_{1},\varepsilon >0$. There is $g\in C_{c}\left( \mathbf{R}%
^{d+1}\right) $ so that $\left\vert f-g\right\vert _{L_{1}}\leq \varepsilon $%
. Let $\eta >0$. Since $g$ is uniformly continuous, for all $\left(
t,x\right) $%
\begin{eqnarray*}
&&\left\vert A_{\delta }g\left( t,x\right) -g\left( t,x\right) \right\vert 
\\
&\leq &\frac{1}{\left\vert Q_{\delta }\left( t,x\right) \right\vert }%
\int_{Q_{\delta }\left( t,x\right) }\left\vert g\left( s,y\right) -g\left(
t,x\right) \right\vert dsdy\leq \eta 
\end{eqnarray*}%
if $\delta \leq \delta _{0}$ for some $\delta _{0}>0$. Hence $%
\sup_{t,x}\left\vert A_{\delta }g\left( t,x\right) -g\left( t,x\right)
\right\vert \rightarrow 0$ as $\delta \rightarrow 0.$ Now for $\left(
t,x\right) \in \mathbf{R}^{d+1},$ 
\begin{eqnarray*}
&&\lim \sup_{\delta \rightarrow 0}\left\vert A_{\delta }f\left( t,x\right)
-f\left( t,x\right) \right\vert  \\
&\leq &\lim \sup_{\delta \rightarrow 0}\left\vert A_{\delta }f\left(
t,x\right) -A_{\delta }g\left( t,x\right) \right\vert +\left\vert g\left(
t,x\right) -f\left( t,x\right) \right\vert  \\
&\leq &\mathcal{M}\left( f-g\right) \left( t,x\right) +\left\vert g\left(
t,x\right) -f\left( t,x\right) \right\vert .
\end{eqnarray*}%
Hence for any $\alpha >0,$ by Theorem \ref{ft},%
\begin{eqnarray*}
&&\left\vert \left\{ \lim \sup_{\delta \rightarrow 0}\left\vert A_{\delta
}f\left( t,x\right) -f\left( t,x\right) \right\vert >\alpha \right\}
\right\vert  \\
&\leq &\left\vert \left\{ \mathcal{M}\left( f-g\right) >\alpha /2\right\}
\right\vert +\left\vert \left\{ \left\vert g-f\right\vert >\alpha /2\right\}
\right\vert  \\
&\leq &\frac{2c\varepsilon }{\alpha }+\frac{2\varepsilon }{\alpha }.
\end{eqnarray*}%
Since $\varepsilon $ and $\alpha $ are arbitrary, it follows that $\lim
\sup_{\delta \rightarrow 0}\left\vert A_{\delta }f\left( t,x\right) -f\left(
t,x\right) \right\vert =0$ a.e. Hence for almost all $\left( t,x\right) ,$%
\begin{eqnarray*}
\left\vert f\left( t,x\right) \right\vert  &=&\left\vert \lim_{\delta
\rightarrow 0}A_{\delta }f\left( t,x\right) \right\vert \leq \lim_{\delta
\rightarrow 0}\frac{1}{\left\vert Q_{\delta }\left( t,x\right) \right\vert }%
\int_{Q_{\delta }\left( t,x\right) }\left\vert f\left( t,y\right)
\right\vert dtdy \\
&\leq &\sup_{\delta >0}\frac{1}{\left\vert Q_{\delta }\left( t,x\right)
\right\vert }\int_{Q_{\delta }\left( t,x\right) }\left\vert f\left(
t,y\right) \right\vert dtdy=\mathcal{M}f\left( t,x\right) .
\end{eqnarray*}%
Finally, for $f_{1}=f\chi _{\left\{ \left\vert f\right\vert >\alpha
/2\right\} }$ we have $\mathcal{\tilde{M}}f\leq \mathcal{\tilde{M}}f_{1}+%
\frac{\alpha }{2}$, and by Theorem \ref{ft}(b),%
\begin{equation*}
\left\vert \left\{ \mathcal{\tilde{M}}f>\alpha \right\} \right\vert \leq
\left\vert \left\{ \mathcal{\tilde{M}}f_{1}>\alpha /2\right\} \right\vert
\leq \frac{2c}{\alpha }\int_{\left\vert f\right\vert >\alpha /2}\left\vert
f\right\vert dm\leq \frac{2c}{\alpha }\int_{\mathcal{\tilde{M}}f>\alpha
/2}\left\vert f\right\vert dm.
\end{equation*}
\end{proof}

\subsection{Calderon-Zygmund decomposition}

Assume \textbf{A1, A2} hold. Let $F\subseteq \mathbf{R\times R}^{d}$ be
closed and $O=F^{c}=\mathbf{R}^{d+1}\backslash F.$ For $\left( t,x\right)
\in O$, let%
\begin{equation*}
D\left( t,x\right) =\inf \left\{ \delta >0:Q_{\delta }\left( t,x\right) \cap
F\neq \emptyset \right\} .
\end{equation*}%
For each $\left( t,x\right) \in O$, $D\left( t,x\right) \in \left( 0,\infty
\right) $. Let $K_{0}$ be a constant in Lemma \ref{le1}. We fix $A>1$ so
that $l\left( 1/A\right) <1$ and $\varepsilon >0$ so that $l\left(
2K_{0}\varepsilon \right) <1,\varepsilon \leq \frac{1}{4AK_{0}^{3}}<1.$
Then, denoting $D=D(t,x),$ we have 
\begin{equation*}
\kappa \left( \varepsilon D\right) \leq l\left( 2\varepsilon \right) \kappa
\left( D/2\right) \leq \kappa \left( \frac{D}{2}\right) ,\kappa \left(
\varepsilon D\right) \leq l\left( \varepsilon \right) \kappa \left( D\right)
\leq \kappa \left( D\right) ,
\end{equation*}%
\begin{equation*}
\kappa (D)\leq l\left( 1/A\right) \kappa \left( AD\right) <\kappa \left(
AD\right) ,
\end{equation*}%
and 
\begin{equation*}
\kappa \left( \varepsilon D\right) \leq l\left( 2K_{0}\varepsilon \right)
\kappa \left( D/2K_{0}\right) \leq \kappa \left( D/2K_{0}\right) .
\end{equation*}%
Consider the covering $Q_{\varepsilon D(t,x)}\left( t,x\right) ,\left(
t,x\right) \in O$, of $O$. Let 
\begin{equation*}
Q^{k}=Q_{\varepsilon D(t_{k},x_{k})}\left( t_{k},x_{k}\right) ,k\geq 1,
\end{equation*}
be its maximal disjoint subcollection: for any $Q_{\varepsilon D\left(
t,x\right) }\left( t,x\right) $ there is $k$ so that $Q^{\varepsilon
}(t,x)\cap Q^{k}\neq \emptyset .$ Let 
\begin{equation*}
Q^{\ast k}=Q_{D(t_{k},x_{k})/2}\left( t_{k},x_{k}\right) ,Q^{\ast \ast
k}=Q_{AD\left( t_{k},x_{k}\right) }\left( t_{k},x_{k}\right) .
\end{equation*}%
Note that $Q^{k}\subseteq Q^{\ast k}\subseteq Q_{D(t_{k},x_{k})}\left(
t_{k},x_{k}\right) \subseteq O,Q^{\ast \ast k}\cap F\neq \emptyset $. We
will show that $\cup _{k}Q^{\ast k}=O$. Let $\left( t,x\right) \in O$ and $%
Q_{\varepsilon D\left( t_{k}x_{k}\right) }\left( t_{k},x_{k}\right) \cap
Q_{\varepsilon D(t,x)}\left( t,x\right) \neq \emptyset $ for some $k$. Since 
\begin{eqnarray*}
Q_{\varepsilon D\left( t_{k},x_{k}\right) }\left( t_{k},x_{k}\right)
&\subseteq &Q_{D\left( t_{k},x_{k}\right) }\left( t_{k},x_{k}\right)
\subseteq Q_{AD\left( t_{k},x_{k}\right) }\left( t_{k},x_{k}\right) , \\
Q_{\varepsilon D\left( t,x\right) }\left( t,x\right) &\subseteq
&Q_{D(t,x)/(2K_{0})},
\end{eqnarray*}%
it follows that 
\begin{equation*}
Q_{D(t,x)/(2K_{0})}\left( t,x\right) \cap Q_{AD\left( t_{k},x_{k}\right)
}\left( t_{k},x_{k}\right) \neq \emptyset .
\end{equation*}

We show by contradiction that $AD\left( t_{k},x_{k}\right) \geq D\left(
t,x\right) /2K_{0}.$ If not so, then \thinspace $AD\left( t_{k},x_{k}\right)
<D\left( t,x\right) /2K_{0}$, and, by Lemma \ref{le1}, $Q_{D(t,x)/2K_{0}}%
\left( t,x\right) $ and $Q_{AD\left( t_{k},x_{k}\right) }\left(
t_{k},x_{k}\right) $ are contained in $Q_{_{D\left( t,x\right)
/2}}(t,x)\subseteq O$: a contradiction to $Q_{AD\left( t_{k},x_{k}\right)
}\left( t_{k},x_{k}\right) \cap F\neq \emptyset .$ Therefore $AD\left(
t_{k},x_{k}\right) \geq D\left( t,x\right) /2K_{0}$ and $2AK_{0}\varepsilon
D\left( t_{k},x_{k}\right) \geq \varepsilon D\left( t,x\right) $. Now, $%
Q_{\varepsilon D\left( t_{k}x_{k}\right) }\left( t_{k},x_{k}\right)
\subseteq Q_{2AK_{0}\varepsilon D\left( t_{k}x_{k}\right) }\left(
t_{k},x_{k}\right) $ and $Q_{\varepsilon D\left( t_{k}x_{k}\right) }\left(
t_{k},x_{k}\right) \cap Q_{\varepsilon D(t,x)}\left( t,x\right) \neq
\emptyset \,$. Hence by Lemma \ref{le1}, $Q_{\varepsilon D\left( t,x\right)
} $ is contained in $Q_{2AK_{0}^{2}\varepsilon D\left( t_{k},x_{k}\right)
}\left( t_{k},x_{k}\right) $. Since $2AK_{0}^{2}\varepsilon \leq \frac{1}{%
2K_{0}}$, it follows by Lemma 1 that%
\begin{equation*}
Q_{\varepsilon D\left( t,x\right) }\left( t,x\right) \subseteq
Q_{2AK_{0}^{2}\varepsilon D\left( t_{k},x_{k}\right) }\left(
t_{k},x_{k}\right) \subseteq Q_{D\left( t_{k},x_{k}\right) /2}\left(
t_{k},x_{k}\right) =Q^{\ast k}.
\end{equation*}%
So we proved the following statement.

\begin{lemma}
\label{wh}Assume \textbf{A1, A2} hold. Given a closed nonempty $F$, there
are sequences $Q^{k}$, $Q^{\ast k}$ and $Q^{\ast \ast k}$ in $\mathbb{Q\,\ }$%
having the same center but with radius expanded by the same factor $%
c_{1}^{\ast \ast }>c_{1}^{\ast }>c_{1}\,\ $so that $Q^{k}\subseteq Q^{\ast
k}\subseteq Q^{\ast \ast k}$ (all of them are of the form $Q_{bD\left(
t_{k},x_{k}\right) }\left( t_{k},x_{k}\right) $ with $b=c_{1},c_{1}^{\ast
},c_{1}^{\ast \ast }$ correspondingly) and

(a) the sets $Q^{k}$ are disjoint.

(b) $\cup _{k}Q^{\ast k}=O=F^{c}.$

(c) $Q^{\ast \ast k}\cap F\neq \emptyset $ for each $k$.
\end{lemma}

\begin{remark}
\label{re3}Assume \textbf{A1, A2} hold and $Q^{k}\subseteq Q^{\ast
k}\subseteq Q^{\ast \ast k}$ be the sequences in $\mathbb{Q}$ from Lemma \ref%
{wh}. It is easy to find a sequence of disjoint measurable sets $C^{k}$ so
that $Q^{k}\subseteq C^{k}\subseteq Q^{\ast k}$ and $\cup _{k}C^{k}=O$. For
example (see Remark, p. 15, in \cite{stein1}),%
\begin{equation*}
C^{k}=Q^{\ast k}\cap \left( \cup _{j<k}C^{j}\right) ^{c}\cap \left( \cup
_{j>k}Q^{j}\right) ^{c}.
\end{equation*}
\end{remark}

Now we derive Calderon-Zygmund decomposition for $\mathbb{Q}.\,\ $

\begin{theorem}
\label{whit}Assume \textbf{A1, A2} hold. Let $f\in L_{1}\left( \mathbf{%
R\times R}^{d}\right) $, $\alpha >0$ and $O_{\alpha }=\left\{ \widetilde{%
\mathcal{M}}f>\alpha \right\} .$ Consider the sets $Q^{k}\subseteq
C^{k}\subseteq Q^{\ast k}\subseteq O$ of Lemma \ref{wh} and Remark \ref{re3}
associated to $O_{\alpha }.$

There is a decomposition $f=g+b$ with%
\begin{equation}
g\left( x\right) =\left\{ 
\begin{array}{cc}
f(x) & \text{if }x\notin O_{\alpha }, \\ 
\frac{1}{\left\vert C^{k}\right\vert }\int_{C^{k}}f & \text{ if }x\in
C^{k},k\geq 1,%
\end{array}%
\right.  \label{2}
\end{equation}%
and with $b=\sum_{k}b_{k}$, where 
\begin{equation}
b_{k}=\chi _{C^{k}}\left[ f\left( x\right) -\frac{1}{\left\vert
C^{k}\right\vert }\int_{C^{k}}f\text{ }\right] ,k\geq 1,  \label{3}
\end{equation}%
(note $C^{k}$ are disjoint, $\cup _{k}C^{k}=O_{\alpha }$). Also,

(i) $\left\vert g\left( x\right) \right\vert \leq c\alpha $ for a.e. $x.$

(ii) support($b_{k})\subseteq Q^{\ast k},$%
\begin{equation*}
\int b_{k}=0\text{ and }\int \left\vert b_{k}\right\vert \leq c\alpha
\left\vert Q^{\ast k}\right\vert .
\end{equation*}

(iii) $\sum_{k}\left\vert Q^{\ast k}\right\vert \leq \frac{c}{\alpha }\int
\left\vert f\right\vert .$
\end{theorem}

\begin{proof}
The set $O_{\alpha }=\left\{ \widetilde{\mathcal{M}}f>\alpha \right\} $ is
open. We cab apply Lemma \ref{wh} and Remark \ref{re3} to it and consider
the sets $Q^{k}\subseteq C^{k}\subseteq Q^{\ast k}\subseteq E_{\alpha }$
with $C^{k}$ disjoint and $\cup _{k}C^{k}=E_{\alpha }.$

Define $g$ by (\ref{3}). Hence $f=g+\sum_{k}b_{k}$ with $b_{k}$ given by (%
\ref{3}). Obviously%
\begin{equation*}
\sum_{k}\left\vert Q^{k}\right\vert \leq \left\vert E_{\alpha }\right\vert .
\end{equation*}

(i) By Corollary \ref{co1}, $\left\vert f\left( x\right) \right\vert \leq
\alpha $ a.e. on $O_{\alpha }^{c}=\left\{ \widetilde{\mathcal{M}}f\left(
t,x\right) \leq \alpha \right\} $. Hence$:$ so $\left\vert g\left( x\right)
\right\vert \leq \alpha $ a.e. on $E_{\alpha }^{c}$. On the other hand, if $%
Q^{\ast \ast k}\in \mathbb{Q}$ is the sequence of Lemma \ref{wh}, then 
\begin{equation*}
\frac{1}{\left\vert Q^{\ast \ast k}\right\vert }\int_{Q^{\ast \ast
k}}\left\vert f\right\vert \leq \alpha
\end{equation*}%
because $Q^{\ast \ast k}\cap O_{\alpha }^{c}\neq \emptyset $ and $\widetilde{%
\mathcal{M}}f\left( t,x\right) \leq \alpha $ on $O_{\alpha }^{c}$ (the
definition of $\widetilde{\mathcal{M}}$ implies it). Since $\left\vert
Q^{k}\right\vert \leq \left\vert C^{k}\right\vert \leq \left\vert Q^{\ast
k}\right\vert \leq \left\vert Q^{\ast \ast k}\right\vert \leq l\left( \frac{%
c_{1}^{\ast \ast }}{c_{1}}\right) \left\vert Q^{k}\right\vert $ and $%
C^{k}\subseteq Q^{\ast \ast k}$, it follows that%
\begin{equation*}
\left\vert g\right\vert \leq \bar{c}\alpha .
\end{equation*}

(ii) \ Only inequality is not trivial:%
\begin{equation*}
\int \left\vert b_{k}\right\vert \leq 2\int_{C^{k}}\left\vert f\right\vert
\leq 2\left\vert Q^{\ast \ast k}\right\vert \frac{1}{\left\vert Q^{\ast \ast
k}\right\vert }\int_{Q^{\ast \ast k}}\left\vert f\right\vert \leq c\alpha
\left\vert Q^{\ast k}\right\vert .
\end{equation*}

(iii) We have 
\begin{equation*}
\left\vert O_{\alpha }\right\vert =\left\vert \left\{ \widetilde{\mathcal{M}}%
f\left( t,x\right) >\alpha \right\} \right\vert \geq \sum_{k}\left\vert
Q^{k}\right\vert \geq \tilde{c}\sum_{k}\left\vert Q^{\ast k}\right\vert
\end{equation*}%
and the inequality follows by Theorem \ref{ft}.
\end{proof}

\subsection{$L_{p}$-estimates}

Let now%
\begin{equation*}
\left( Tf\right) \left( t,x\right) =\int_{\mathbf{R}^{d+1}}K\left(
t,x,s,y\right) f\left( s,y\right) dsdy,(t,x)\in \mathbf{R}^{d+1}.
\end{equation*}%
We assume that $T$ is defined and bounded on $L_{q}$:%
\begin{equation}
\left\vert Tf\right\vert _{L_{q}}\leq C\left\vert f\right\vert _{L_{q}},f\in
L_{q}.  \label{f1}
\end{equation}%
In addition, we assume that there is a constant $A>0$ so that (denoting $%
\left( s_{k},y_{k}\right) $ the center of $Q_{k}^{\ast }$, 
\begin{equation}
\sup_{(s,y)\in Q_{k}^{\ast }}\int_{(Q_{k}^{\ast \ast })^{c}}|K\left(
t,x,s,y\right) -K\left( t,x,s_{k},y_{k}\right) |dtdx\leq A.  \label{f22}
\end{equation}

The assumption (\ref{f22}) holds if there are constants $c>1,A>0$ so that
for any $Q_{\delta }\in \mathbb{Q}$, 
\begin{equation}
\int_{\mathbf{R}^{d+1}\backslash Q_{c\delta }\left( s,y\right) }\left\vert
K\left( t,x,\bar{s},\bar{y}\right) -K\left( t,x,s,y\right) \right\vert
dxdt\leq A\text{ }\forall (\bar{s},\bar{y})\in Q_{\delta }\left( s,y\right)
\label{f23}
\end{equation}

\begin{theorem}
\label{t3}Let \textbf{A1, A2}, (\ref{f1}) and (\ref{f23}) hold. Then $T$ is
bounded in $L_{p}$-norm on $L_{p}\cap L_{q}$ if $1<p<q$. More precisely,%
\begin{equation*}
\left\vert T\left( f\right) \right\vert _{L_{p}}\leq A_{p}\left\vert
f\right\vert _{L_{p}},f\in L_{p}\cap L_{q},1<p<q,
\end{equation*}%
where $A_{p}$ depends only on the constant $A$ and $p$.
\end{theorem}

\begin{proof}
As \cite{stein1} says, it is enough to prove that%
\begin{equation*}
m\left( \left\vert Tf\right\vert >\alpha \right) \leq \frac{A^{\prime }}{%
\alpha }\int \left\vert f\right\vert dx,f\in L_{1}\cap L_{q},\alpha >0,
\end{equation*}%
where $A^{\prime }$ depends on $A$.

For a large constant $c^{\prime }$ (to be determined) we estimate $m\left(
\left\vert Tf\right\vert >c^{\prime }\alpha \right) $. For a fixed $\alpha
>0 $ we consider the decomposition $f=g+b$ in Theorem \ref{whit}. First note
that%
\begin{equation*}
m\left( \cup _{n}Q_{n}^{\ast \ast }\right) \leq \sum_{n}m\left( Q_{n}^{\ast
\ast }\right) \leq c\sum_{n}m\left( Q_{n}^{\ast }\right) \leq \frac{c}{%
\alpha }\int \left\vert f\right\vert .
\end{equation*}

It is enough to show that%
\begin{equation*}
\left\vert \left\{ \left\vert Tg\right\vert >\left( c^{\prime }/2\right)
\alpha \right\} \right\vert +\left\vert \left\{ \left\vert Tb\right\vert
>\left( c^{\prime }/2\right) \alpha \right\} \right\vert \leq A^{\prime
}/\alpha \int \left\vert f\right\vert dx.
\end{equation*}

First $g\in L_{q}$. Indeed,%
\begin{equation*}
\int \left\vert g\right\vert ^{q}dx=\int_{\left( \cup _{k}Q_{k}^{\ast
}\right) ^{c}}|g|^{q}+\int_{\cup _{k}C_{k}}\left\vert g\right\vert ^{q}\leq
c\alpha ^{q-1}\int |f|
\end{equation*}%
because%
\begin{eqnarray*}
\int_{\left( \cup _{k}Q_{k}^{\ast }\right) ^{c}}\left\vert g\right\vert ^{q}
&\leq &c\alpha ^{q-1}\int_{\left( \cup _{k}Q_{k}^{\ast }\right) ^{c}}|f|, \\
\int_{\cup _{k}Q_{k}^{\ast }}\left\vert g\right\vert ^{q} &\leq &c\alpha
^{q}\sum_{k}\left\vert Q_{k}^{\ast }\right\vert \leq c\alpha ^{q-1}\int
\left\vert f\right\vert .
\end{eqnarray*}%
By Chebyshev inequality,%
\begin{eqnarray*}
&&\left\vert \left\{ \left\vert Tg\right\vert >\left( c^{\prime }/2\right)
\alpha \right\} \right\vert \leq \left( \frac{c^{\prime }\alpha }{2}\right)
^{-q}\left\vert Tg\right\vert _{L_{q}}^{q}\leq \left( \frac{c^{\prime
}\alpha }{2}\right) ^{-q}A^{q}\left\vert g\right\vert _{L_{q}}^{q} \\
&\leq &c\left( \frac{c^{\prime }\alpha }{2}\right) ^{-q}A^{q}\alpha
^{q-1}\int \left\vert f\right\vert \leq \frac{A^{\prime }}{\alpha }%
\left\vert f\right\vert _{L_{1}}.
\end{eqnarray*}

Now,%
\begin{equation*}
\int_{\left( \cup _{k}Q_{k}^{\ast \ast }\right) ^{c}}|Tb|\leq
\sum_{k}\int_{\left( Q_{k}^{\ast \ast }\right) ^{c}}\left\vert
Tb_{k}\right\vert
\end{equation*}%
Since for $x\notin Q_{k}^{\ast },$ denoting by $(s_{k},y_{k})$ the center of 
$Q_{k}^{\ast }$ (and $Q_{k}^{\ast \ast }$), we have. denoting $%
f_{k}=1/\left\vert C^{k}\right\vert \int_{C_{k}}f$,%
\begin{eqnarray*}
Tb_{k} &=&\int_{C_{k}}K\left( t,x,s,y\right) [f\left( s,y\right) -f_{k}]dsdy
\\
&=&\int_{C_{k}}[K\left( t,x,s,y\right) -K\left( t,x,s_{k},y_{k}\right)
][f\left( s,y\right) -f_{k}]dsdy
\end{eqnarray*}%
and%
\begin{eqnarray*}
\int_{(Q_{k}^{\ast \ast })^{c}}\left\vert Tb_{k}\right\vert dx &\leq &\int
~\left\vert b_{k}\right\vert dsdy\sup_{(s,y)\in Q_{k}^{\ast
}}\int_{(Q_{k}^{\ast \ast })^{c}}|K\left( t,x,s,y\right) -K\left(
t,x,s_{k},y_{k}\right) |dtdx \\
&\leq &cA\alpha \left\vert Q_{k}^{\ast }\right\vert .
\end{eqnarray*}%
Hence%
\begin{equation*}
\int_{\left( \cup _{k}Q_{k}^{\ast \ast }\right) ^{c}}|Tb|\leq cA\alpha
\sum_{k}\left\vert Q_{k}^{\ast }\right\vert \leq cA\int \left\vert
f\right\vert .
\end{equation*}
\end{proof}

\end{document}